\documentclass{amsart}
\usepackage{amscd,amsfonts,amssymb,amsmath}
\usepackage{graphicx,pst-all}
\usepackage{caption}
\usepackage{float}
\usepackage{color}

\usepackage[active]{srcltx}
\usepackage[all]{xy}
\usepackage{subfig}
\usepackage{hyperref}
\usepackage{mathrsfs}

\usepackage{amsmath}
\usepackage{amssymb,latexsym}
\usepackage{mathrsfs}
\usepackage{graphics}
\usepackage{latexsym}
\usepackage{psfrag}
\usepackage{import}
\usepackage{verbatim}
\usepackage{graphicx}
\usepackage{pifont,marvosym}

\theoremstyle{plain}
\newtheorem{lemma}{Lemma}[section]
\newtheorem{theorem}[lemma]{Theorem}
\newtheorem{proposition}[lemma]{Proposition}
\newtheorem{corollary}[lemma]{Corollary}

\theoremstyle{definition}

\newtheorem{definition}[lemma]{Definition}
\newtheorem{remark}[lemma]{Remark}

\numberwithin{equation}{section}

\newcommand{\R}{\mathbb{R}}
\newcommand{\HH}{\mathbb{H}}
\newcommand{\bS}{\mathbb{S}}
\newcommand{\N}{\mathbb{N}}

\newcommand{\spt}{\text{\rm spt}}

\newcommand{\C}{\mathbb{C}}

\newcommand{\id}{\mathrm{Id}}

\newcommand{\ve}{\varepsilon}

\newcommand{\MCP}{\mathsf{MCP}}

\newcommand{\Geo}{{\rm Geo}}
\newcommand{\Opt}{{\rm OptGeo}}
\newcommand{\Span}{{\rm Span}}
\newcommand{\weak}{\rightharpoonup}

\newcommand{\M}{\mathcal{M}}
\newcommand{\ee}{{\rm e}}
\newcommand{\Ip}{\mathsf{Ip}}
\newcommand{\Exp}{\text{\rm Exp}}
\renewcommand{\H}{\mathcal{H}}

\renewcommand{\L}{\mathcal{L}}

\newcommand{\sfd}{{\mathsf d}}
\newcommand{\mm}{{\mathsf m}}

\begin{document}
\title{Measure rigidity of Ricci curvature lower bounds}

\author{Fabio Cavalletti and Andrea Mondino}

\address{CRM De Giorgi - Scuola Normale Superiore, Piazza dei Cavalieri 3, I-56126 Pisa (Italy)}
\email{fabio.cavalletti@sns.it}

\address{ETH,  R\"amistrasse 101,  8092 Zurich (Switzerland)}
\email{andrea.mondino@math.ethz.ch}


\bibliographystyle{plain}

\begin{abstract}
The measure contraction property, $\MCP$ for short, is a weak Ricci curvature lower bound conditions  for metric measure spaces. 
The goal of this paper is to understand which structural properties such assumption  (or even  weaker modifications) implies on the measure, 
on its support and on the geodesics of the space.  

We start our investigation from the euclidean case by proving that if a positive Radon measure $\mm$ over $\R^{d}$ is such that $(\R^{d},|\cdot |, \mm)$ verifies 
a weaker variant of $\MCP$, then  its support $\spt(\mm)$ must be convex and $\mm$ has to be absolutely continuous with respect to 
the relevant Hausdorff measure of $\spt(\mm)$. This result is then used  as a starting point to investigate the rigidity of $\MCP$ in the metric framework. 

We introduce the new notion of \emph{reference measure} for a metric space and prove that if $(X,\sfd,\mm)$ is essentially non-branching and 
verifies $\MCP$, and $\mu$ is an essentially non-branching $\MCP$ reference measure for $(\spt(\mm), \sfd)$, 
then $\mm$ is absolutely continuous with respect to $\mu$, on the set of points where an inversion plan exists.
As a consequence, an essentially non-branching $\MCP$ reference measure enjoys a weak type of uniqueness, up to densities.  
We also prove a stability property for reference measures under measured Gromov-Hausdorff convergence, provided an additional uniform bound holds.

In the final part we present concrete examples of metric spaces with reference measures, both in smooth and non-smooth setting.  
The main example will be the Hausdorff measure over an Alexandrov space. 
Then we prove that the following are reference measures over smooth spaces: 
the volume measure of a Riemannian manifold, the Hausdorff measure of an Alexandrov space with bounded
curvature, and the Haar measure of the subRiemannian Heisenberg group. 
\end{abstract}

\maketitle

%%%%%%%%%%%%%%%%%%%%%%%%%%%%%%%%%%%%%%%%%%%%%%%%%%%%%%%%%%%%%%%%%%%%%%%%
%%%%%%%%%%%%%%%%%%%%%%%%%%%%%%%%%%%%%%%%%%%%%%%%%%%%%%%%%%%%%%%%%%%%%%%%
%%%%%%%%%%%%%%%%%%%%%%%%%%%%%%%%%%%%%%%%%%%%%%%%%%%%%%%%%%%%%%%%%%%%%%%%
%%%%%%%%%%%%%%%%%%%%%%%%%%%%%%%%%%%%%%%%%%%%%%%%%%%%%%%%%%%%%%%%%%%%%%%%
%--------------------------------------------INTRODUCTION------------------------------------------------------------------------------------------

\section{Introduction}\label{S:intro}
The notion of curvature for a smooth space, i.e. a Riemannian $n$-dimensional manifold $(M,g)$, is one of most basic geometrical concepts and goes back to the work of Gauss and Riemann; the idea being to consider suitable combinations of second derivatives of the metric $g$. If the space $M$ is non-smooth this approach has few chances to be carried out, so one has to understand what are the fundamental geometric consequences of the curvature and then set these as defining properties.  
\\This was the approach by Alexandrov who, in last century,  defined what it means  for a metric space $(X,\sfd)$ to satisfy an upper or lower bound on the sectional curvature. Such  non-smooth spaces are now called Alexandrov spaces and there is a huge literature about their properties  (we do not even attempt to give a selection of papers; we just recall that Alexandrov, Gromov and  Perelman among others gave major contributions to this theory and we refer to the textbook \cite{BBI} for more  references).
\\

While, as just recalled, the  notion of lower  bound on the sectional curvature makes perfect sense for a metric space $(X,\sfd)$, i.e. a set endowed with a distance function, for defining what it means for a non-smooth space to satisfy a lower Ricci curvature bound one also has to fix a measure $\mm$  on  $(X,\sfd)$, 
thus getting a so called metric measure space $(X,\sfd, \mm)$, m.m.s. for short. In the last ten years there has been a flourishing  of literature about different notions and properties of lower Ricci curvature bounds for m.m.s.: on one hand Bakry-\'Emery-Ledoux  \cite{BakryEmery_diffusions, BakryLedoux} developed the so called Gamma-calculus based on the fact that, roughly speaking, the Bochner inequality characterizes  lower Ricci bounds; 
on the other hand there is a parallel theory via optimal transport based on the idea that one can detect Ricci curvature by examining 
how the volume changes when optimally transported along geodesics.  A third way to study the non-smooth spaces with lower Ricci bounds, having its origins in the work of Gromov and  mainly due to Cheeger-Colding \cite{CC1, CC2, cc:IIIJDE2000} (see also the more recent developments by Colding-Naber \cite{CN}),  is to concentrate the attention on those m.m.s. arising as limits  (in the measured Gromov-Hausdorff sense) of smooth Riemannian manifolds satisfying Ricci curvature lower bounds; such a point of view is very powerful if one is interested in the structure of these limit spaces. \\
\\In this paper we will focus on the second approach via optimal transportation, but let us mention that there is a precise correspondence between the first two (see  \cite{AGS} for the infinite dimensional case and \cite{EKS, AMS} for the finite dimensional one) and there are some recent results on the local structure of such spaces even in the abstract framework (see \cite{GigliSplit,GMR, MN}).  
\\

The weakest Ricci curvature condition for a m.m.s. is the so called measure contraction property, $\MCP$ for short, which keeps track of the distortion of the volume of a set  when it is transported to a Dirac delta. The quantitative formalization of having Ricci bounded from below by $K\in \R$ and dimension bounded from above by $N>1$ via this approach, the so called $\MCP(K,N)$ condition, is due to  Ohta \cite{Ohta07} and has its roots in earlier works by Grigor'yan \cite{Gri92},  Sturm \cite{Sturm98} and Kuwae-Shioya \cite{KS01}.  Such condition is compatible with the smooth counterpart (i.e. 
a smooth Riemannian manifold of dimension $N$ has Ricci curvature bounded below by $K$ if and only if it satisfies $\MCP(K,N)$) and is stable under pointed  measured Gromov-Hausdorff convergence,  $pmGH$ for short, so that the limit spaces of Cheeger-Colding are included in this  theory. Also finite dimensional Alexandrov spaces with lower curvature bounds endowed with the corresponding Hausdorff measure satisfy the measure contraction property \cite{Ohta07} (see also \cite{KS10}  for a subsequent independent proof) as well as the Heisenberg group endowed with the Haar measure and the Carnot-Carath\'eodory distance \cite{Juillet}. 

Stronger notions of ``Ricci bounded  from below by $K\in \R$ and dimension bounded above by $N\in (1, \infty]$''  for m.m.s.  are the so called ${\sf CD}(K,N)$-conditions introduced independently  by Lott-Villani \cite{LV} and by Sturm \cite{Sturm06I, Sturm06II},  and  the even stronger notions of ${\sf RCD}(K,\infty)$  and ${\sf RCD}^*(K,N)$-spaces (the first were introduced in \cite{AGS11b} and further investigated in \cite{AGMR, AGS}, for the  second ones see  \cite{AMS, EKS, GigliSplit, GMR, MN, RS}). All these classes include the smooth Riemannian manifolds satisfying the corresponding curvature-dimension bounds and  their $pmGH$-limits, as well as finite dimensional Alexandrov spaces with curvature bounded from below. On the other hand  let us recall  that the Heisenberg group does not satisfy any ${\sf CD}(K,N)$ condition  but it does verify $\MCP$, see \cite{Juillet}. So $\MCP$ is a strictly weaker notion of curvature than ${\sf CD}$.
\\

The goal of the present paper is to investigate the structural properties (in particular on the measure, on its support, and on the geodesics)  forced by $\MCP$.  
\\  In order to introduce the problem let us start by analyzing the behavior of the euclidean space $\R^d$, endowed with the   $d$-dimensional Lebesgue measure  $\mathcal{L}^{d}$ and the  Euclidean distance $|\cdot|$.    It is almost trivial to check that the triple $(\R^{d}, |\cdot |, \mathcal{L}^{d} )$ verifies the measure contraction property with curvature $K=0$  and dimension $N = d$,  i.e.  $\MCP(0,d)$. Moreover it is well known that,  at the price of changing the lower bound on the curvature and the upper bound on the dimension,   in the triple $(\R^{d}, |\cdot |, \mathcal{L}^{d} )$ it is possible to replace the Lebesgue measure with some equivalent measure and still obtain $\MCP$, provided the density verifies some concavity estimates. 
\\A first goal of this work is to investigate, and give affermative answer to, the reverse question: 
\begin{itemize}
\item Let $\mm$ be a positive Radon measure on $\R^d$  such that the triple $(\R^{d}, |\cdot|, \mm)$ verifies  $\MCP$ (or even a weaker condition). Can we deduce that $\mm$ is absolutely continuous with respect to $\L^{k}$, for some $k\leq d$? 
\item Moreover, does the support of $\mm$ have some nice geometric properties?
\end{itemize}
The next theorem contains an affirmative answer to these natural questions; before stating it  let us remark that the   \emph{non-degeneracy condition} for a positive Radon  measure $\mm$ is given in Definition \ref{D:nondegenerate} and  it is ensured by the measure contraction property (it is indeed much weaker since no quantitative or uniform lower bound on the transported measure is assumed).

\begin{theorem}\label{T:1}
Let $\mm$ be a positive Radon  measure over $\R^{d}$ and denote with $\Omega$ its support. 
If the  metric measure space $(\Omega,|\cdot|, \mm)$ verifies the \emph{non-degeneracy condition} then
there exists a natural number $k \leq d$ so that: 
\begin{itemize}
\item  $\Omega$ is convex and contained in a $k$-dimensional affine subspace of $\R^{d}$, say $V^{k}$;
\item the measure $\mm$ is absolutely continuous with respect to $\mathcal{L}^{k}\llcorner {V^{k}}$.
\end{itemize}
\end{theorem}

The same question can be reformulated in the metric framework only once the choice of a favorite measure is done. 
More precisely, given a m.m.s. $(X,\sfd,\mu)$  satisfying $\MCP$ and maybe some other structural assumption,  one can  ask:
\begin{itemize}
\item  Is the support of $\mu$ convex? 
\item  If $\mm$ is a positive Radon measure on $X$  
also satisfying $\MCP$, can we deduce that $\mm$ is absolutely continuous with respect to $\mu$? 
\end{itemize}
This kind of properties were proved by Cheeger-Colding \cite{cc:IIIJDE2000} in the framework of $pmGH$-limits of Riemannian manifolds satisfying lower Ricci curvature bounds: more precisely they showed that in the limit space the ``favorite measure'' is the Hausdorff measure of the relevant dimension, and any other possible limit measure has to be absolutely continuous with respect to it. Therefore this paper can be seen as an intrinsic-non smooth analogue of the Cheeger-Colding result.
  
Regarding the first question, in Proposition \ref{Prop:d-convex}, we will prove the affirmative answer: 
the support of any Radon measure satisfying $	\MCP$  has to be weakly convex 
(more precisely the support of any measure satisfying the strong qualitative $\MCP$ condition, defined in \eqref{eq:StrongQualMCP}, 
must be weakly convex; i.e. every couple of points of $\spt(\mm)$ is joined by a geodesic entirely contained in  $\spt(\mm)$).  
In order to solve the second problem, we will give a precise meaning of what is for us a ``favorite measure'': 
this is what we call \emph{reference measure} (see Definition \ref{D:referencemeasure}) for a complete and separable metric space $(X,\sfd)$. 
The crucial property being to behave nicely under geodesic inversion with respect to almost every point of the space.

\begin{definition}[Reference measure]
A positive Radon measure  $\mu \in \M^+(X)$ is a \emph{reference measure for $(X,\sfd)$} provided it is non-zero, and
for $\mu$-a.e. $z \in X$ there exists $\pi^{z} \in \mathcal{M}^+(X \times X)$ so that 
\[
(P_{1})_{\sharp}\,\pi^{z} = \mu, \qquad \pi^{z}  (X\times X \setminus H(z))=0,  \qquad (P_{2})_{\sharp}\, \pi^{z} \ll \mu, 
\]
where $P_{i} : X\times X \to X$ is the projection on the $i$-th component, for $i =1,2$ and 
\[
H(z): = \{(x,y)\in X\times X \,:\, \sfd(x,y)=\sfd(x,z)+\sfd(z,y)\}.
\]
The measure $\pi^{z}$ will be called  \emph{inversion plan} and those points where an inversion plan exists will be called \emph{inversion points} and denoted by $\mathsf{Ip}(\mu)$.
\end{definition}

Let us stress that this concept is completely new, to best of our knowledge. In Remark \ref{rem:Inversion} we will also 
point out that existence of an inversion plan in a point $z$ is closely related to the regularity of the ambient space $(X,\sfd)$ at $z$. 
In particular if a conical singularity happens in $z$ then no inversion plan at $z$ exists.

Another ingredient in the next theorem will be the essentially non-branching condition (recalled in Section \ref{Ss:essentially}): it 
is an important structural assumption on a m.m.s. and it is fulfilled by a large class of geometrically relevant examples, for instance Riemannian manifolds, 
Alexandrov spaces with lower curvature bounds, $pmGH$-limits of Riemannian manifolds with lower Ricci curvature bounds, 
metric measure spaces verifying ${\sf RCD}(K,\infty)$ or ${\sf RCD}^{*}(K,N)$,
the Heisenberg Group endowed with the Carnot-Carath\'eodory metric, etc.. 
Moreover, since in the proof of the next result we do not need the quantitative controls assumed in $\MCP(K,N)$, 
we just assume a weaker  qualitative $\MCP$ condition (for the precise notion see Definition \ref{D:qualitative}).

\begin{theorem}\label{T:2}
Let $(X,\sfd,\mm)$ be an essentially non-branching m.m.s. that verifies the qualitative $\MCP$ condition  \eqref{eq:QualMCP}.
Assume the existence of a reference measure $\mu$ for $(\Omega,\sfd)$, where $\Omega= \spt (\mm)$,
so that $(\Omega, \sfd, \mu)$ verifies the qualitative $\MCP$  condition  \eqref{eq:QualMCP} and it is essentially non-branching. \\ 
\noindent
If $\mm (\Omega \setminus \mathsf{Ip}(\mu)) = 0$, then
$$
\mm \ll \mu.
$$
In particular, if also $\mm$ is a reference measure and $\spt(\mm)=\spt(\mu)$ with $\mu(\Omega \setminus \mathsf{Ip}(\mm)) = 0$, 
then $\mu$ and $\mm$ are equivalent measures, i.e.
$$
\mu \ll \mm \quad \textrm{ and } \quad \mm \ll \mu. 
$$
In other words, once the support is fixed,  a reference measure satisfying $\MCP$ is uniquely determined up to densities.
\end{theorem}

Once observed that metric measure spaces verifying ${\sf RCD}$ are all essentially non-branching,
Theorem \ref{T:2} can be used straightforwardly to obtain the next uniqueness result. 

\begin{theorem}\label{T:RCD}
Let $(X,\sfd,\mu)$ be a m.m.s. that verifies ${\sf RCD}^{*}(K,N)$, with $N < \infty$ and $X = \spt (\mu)$. Assume that $\mu$ is a reference measure for $(X,\sfd)$.
If $(X,\sfd,\mm)$ is a m.m.s. verifying ${\sf RCD}^{*}(K',N')$, possibly for different $K'$ and $N'$, with $X = \spt (\mm)$ and  $\mm (X \setminus \mathsf{Ip}(\mu)) = 0$, then 
\[
\mm \ll \mu.
\]
In particular, if also $\mm$ is a reference measure with $\mu(X \setminus \mathsf{Ip}(\mm)) = 0$,
then $\mu$ and $\mm$ are equivalent measures, i.e. $\mu \ll \mm$ and $\mm \ll \mu$.
\end{theorem}

The second part of Theorem \ref{T:RCD} therefore says that once the support $X$ and a metric $\sfd$ are fixed, there is a unique reference measure, 
up to densities, so that the triple $(X,\sfd,\mu)$ verifies ${\sf RCD}^{*}$.

Both Theorem \ref{T:2} and Theorem \ref{T:RCD} can be restated dropping the assumption of $\mm (\Omega \setminus \mathsf{Ip}(\mu)) = 0$ (and the analogous one 
in the uniqueness part) at the price of assuming regularity properties on the geometry of $\mathsf{Ip}(\mu)$ (and on $\mathsf{Ip}(\mm)$):  one way (see Theorem \ref{T:metric2}) is to assume that $\mathsf{Ip}(\mu)$ is $\mu$-connected. For the precise meaning of this assumption 
see Definition \ref{D:muconnected}, here we only mention that it is inspired by, and resemble, the convexity property of regular points in Alexandrov geometry.
Another way is to assume that  $\mathsf{Ip}(\mu)$ forms an open set in $X$, see Theorem \ref{T:metric3}. 

\medskip

Section \ref{S:stability} is devoted to the investigation of general properties of reference measures:  we will establish the local-to-global property of reference measures under the non-branching assumption,  we will show that  the multiplication by a density does not affect the property of being a reference measure and  finally we will address  the problem of  stability of reference measures  with respect to  measured Gromov-Hausdorff convergence: 
it will be proved that if the reference measures are probabilities and if a uniform bound  on the second marginal of the inversion plan holds, 
the notion of reference measure is stable. See Theorem \ref{T:stability} for the precise statement and Theorem  \ref{T:unique} for its 
application to uniqueness of limit reference measures. This last result should be compared with the aforementioned analogous statement for Ricci limits spaces obtained by Cheeger-Colding \cite{cc:IIIJDE2000}.

% if we consider a sequence of proper metric measure spaces $(X_{k},\sfd_{k},\mu_{k})$, 
%with $\mu_{k}$ reference measure for $(X_{k},\sfd_{k})$,  converging in the pointed 
%measured Gromov-Hausdorff sense to a limit compact metric measure space $(X,\sfd,\mu)$, 
%then $\mu$ is a reference measure for $(X,\sfd)$, see Theorem \ref{T:stability}. 
%The stability together with Theorem \ref{T:2} permits to obtain the next uniqueness result of the weak limit of reference measures.
%
%
%\begin{theorem}\label{T:3}
%Let $(X_{k}, \sfd_{k}, \mu_{k}, \bar{x}_k)$ be a sequence of essentially non-branching, proper,  pointed  m.m.s. that uniformly  verify the qualitative $\MCP$ condition. 
%Assume also that, for every $k \in \N$,  $\mu_{k}$ is  a positive Radon reference measure for  $(X_{k}, \sfd_{k})$. 
%
%Let $(X, \sfd, \mu, \bar x)$ and $(X,\sfd, \xi, \bar y)$ be two limit points in the pointed measured Gromov Hausdorff convergence of 
%the two sequences $(X_{k}, \sfd_{k}, \mu_{k}, \bar x_{k})$ and $(X_{k}, \sfd_{k}, \mu_{k}, \bar y_{k})$, respectively.
%If $(X,\sfd, \mu)$ and $(X,\sfd,\xi)$ are both essentially non-branching and
%%
%$$
%\spt(\xi)  = \spt(\mu),
%$$
%%
%then $\mu \sim \xi$, i.e. $\mu \ll \xi$ and $\xi \ll \mu$.
%\end{theorem}

\hyphenation{pro-ba-bi-li-ty se-cond}

%This alternative version will be then applied to Alexandrov spaces with curvature bounded from below. 

\medskip

The final part of the paper, that is formed by Section \ref{S:InvAlex} and Section \ref{S:example}, is devoted to discuss examples 
and applications of the introduced techniques. 

The main result will be the existence of an inversion plan at almost every regular point of an Alexandrov space with curvature bounded from below, where the 
almost everywhere and the inversion plan are referred to the Hausdorff measure of the right dimension, see Theorem \ref{T:Alexandrov-general}.
This shows that the Hausdorff measure is a reference measure and permits to apply all the general theorems stated so far to Alexandrov spaces. 

In the last section we then present smooth examples. The term smooth here is used to emphasize that in those spaces every point admits an inversion plan. 
The smooth examples will be: the standard volume measure on a smooth Riemannian manifold, 
the Hausdorff measure on an Alexandrov space with curvature bounded from above and below (see Corollary \ref{C:Alex}) and the Heisenberg group 
endowed with the Carnot-Carath\'eodory distance and the Haar measure (see Corollary \ref{cor:Heisenberg}).

As a final comment we want to stress that, while the proof of the existence of an inversion plan in a smooth context relies mainly on the 
bi-Lipschitz regularity of the exponential map, such ingredient on a general Alexandrov space fails to hold.  
Hence a finer analysis is needed. 
Combining the Lipschitz regularity of the exponential map together with the fact that it maps sets of positive measure to sets of positive measure,
using also Disintegration Theorem, one can construct a ``local'' inversion plan.  
Suitably iterating this construction one can obtain a ``global'' inversion plan and prove Theorem \ref{T:Alexandrov-general}.
The existence of an inversion plan can be therefore  understood as a measure-theoretic reformulation of a bi-Lipschitz control 
on the behavior of geodesic hinges and, as for a general Alexandrov space not so much can be proved on the regularity of the exponential map, 
Theorem \ref{T:Alexandrov-general} seems to be a novelty also in this direction.

We end the introduction with some comments regarding possible future applications of the techniques here introduced.  

\begin{itemize}
\item In the class of those metric measure spaces admitting a bi-Lipschitz map with a subset of a Euclidean space, 
one of the relevant questions to ask is whether the image of such a metric space has positive Lebesgue measure or not. 
To this aim, let $(X,\sfd,\mu)$  satisfy $\MCP$,
%with $\mu$ reference measure, 
$\phi : X \to \R^{d}$ be a bi-Lipschitz map for some positive $d$ and call $\eta : = (\phi)_{\sharp} \mu$.  
Proving that $\eta$ verifies the non-degeneracy condition (see Definition \ref{D:nondegenerate}) 
would imply, thanks to Theorem \ref{T:1}, that $\mathcal{L}^{d}(\phi(X)) > 0$; yielding the  rectifiability property of $(X,\sfd,\mu)$. 
\item In the recent \cite{CJ}, Csornyei and Jones proved that if a Radon measure 
$\mu$ over $\R^{d}$ makes the triple $(\R^{d},|\cdot|, \mu)$ a Lipschitz differentiability space (see \cite{C} for the definition of Lipschitz differentiability space), 
then $\mu$ is absolutely continuous with respect to the $d$-dimensional Lebesgue measure $\L^{d}$.
Then Theorem \ref{T:1} could be compared with \cite{CJ}. In particular a relevant question could be to check if a metric measure space verifying the 
\emph{non-degeneracy condition} is also a Lipschitz differentiability space.

\item  Understanding and characterizing the singular part of a non smooth space is  an important and  challenging issue. 
As explained above the set of inversion points is strongly linked with the regular set, i.e. the complementary of the singular set. 
It would be interesting  to further investigate this link for instance in Alexandrov, or more generally in ${\sf RCD}^*(K,N)$-spaces.
\item  Given an  ${\sf RCD}^*(K,N)$-space $(X,\sfd,\mm)$, it is an interesting open problem whether the measure $\mm$ 
is absolutely continuous with respect to the Hausdorff measure of the relevant dimension. 
We believe that the techniques here introduced could be useful to attach this question which would be the natural 
generalization to ${\sf RCD}^*(K,N)$-spaces to the aforementioned result of 
Cheeger-Colding \cite{cc:IIIJDE2000} in the framework of $pmGH$-limits of Riemannian manifolds satisfying lower Ricci curvature bounds.
\item In Section \ref{S:stability} we address the problem of stability of a reference measure with respect to the $pmGH$-convergence. 
As a next step one would like to study measured tangent spaces to pointed  m.m.s. endowed with a reference measure. 
Recall that the set ${\rm Tan}(X,\sfd,\mu,\bar x)$ of tangent spaces to $(X,\sfd,\mu)$ at $\bar x$ is the collection of all $pmGH$ limits of sequences 
$\left(X, 1/r_{i}\cdot \sfd, \mu_{\bar x}^{r_{i}}, \bar x \right)$ where $\mu_{\bar x}^{r_{i}}$ is a properly rescaled version of $\mu$.
\end{itemize}

%
%The combination of this fact with Theorem \ref{T:3} will imply the uniqueness of   the measure on Ricci limit spaces, more precisely: 
%
% \begin{corollary}  Let $\{(M_i, g_i, \bar x_i)\}_{i \in \N}$ be a  sequence of complete smooth $n$-dimensional  Riemannian manifolds having Ricci curvature uniformly bounded from below,  i.e. $Ric_{g_i}\geq K\, g_i$, and denote with  $\tilde{\mu}_{g_i}$ the Riemannian volume mesure constantly rescaled so that  $\tilde{\mu}_{g_i}(B_1(\bar{x}_i))=1$.
% 
%  If  $(X, \sfd, \mu, \bar{x})$ and $(X,\sfd, \xi, \bar{y})$ are two limit points of $\{(M_{i}, \sfd_{g_{i}}, \tilde{\mu}_{g_i}, \bar{x}_i)\}_{i \in \N}$  in the  $pmGH$ sense, then
% $$\mu \sim \xi,$$
%  i.e. $\mu \ll \xi$ and $\xi \ll \mu$. 
% \end{corollary}
%
%Let us mention that such  result was already obtained by Cheeger-Colding  \cite{cc:IIIJDE2000}, but one of the main contributions of this work is to give a non-smooth/metric generalization of this statement, namely Theorem \ref{T:3}. A  second reason to include it in the paper is to show the applicability of the introduced techniques.
%\\
%

\bigskip

\noindent {\bf Acknowledgment.}  A.M. acknowledges the support of the ETH Fellowship.  The project started during the conference YEP XI  at EURANDOM (Eindhoven, The Netherlands) and continued during the ``ERC Research Period on Calculus of Variations and Analysis in Metric Spaces'' at Scuola Normale Superiore (Pisa-Italy), the  authors wish to thank the institutions  for the excellent working conditions and  the organizers of the events  for the stimulating atmosphere.  Many thanks also to Luigi Ambrosio and Tapio Rajala for carefully reading a preliminary version of the paper,  to Anton Petrunin for a fruitful conversation on  Alexandrov spaces and to the two anonymous referees for the valuable comments and suggestions.
\bigskip

%%%%%%%%%%%%%%%%%%%%%%%%%%%%%%%%%%%%%%%%%%%%%%%%%%%%%
%%%%%%%%%%%%%%%%%%%%%%%%%%%%%%%%%%%%%%%%%%%%%%%%%%%%%
%%%%%%%%%%%%%%%%%%%%%%%%%%%%%%%%%%%%%%%%%%%%%%%%%%%%%
%%%%%%%%%%%%%%%%%%%%%%%%%%%%%%%%%%%%%%%%%%%%%%%%%%%%%

\section{Setting}
We now recall some terminology and general notation. In what follows we say that a triple $(X,\sfd, \mm)$ is a metric measure space, m.m.s. for short, 
if and only if:
\begin{itemize}
\item $(X, \sfd)$ is a complete and separable metric space;
\item the measure $\mm$ belongs to $\M^+(X)$,
\end{itemize}
where $\M^+(X)$ denotes the space of positive Radon measure over $X$. 
We will also sometimes assume $\mm$ to be a probability measure, that is $\mm(X) =1$, but this will be specified at the beginning of each section. 
Moreover a metric space is a geodesic space if and only if for each $x,y \in X$ 
there exists $\gamma \in \Geo$ so that $\gamma_{0} =x, \gamma_{1} = y$.
Here we are using the following notation for the space of geodesics:
\[
\Geo : = \{ \gamma \in C([0,1], X):  \sfd(\gamma_{s},\gamma_{t}) = (s-t) \sfd(\gamma_{0},\gamma_{1}), s,t \in [0,1] \}.
\]
The metric ball is denoted with   $B_{r}(o): = \{ z \in X: \sfd(z,o) < r \}.$ 
Recall that for complete geodesic spaces local compactness is equivalent to properness (a metric space is proper if every closed ball is compact).
As we will always deal with spaces enjoying some type of measure contraction property implying an upper bound on the dimension of the space,
we directly assume the ambient space $(X,\sfd)$ to be proper. Hence from now on we assume the following:
the ambient metric space $(X, \sfd)$ is geodesic, complete, separable and proper.

\medskip

We will also use quite often the set of geodesics ending in a fixed point of the space: for $z \in X$
$$
\Geo(z) : = \{ \gamma \in \Geo : \gamma_{1} = z \} = \Geo \cap {\rm e}_{1}^{-1}(z),
$$
where for any $t\in [0,1]$,  ${\rm e}_{t}$ denotes the evaluation map: 
$$
  {\rm e}_{t} : \Geo \to X, \qquad {\rm e}_{t}(\gamma) : = \gamma_{t}.
$$
A geodesic metric space $(X,\sfd)$ is said to be non-branching if and only if for any $\gamma^{1},\gamma^{2} \in \Geo$, it holds:
$$
\gamma_{0}^{1} = \gamma_{0}^{2}, \ \gamma_{\bar t}^{1} = \gamma_{\bar t}^{2}, \ \bar t\in (0,1)  
\quad 
\Longrightarrow 
\quad 
\gamma^{1}_{s} = \gamma^{2}_{s}, \quad \forall s \in [0,1].
$$
We denote with $\mathcal{P}_{2}(X)$ the space of probability measures with finite second moment  endowed with the $L^{2}$-Wasserstein distance  $W_{2}$ defined as follows:  for $\mu_0,\mu_1 \in \mathcal{P}_{2}(X)$ we set
\begin{equation}\label{eq:Wdef}
  W_2^2(\mu_0,\mu_1) = \inf_{ \pi} \int_X \sfd^2(x,y) \, \pi(dxdy),
\end{equation}
where the infimum is taken over all $\pi \in \mathcal{P}(X \times X)$ with $\mu_0$ and $\mu_1$ as the first and the second marginal.
Assuming the space $(X,\sfd)$ to be geodesic, also the space $(\mathcal{P}_2(X), W_2)$ is geodesic. 
It turns out that any geodesic $(\mu_t)_{t \in [0,1]}$ in $(\mathcal{P}_2(X), W_2)$  can be lifted to a measure $\nu \in {\mathcal {P}}(\Geo)$, 
so that $({\rm e}_t)_\sharp \nu = \mu_t$ for all $t \in [0,1]$. Given $\mu_{0},\mu_{1} \in \mathcal{P}_{2}(X)$, we denote by 
$\Opt(\mu_{0},\mu_{1})$ the space of all $\nu \in \mathcal{P}(\Geo)$ for which $({\rm e}_0,{\rm e}_1)_\sharp \nu$ 
realizes the minimum in \eqref{eq:Wdef}. If $(X,\sfd)$ is geodesic, then the set  $\Opt(\mu_{0},\mu_{1})$ is non-empty for any $\mu_0,\mu_1\in \mathcal{P}_2(X)$.

It is also convenient to introduce the following closed sets: 
\begin{eqnarray}
H &: =& \left\{ (x,y,z) \in  X^{3} :  \sfd(x, y) = \sfd(x, z) + \sfd(z, y)  \right\},  \nonumber \\
H(z) &: =& P_{12}( H \cap (X^{2} \times \{ z\}))=\{(x,y)\in X^2 \,:\, \sfd(x,y)=\sfd(x,z)+\sfd(z,y)\}, \label{E:couples}
\end{eqnarray}
where for any $n \in \N$, $P_{ij} : X^{n} \to X^2$ is the projection on the $ij$-th component, for $i,j = 1, \dots, n$.

%%%%%%%%%%%%%%%%%%%%%%%%%%%%%%%%%%%%%%%%%%%%%%%%%%%
%%%%%%%%%%%%%%%%%%%%%%%%%%%%%%%%%%%%%%%%%%%%%%%%%%%

\subsection{The essential non-branching property}\label{Ss:essentially}

We recall the following definition. 
\begin{definition}\label{D:essnonbranch}
A metric measure space $(X,\sfd, \mm)$ is \emph{essentially non-branching} if and only if for any $\mu_{0},\mu_{1} \in \mathcal{P}_{2}(X)$
which are absolutely continuous with respect to $\mm$ any element of $\Opt(\mu_{0},\mu_{1})$ is concentrated on a set of non-branching geodesics.
\end{definition}

If $(X,\sfd, \mm)$ is essentially non-branching, one can deduce some information on the branching points as follows. 
Consider the set 
\begin{equation}\label{E:branching}
D : = \{ (z,w) \in X \times X:  \gamma^{1},\gamma^{2} \in \Geo, \gamma^{1}_{0} = \gamma^{2}_{0} = w, \, \gamma^{1}_{1} = \gamma^{2}_{1} = z,
\gamma^{1}\neq \gamma^{2} \},
\end{equation}
and observe that thanks to the properness assumption, it is $\sigma$-compact: it can be written as  countable union of compact sets.
Observe that also $X^{2} \setminus D$ is $\sigma$-compact and $D$ is symmetric: if $(z,w) \in D$ also $(w,z) \in D$.
From essential non-branching we can deduce that the $\sigma$-compact set
$$
D(z) : = P_{2} \left( D \cap (\{z\} \times X) \right),
$$
has $\mm$-measure zero (see for instance \cite{RS}, in any case this statement will follow by the arguments below which give a more detailed description of the branching set). 

We need a more refined property of branching structures 
that can be obtained observing that any $L^{2}$ optimal transportation between any measure and a Dirac delta, say in $z \in X$, is also 
an $L^{1}$-optimal transportation. What follows is contained in \cite{cava:RCD} and is proved for $\mathsf{RCD}$ spaces but the same result 
holds for essentially non-branching spaces. 

Consider the following closed sets: 
$$
\Gamma : = \left\{ (x,y,z) \in X^{3} : \sfd(x,z) - \sfd(y,z) = \sfd(x,y)   \right\}, \qquad  \Gamma^{-1} : = \left\{ (x,y,z) \in X^{3} : (y,x,z) \in \Gamma \right\},
$$
and $R : = \Gamma \cup \Gamma^{-1}$. %We also consider the transport set associated to $R$:\footnote{AM: mi sembra che $\mathcal{T}$ sotto coincida con  tutto $X\times X$: infatti dati ogni $x \neq z\in X$ esiste una geodetica $\gamma$ da $x$ a $z$  e, ponendo $y =\gamma_{1/2}$ otteniamo $(x,y,z)\in \Gamma$; se invece $x=z$, ogni $y\in X$ soddisfa a $(x,y,z)\in \Gamma^{-1}$. Segue che $\mathcal{T}(z)$ coincide con tutto $X$} 
%
%$$
%\mathcal{T} : = P_{13} \left( R \setminus \{ (x,y,z) \in X^{3} :  x= y \} \right).
%$$
%
%As the space $X$ is proper by assumption, the set $\mathcal{T}$ is $\sigma$-compact.
We want to analyze possible branching structures inside $X\times X$.
We therefore consider the set of forward branching: 
\begin{equation}\label{E:branchforward}
A_{+} : = \left\{ (x,z) \in X\times X : \exists \, y, w \in X, (x,y,z), (x,w,z) \in \Gamma, \, (y,w,z) \notin R \right\},
\end{equation}
and the set of backward branching: 
\begin{equation}\label{E:branchback}
A_{-} : = \left\{ (x,z) \in X\times X : \exists \, y, w \in X, (x,y,z), (x,w,z) \in \Gamma^{-1}, \, (y,w,z) \notin R \right\}.
\end{equation}
As the space $X$ is proper by assumption, it is easily seen that both $A_{+}$ and $A_{-}$ are $\sigma$-compact sets. The information we will use is the following: define 
$$
A_{+}(z) : = P_{1} \left( (X \times \{z \}) \cap A_{+} \right), \qquad A_{-}(z) : = P_{1} \left( (X \times \{z \}) \cap A_{-} \right),
$$
then it follows from Proposition 4.5 of \cite{cava:RCD} (notice that what is denoted with $\mathcal{T}$ in \cite{cava:RCD},  in the present framework coincides with the whole $X\times X$) that 
\begin{equation}\label{eq:Apm0}
\mm (A_{+ }(z) ) = \mm(A_{-}(z)) = 0,
\end{equation}
provided $(X,\sfd,\mm)$ is essentially non-branching. In particular % if $\mathcal{T}(z) : = P_{1} \left( (X \times \{z \}) \cap \mathcal{T} \right)$ then 
the $\sigma$-compact set
$$
\mathcal{T}_{nb}(z) : = X \setminus \left( A_{+}(z) \cup A_{-}(z) \right),
$$
contains all the points, up to a set of $\mm$-measure zero, moved via optimal transportation towards the end point $z$ and the 
whole set $\mathcal{T}_{nb}(z)$ is covered by a family of non-branching geodesics.
It will also convenient to introduce the set $\mathcal{T}_{nb} : = X \times X \setminus (A_{+} \cup A_{-})$ and to observe that trivially 
$$
\mathcal{T}_{nb}(z) = P_{1} \left( (X \times \{z \}) \cap \mathcal{T}_{nb} \right).
$$
\medskip

\subsection{Disintegration of measures}
We include here a version of the Disintegration Theorem (for a comprehensive treatment see for instance  \cite{Fre:measuretheory4}).

Given a measurable space $(R, \mathscr{R})$, i.e.  $\mathscr{R}$ is  a $\sigma$-algebra of subsets of $R$, 
and a function $r: R \to S$, with $S$ general set, we can endow $S$ with the \emph{push forward $\sigma$-algebra} $\mathscr{S}$ of $\mathscr{R}$:
\[
Q \in \mathscr{S} \quad \Longleftrightarrow \quad r^{-1}(Q) \in \mathscr{R},
\]
which could be also defined as the biggest $\sigma$-algebra on $S$ such that $r$ is measurable. Moreover given a probability measure  $\rho$  on 
$(R,\mathscr{R})$, define a probability  measure $\eta$ on $(S,\mathscr{S})$  by push forward via $r$, i.e. $\eta := r_{\sharp} \rho$.  

\begin{definition}
\label{defi:dis}
A \emph{disintegration} of $\rho$ \emph{consistent with} $r$ is a map (with slight abuse of notation still denoted with) $\rho: \mathscr{R} \times S \to [0,1]$ such that, set $\rho_s(B):=\rho(B,s)$, the following hold:
\begin{enumerate}
\item  $\rho_{s}(\cdot)$ is a probability measure on $(R,\mathscr{R})$ for all $s\in S$,
\item  $\rho_{\cdot}(B)$ is $\eta$-measurable for all $B \in \mathscr{R}$,
\end{enumerate}
and satisfies for all $B \in \mathscr{R}, C \in \mathscr{S}$ the consistency condition
\[
\rho\left(B \cap r^{-1}(C) \right) = \int_{C} \rho_{s}(B)\, \eta(ds).
\]
A disintegration is \emph{strongly consistent with respect to $r$} if for all $s$ we have $\rho_{s}(r^{-1}(s))=1$.
The measures $\rho_s$ are called \emph{conditional probabilities}.
\end{definition}

We recall the following version of the disintegration theorem that can be found in \cite[Section 452]{Fre:measuretheory4}.
Recall that a $\sigma$-algebra $\mathcal{J}$ is \emph{countably generated} if there exists a countable family of sets so that 
$\mathcal{J}$ coincide with the smallest $\sigma$-algebra containing them.

\begin{theorem}[Disintegration of measures]
\label{T:disintr}
Assume that $(R,\mathscr{R},\rho)$ is a countably generated probability space and  $R = \cup_{s \in S}R_{s}$ is a partition of R. Denote with $r: R \to S$ the quotient map: 
$$
 s = r(x) \iff x \in R_{s},
$$
and with $\left( S, \mathscr{S},\eta \right)$ the quotient measure space. Assume $(S,\mathscr{S})=(X,\mathcal{B}(X))$ with $X$ Polish space, where $\mathcal{B}(X)$ denotes the Borel $\sigma$-algebra.
Then there exists a unique strongly consistent disintegration $s \mapsto \rho_{s}$ w.r.t. $r$, where uniqueness is understood in the following sense:
if $\rho_{1}, \rho_{2}$ are two consistent disintegrations then $\rho_{1,s}(\cdot)=\rho_{2,s}(\cdot)$ for $\eta$-a.e. $s \in S$.
\end{theorem}

\section{Non-degenerate reference measure on $\R^{d}$}\label{S:Rn}

Let $\mm$ be a positive  Radon measure over $\R^d$ that we will consider to be equipped with the Euclidean distance. 
Denote with $\Omega$ the support of $\mm$: $\Omega:=\spt(\mm)$. 

\begin{definition}[Non-Degenerate measure]\label{D:nondegenerate}
Consider the m.m.s. $(\Omega, |\cdot|, \mm)$. We say that the measure $\mm$ is \emph{non-degenerate} if and only if for any compact set $A\subset \Omega$ of positive $\mm$-measure  it holds 
\begin{equation}\label{E:nondegenerate}
  \mm(A_{t,o}) > 0, \quad \forall \, t \in [0,1),
\end{equation}
where $o$ is any element of $\Omega$ and
\begin{equation}\label{E:evolution}
   A_{t,o} : = \{ z \in \R^{d} : z = (1-t) y + t o,  \, y \in A \}.
\end{equation}
\end{definition}
\medskip

\begin{remark}
The non degeneracy condition may be regarded as a very weak curvature condition on the m.m.s. $(\Omega, |\cdot|, \mm)$. Indeed it is implied by $\MCP(K,n)$, for any $K$ and $n$.
\end{remark}

For the rest of this section we tacitly  assume that $\mm$ is a non-degenerate measure.
We start by proving a geometric property of $\Omega$.

\begin{lemma}\label{L:convex}
The set $\Omega$ is a convex subset of $\R^{d}$.
\end{lemma}
\begin{proof}
Consider $x,y\in \spt(\mm)$. Then for any positive radius $\delta$ it holds $\mm(B_{\delta}(x)\cap \Omega) >0$. 
Then \eqref{E:nondegenerate} implies that
$$
 \mm( (\bar B_{\delta}(x) \cap \Omega)_{t,y}) > 0, \quad \forall \, t \in [0,1).
$$
Now observe that for any $t\in [0,1)$, 
$$
(\bar B_{\delta}(x)\cap \Omega)_{t,y} \subset B_{(1-t)\delta} (x_{t}), \qquad \textrm{with }\ x_{t} : = (1-t) x + t y.
$$
It follows that $\mm(B_{(1-t)\delta} (x_{t})) >0$ for any positive $\delta$, therefore $x_{t} \in \spt(\mm)= \Omega$ and the claim follows.
\end{proof}

We deduce from Lemma \ref{L:convex} that $\Omega$ has a well defined dimension: 
there exists a minimal $k \in \N$ with $k \leq d$ so that $\Omega$ 
is contained in a $k$-dimensional affine subspace of $\R^{d}$ that we identify with $\R^{k}$. 

We can therefore consider the relative interior and the relative boundary of $\Omega$ in $\R^{k}$: there exist $U, C \subset \R^{k}$ 
so that $\Omega = U \cup C$ and $U \cap C = \emptyset$ with $U$ maximal open subset in $\R^{k}$ contained in $\Omega$.

\begin{proposition}\label{P:non-degenerateinterior}
If $k$ is the dimension of $\Omega$, the measure $\mm\llcorner_{U}$ is absolutely continuous with respect to the $k$-dimensional Lebesgue measure $\L^{k}$. 
\end{proposition}

\begin{proof}

Assume by contradiction the existence of a compact set $K \subset U$ so that 
\[
\mm(K) >0 ,\quad \mathcal{L}^{k}(K) = 0. 
\]

\noindent
The statement we are proving is local, therefore we can consider $x_{0} \in K$ and $\delta > 0$ so that $B_{\delta}(x_{0}) \subset U$ and $K \subset B_{\delta}(x_{0})$.
Since the measure $\mm$ is non-degenerate in $(\Omega, |\cdot|, \mm)$, for every $o \in \partial B_{\delta}(x_{0})$ and $t \in [0,1]$ it holds $\mm(K_{t,o}) > 0$. 
Therefore we have:
\begin{eqnarray}
0 &< &  \int_{[0,1]} \int_{\partial B_{\delta(x_{0})}} \mm(K_{t,o} ) \, \H^{k-1}(do)\, \mathcal{L}^{1}(dt)  \nonumber \\
&= & \mm\otimes \H^{k-1} \otimes \mathcal{L}^{1} \left( \{ (z,o,t) \in B_{\delta}(x_{0})\times \partial B_{\delta(x_{0})} \times [0,1) : z = (1-t)x +  t o, \, x \in K  \} \right)  \nonumber \\
& = &  \int_{B_{\delta}(x_{0})} \mathcal{H}^{k-1}\llcorner_{\partial B_{\delta(x_{0})}}\otimes \mathcal{L}^{1}\llcorner_{[0,1)} 
\left( \{ (o,t) :  f_{z}(o,t) \in K  \} \right) \mm(dz)  \nonumber \\
& = &  \int_{B_{\delta}(x_{0})} (f_{z})_\sharp\left( \mathcal{H}^{k-1}\llcorner_{\partial B_{\delta(x_{0})}}\otimes 
\mathcal{L}^{1}\llcorner_{[0,1)}\right)(K)~ \mm(dz), \label{eq:contr2.2}
\end{eqnarray}
where 
\[
f_{z}(o,t) : = \frac{z -to}{1-t}, \quad \forall t \in [0,1).
\]
Since $\mathcal{H}^{k-1}\llcorner_{\partial B_{\delta(x_{0})}}\otimes \mathcal{L}^{1}\llcorner_{[0,1)}$ is equivalent to $\mathcal{L}^{k}\llcorner_{B_{\delta}(x_{0})}$, 
the function $f_{z}$ can be defined directly on $B_{\delta}(x_{0})$.
For ease of notation we also assume $x_{0}=0$, then:
\[
f_{z}(w) : = \frac{z-w}{1-|w|/\delta} = \delta \cdot \frac{z-w}{\delta - |w|}.
\]
Then 
\begin{align*}
df_{z}(w) &~ = \delta (z-w) \nabla\left(\frac{1}{\delta-|w|}\right)^{T}  - \frac{\delta}{\delta -|w|} Id \crcr
&~ = \delta (z-w) \frac{1}{(\delta - |w|)^{2}} \left(\frac{w}{|w|}\right)^{T} - \frac{\delta}{\delta -|w|} Id. 
\end{align*}
Using the following Lemma \ref{lem:LinAlg}, we get
\begin{align*}
\det\left( - \frac{(\delta - |w|)}{\delta} df_{z}(w)\right) &~= \det \left(Id - \frac{1}{\delta - |w|} (z-w) \left(\frac{w}{|w|}\right)^{T} \right) \crcr
&~ = 1 - \frac{1}{(\delta - |w|)|w|} \langle z - w, w \rangle.
\end{align*}
Thus $\det\left( - \frac{(\delta - |w|)}{\delta} df_{z}(w)\right) = 0$ only when $\delta|w| = \langle z,w \rangle$. 
Since $z,w \in B_{\delta}(0)$, it follows that $\det(df_{z}(w)) \neq 0$ for all $w$ and $z$.
Therefore, $\mathcal{L}^{k}(K) = 0$ implies $(f_{z})_\sharp\left( \mathcal{H}^{k-1}\llcorner_{\partial B_{\delta(x_{0})}}\otimes \mathcal{L}^{1}\llcorner_{[0,1)}\right)(K) = 0$ 
contradicting  \eqref{eq:contr2.2} and  the claim is proved.
\end{proof}

\begin{lemma}\label{lem:LinAlg}
Let $Id$ be the $n\times n$ identity matrix and let $a, b$ be  $n\times 1$ column vectors.  
Then
$$
\det(Id+a b^T)=1+b^T \, a.
$$
\end{lemma}

\begin{proof}
Observe that one can write
$$
\begin{bmatrix} 
Id & 0 \\  
b^T & 1 
\end{bmatrix} 
\cdot 
\begin{bmatrix} 
Id+ab^T & a \\  
0 & 1 
\end{bmatrix}  
\cdot  
\begin{bmatrix} 
Id & 0 \\ 
-b^T & 1 
\end{bmatrix}  
= 
\begin{bmatrix} 
Id & a \\  
0 & 1+b^T \, a 
\end{bmatrix} 
$$
The proof follows then by the standard product rules of the determinant and  of matrices in block forms.
\end{proof}

To conclude we analyze the boundary. Putting together the next Proposition and the previous results, 
Theorem \ref{T:1} is proved.

\begin{proposition}
If $C$ is the boundary of $U$ for the trace topology of $\R^{k}$, then $\mm(C) = 0$.
\end{proposition}

\begin{proof}
Since $\Omega$ is convex, there exists $o \in U$ so that $C_{t,o} \subset U$ for all $t \in (0,1)$, where $C_{t,o}$ is defined as in \eqref{E:evolution}.
Suppose now by contradiction that $\mm(C)>0$. Then by non-degeneracy 
\[
\mm(C_{t,o}) >0, \quad \forall\, t \in (0,1),
\]
and therefore from Proposition \ref{P:non-degenerateinterior}, $\mathcal{L}^{k}(C_{t,o}) >0$ for all $t \in (0,1)$. But on the other hand, $C$ is the boundary of a convex set $U$ in $\R^{k}$,  thus $\mathcal{L}^{k}(C) = 0$; therefore, 
since the map
\[
C \ni x \mapsto T_{t} (x) : = (1-t)x + t o
\]
is a diffeomorphism, it follows that $\mathcal{L}^{k}(C_{t,o}) = 0$. This gives a contradiction and the claim follows.
\end{proof}

\medskip

%%%%%%%%%%%%%%%%%%%%%%%%%%%%%%%%%%%%%%%%%%%%%%%%%%%
%%%%%%%%%%%%%%%%%%%%%%%%%%%%%%%%%%%%%%%%%%%%%%%%%%%
%%%%%%%%%%%%%%%%%%%%%%%%%%%%%%%%%%%%%%%%%%%%%%%%%%%
%%%%%%%%%%%%%%%%%%%%%%%%%%%%%%%%%%%%%%%%%%%%%%%%%%%
%%%%%%%%%%%%%%%%%%%%%%%%%%%%%%%%%%%%%%%%%%%%%%%%%%%

\section{General properties of $\MCP$}

This Section serves as a general picture of the properties of m.m.s. verifying $\MCP$ and it 
is divided into two parts. In the first one we analyze the geodesic convexity of the support of the measure,  in the second one we investigate the implications of $\MCP$ on the extendability  of geodesics.

So again we recall that $(X,\sfd, \mm)$ is a metric measure space; i.e.  $(X,\sfd)$ is a   geodesic, proper, complete and separable metric space and $\mm$ is a positive Radon measure on $X$.

%%%%%%%%%%%%%%%%%%%%%%%%%%%%%%%%%%%%%%%

The quantitative dependance of $\MCP$ with respect to lower curvature bound $K$ and the dimension upper bound $N$ will not play any role in our proofs,
we prefer therefore to consider the following \emph{qualitative} form of $\MCP$. 

\begin{definition}\label{D:qualitative}
Let $(X, \sfd, \mm)$ be a m.m.s. as before, denote with $\Omega:=\spt (\mm)$ and let $\bar{x}\in X$. The triple $(X,\sfd, \mm)$ verifies the \emph{qualitative} $\MCP$ 
if and only if there exist a sequence $R_j\uparrow +\infty$  and  continuous functions $f_j:[0,1]\to \R^+$ with $f_j(0)=1$,  $f_j(1)=0$ and $f_j(t)>0$ for any $t \in [0,1)$, $j \in \N$, 
so that for every $o \in \Omega \cap B_{R_j}(\bar{x})$:  
\begin{equation}\label{eq:QualMCP}
  \mm(A_{t,o}) \geq  f_j(t) \; \mm (A ), \qquad \forall \, t \in [0,1),
\end{equation}
where 
\begin{equation}\label{E:d-evolution}
A_{t,o} : = \{ \gamma_{t} : \gamma \in \Geo(o), \gamma_{0} \in A \,\},
\end{equation}
and $A \subset \Omega \cap B_{R_j}(\bar{x})$ is any Borel set.
\end{definition}

Clearly the property of satisfying the qualitative $\MCP$ condition  above is independent of the center $\bar{x}$, just by changing the sequence of radii $R_j\uparrow +\infty$.

At this level of generality not so much can be said on the geometry of $\Omega$.
We therefore introduce the next stronger version of the qualitative measure contraction property.

\begin{definition}[Strong Qualitative $\MCP$]\label{D:strongMCP}
A triple $(\Omega,\sfd,\mm)$ as above satisfies the \emph{strong qualitative} $\MCP$ provided 
there exist a sequence $R_j\uparrow +\infty$  and  continuous functions $f_j:[0,1]\to \R^+$ with $f_j(0)=1$,  $f_j(1)=0$ and $f_j(t)>0$ for any $t \in [0,1)$, $j \in \N$, 
so that for every $o \in \Omega \cap B_{R_j}(\bar{x})$:  
\begin{equation}\label{eq:StrongQualMCP}
  \mm(A) \geq  f_j(t) \; \mm\left( {\rm e}_{0} \left( {\rm e}_{t}^{-1}(A) \cap \Geo(o) \right) \right), \qquad \forall \, t \in [0,1),
\end{equation}
where $A \subset \Omega \cap B_{R_j}(\bar{x})$ is any Borel set. 
\end{definition}

\medskip

\begin{remark} \label{rem:strongMCP}
Few comments on Definition \ref{D:strongMCP} are in order. 
The adjective ``strong'' is due to the fact that Definition \ref{D:strongMCP} is equivalent to impose \eqref{eq:QualMCP}
replacing $A_{t,o}$ with ${\rm e}_{t}(H)$, for all the $H \subset \Geo(o)$ so that ${\rm e}_{0}(H) = A$.
Hence Definition \ref{D:strongMCP} imposes the evolution estimate along the support of any $W_{2}$-geodesic connecting $\mm(A)^{-1} \mm\llcorner_{A}$ to $\delta_{o}$.
It follows therefore that strong qualitative $\MCP$ implies qualitative $\MCP$.

On the other hand, if $(X,\sfd, \mm)$ satisfies the qualitative $\MCP$ 
and it is essentially non-branching (see \eqref{E:branching}), then it
verifies the strong qualitative $\MCP$: 

Fix any $o \in \Omega$ and  $A\in B_{R_j}(\bar{x})$, and let $H = {\rm e}_{t}^{-1}(A) \cap \Geo(o)$.    Denote with $\hat X$ the set $X \setminus D(o)$ and recall that,  by the essentially non-branching condition, one has  $\mm(D(o)) =0$.
The qualitative $\MCP$ implies that 
\begin{eqnarray}
\mm(A)&\geq& \mm \left({\rm e}_{t}(H)\right) \geq \mm \left(  \left({\rm e}_{0}(H) \cap \hat X\right)_{t,o}  \right) \geq f_j(t) \, \mm( {\rm e}_{0}(H) \cap \hat X) = f_j(t) \, \mm( {\rm e}_{0}(H) )   \nonumber \\
&=&  f_j(t)\, \mm({\rm e}_{0}(  {\rm e}_{t}^{-1}(A)  \cap \Geo(o)  ) ),  \nonumber
\end{eqnarray}
which is  \eqref{eq:StrongQualMCP}.

\end{remark}

All the constructions that we will consider will be then used for spaces verifying the essentially non-branching property.
Hence we prefer to not pursue the most general result and to prove the next two statements under the strong version of the qualitative $\MCP$. 
At the moment we do not know if they hold under the general qualitative $\MCP$.

The terminology ``strong'' is justified by the next  result.

\begin{lemma}\label{L:strongeodesic}
Let $(\Omega,\sfd,\mm)$ be a metric measure space verifying the strong qualitative $\MCP$ condition  \eqref{eq:StrongQualMCP}. 
Then for any $A \subset \Omega$ measurable set with $\mm(A) \in (0,\infty)$, $o \in \Omega$ and for any geodesic 
$(\mu_{t})_{t\in [0,1]} \subset \mathcal{P}_{2}(X)$ so that 
$$
\mu_{0} = \frac{1}{\mm(A)} \mm \llcorner_{A} ,\qquad \mu_{1} = \delta_{o},
$$
it holds $\mu_{t}\ll \mm$ for all $t \in [0,1)$.
\end{lemma}

\begin{proof}
Fix $A\subset \Omega$ measurable with $\mm(A)\in(0, \infty)$, a point $o \in \Omega$ and let $(\mu_{t})_{t\in [0,1]} \subset \mathcal{P}_{2}(X)$ be an arbitrary $W_2$-geodesic from $\mu_{0} := \frac{1}{\mm(A)} \mm \llcorner_{A}$ to $\mu_{1} := \delta_{o}$. 
Since $(X,\sfd)$ is a geodesic space, we can associate a measure $\nu \in \mathcal{P}(\Geo)$ so that 
for any $t \in [0,1]$ it holds $\left({\rm e}_{t}\right)_{\sharp} \nu = \mu_{t}$; finally denote with $H \subset \Geo(o)$ a set where $\nu$ is concentrated.

Fix $t \in [0,1)$ and $K\subset \Omega$ compact subset so that $\mm(K) =0$. 
Then, by \eqref{eq:StrongQualMCP}, we have 
$$
\mm \left({\rm e}_{0} \left( {\rm e}_{t}^{-1}(K) \cap \Geo(o) \right) \right) = 0,
$$
and consequently 
\begin{align*}
\mu_t(K)&=({\rm e}_{t})_{\sharp}\, \nu(K)  =  \nu  \left( {\rm e}_{t}^{-1}(K) \cap \Geo(o)   \right)  \crcr
&~ \leq \nu \left( {\rm e}_{0}^{-1} \left( {\rm e}_{0} \left( {\rm e}_{t}^{-1}(K) \cap \Geo(o)   \right) \right)  \right)  \crcr
&~ =  ({\rm e}_{0})_{\sharp}\, \nu \left( {\rm e}_{0} \left( {\rm e}_{t}^{-1}(K) \cap \Geo(o)   \right) \right) \crcr
&~ =  \frac{1}{\mm(A)} \, \mm \llcorner_{A}  \left( {\rm e}_{0} \left( {\rm e}_{t}^{-1}(K) \cap \Geo(o)   \right) \right) \crcr
&~ = 0.
\end{align*}
\end{proof}

\noindent
In analogy of what we proved in $\R^d$, also in this framework one has that the support $\Omega$ of $\mm$ is convex.

\begin{proposition}\label{Prop:d-convex}
Let $(X,\sfd,\mm)$ be a m.m.s. and $\Omega:=\spt (\mm)$ is so that the triple $(\Omega,\sfd, \mm)$ satisfies the strong qualitative $\MCP$ condition \eqref{eq:StrongQualMCP}.  
Then  $\Omega$ is a weakly geodesically convex subset, i.e for every $x_0, x_1 \in \Omega$ there exists $\hat{\gamma}\in \Geo$ with $\hat{\gamma}_i=x_i$, $i=0,1$, such that $\hat{\gamma}([0,1])\subset \Omega$.  
\end{proposition}

\begin{proof}
Let $x,y\in \Omega$ and $\delta >0$. We consider the $\sfd^{2}$-optimal transportation problem between the marginal measures
$$
  \mu_{0} : = \frac{1}{\mm(B_{\delta}(x))}\mm \llcorner_{B_{\delta}(x)}, \qquad \mu_{1}: = \delta_{y}.
$$
The associated measure $\nu^{\delta} \in \mathcal{P}\left(\Geo \right)$ is so that $({\rm e}_{t})_\sharp (\nu^{\delta})_{t \in [0,1]}$ is a $W_{2}$-geodesic 
connecting $\mu_{0}$ to $\mu_{1}$. We now consider a sequence $\delta_{n} > 0$ converging to $0$ as $n \to \infty$. 
The associated measure $\nu^{\delta_{n}}$ gives measure one to the following set
$$
  G_{n} : = \{ \gamma \in \Geo : \gamma_{0} \in \overline{B}_{\delta_{n}}(x), \gamma_{1} = y \}.
$$
By compactness, there exists a subsequence, still denoted with $\delta_{n}$, so that $G_{n}\to H$ in the Hausdorff distance, with $H$ compact subset of 
$$
  G : = \{ \gamma \in \Geo : \gamma_{0} = x, \gamma_{1} = y \}.
$$
Also the set $\{ \nu^{\delta^{n}} \}_{n \in \N}\subset  \mathcal{P}\left(\Geo \right)$  is $W_2$-precompact, and extracting another subsequence we deduce the existence of $\nu \in  \mathcal{P}\left(\Geo \right)$ so that 
$\nu^{\delta_{n}} \to \nu$ in the $W_2$-metric. Moreover $\nu(H) = 1$ and
there exists $\hat \gamma \in H$ so that $\nu(B^{\Geo}_{\alpha}(\hat \gamma)) >0$ for all $\alpha > 0$, where 
$$ 
B^{\Geo}_{\alpha}(\hat \gamma) : = \{ \gamma \in \Geo :  \sup_{t \in [0,1]} \sfd(\gamma_{t},\hat \gamma_{t}) < \alpha \}.
$$
By lower-semicontinuity over open sets, 
$$
\liminf_{n\to \infty} \nu^{\delta_{n}} (B^{\Geo}_{\alpha}(\hat \gamma)  ) \geq \nu (B^{\Geo}_{\alpha}(\hat \gamma)  ) > 0.
$$
Hence for every $\alpha>0$ there exists $\delta_{n}$ so that 
$$
0< \nu^{\delta_{n}} (B^{\Geo}_{\alpha}(\hat \gamma)  ) \leq \left( ({\rm e}_{t})_\sharp \, \nu^{\delta_{n}} \right) \left( { \rm e}_{t} (B^{\Geo}_{\alpha}(\hat \gamma)) \right) \leq 
\left( ({\rm e}_{t})_\sharp \, \nu^{\delta_{n}} \right) \left( B_{\alpha}(\hat{\gamma}_{t}) \right),
$$
for all $t \in [0,1]$. 
By Lemma \ref{L:strongeodesic} we also  deduce that $\left(({\rm e}_{t})_\sharp \, \nu^{\delta_{n}}\right) \ll \mm\llcorner_{{\rm e}_{t}(G_{n})}$, for all $n \in \N$ and $t \in [0,1)$.
We can therefore summarize what  we obtained as follows: there exists $\hat \gamma \in \Geo(X)$ so that 
$$
   \mm\big(B_{\alpha}(\hat \gamma_{t})\big) > 0, \quad \text{for every $\alpha>0$ and $t\in [0,1)$}. 
$$
It follows that $\hat \gamma_{t} \in \spt(\mm)=\Omega$ for every $t \in [0,1]$, as desired.
\end{proof}

Thanks to Proposition \ref{Prop:d-convex}, $(\Omega,\sfd, \mm)$ is a geodesic metric measure space. 
We will identify with no loss of generality the whole space $X$ with the support $\Omega$ of the measure $\mm$. 
Therefore any future assumption on the smoothness or geometry of the m.m.s. $(X,\sfd, \mm)$ is indeed an assumption
that should be verified only on $\Omega$.

%%%%%%%%%%%%%%%%%%%%%%%%%%%%%%%%%%%%%%%%%%%%%%%%%%%%%
%%%%%%%%%%%%%%%%%%%%%%%%%%%%%%%%%%%%%%%%%%%%%%%%%%%%%
%%%%%%%%%%%%%%%%%%%%%%%%%%%%%%%%%%%%%%%%%%%%%%%%%%%%%

\subsection{Structure of geodesics under $\MCP$}
In this section we establish some structural properties of the geodesics implied by the qualitative $\MCP$ condition \eqref{eq:QualMCP} 
under the essentially non-branching condition. 
The next is a well-known property of m.m.s. enjoying $\MCP(K,N)$ and it is valid as 
well under the qualitative $\MCP$ (see for instance the proof of \cite[Lemma 3.1]{GMR}). 
\medskip

\noindent
\emph{Let the  m.m.s. $(X,\sfd, \mm)$ satisfy the qualitative $\MCP$.
Then for any $z \in X$, $\mm$-a.e. point of $X$ is an interior point of some geodesic ending in $z$:
$$
\mm \left( \left\{ \gamma_{t} : \gamma \in \Geo(z),\, t \in (0,1) \right\} \right) = 1.
$$
}

\noindent
The goal of this section is to improve such statement. For any point $z$ of the space,  consider the set of geodesics having $z$ as an interior point, called  \emph{rays through} $z$ and 
the evaluation of them,  called \emph{points through} $z$. They are defined in the following way:
$$
R(z) : = \{ \gamma \in \Geo : \gamma_{s} = z,\text{ for some } s \in (0,1) \}, \qquad {\rm e}_{(0,1)}(R(z)) : =\{ \gamma_{t} \colon \gamma \in R(z), \ t \in (0,1) \}.
$$
We can now state the result. 

%\begin{proposition}\label{P:manycurves}
%Let the m.m.s. $(X,\sfd,\mm)$ with $\mm(X)=1$ be bounded, essentially non-branching, and  satisfy the qualitative $\MCP$ condition \eqref{eq:QualMCP}. 
%Then for $\mm$-a.e. $z \in X$ the following holds: 
%\begin{equation}\label{E:manycurves}
%\mm \left( {\rm e}_{(0,1)}(R(z) ) \right)= 1.
%\end{equation}
%\end{proposition}
%
%
%\begin{proof}
%Fix $t\in (0,1)$.  By the qualitative $\MCP$ condition \eqref{eq:QualMCP} we have
%%
%$$
%\int_{X} \mm \left(  \left(X \right)_{t,x}   \right) \mm(dx) \geq f(t).
%$$
%%
%In equivalent terms,
%\begin{align*}
%f(t)   &~ \leq  \mm\otimes \mm \left( \left\{ (z,x) \in X \times X : z = \gamma_{t}, \, \gamma \in \Geo(x)  \right\} \right) \crcr
%&~ = \int_{X} \mm\left( \left\{ x \in X : z = \gamma_{t}, \, \gamma \in \Geo(x) \right\} \right) \mm(dz).
%\end{align*}
%Adding the time variable, as $f$ is continuous and $f(0)=1$, we get that for $\mm$-a.e. $z \in X$ 
%%
%$$
%\mm \left( \left\{ x \in X : z = \gamma_{t}, \, \gamma \in \Geo(x), \, t\in (0,1) \right\} \right) = 1.
%$$
%%
%To conclude we prove that set in the left hand side of the previous identity 
%coincides with  ${\rm e}_{(0,1)}(R(z))$, up to a set of $\mm$-measure zero. 
%To this purpose it is enough to show that the set of initial points of maximal geodesics has $\mm$-measure zero; 
%but this follows from Proposition \ref{P:disintegration} below, since $\mm(D(z))=0$ by the essential non-branching assumption.
%\end{proof}

\begin{proposition}\label{P:manycurves}
Let the m.m.s. $(X,\sfd,\mm)$  be  essentially non-branching and  satisfy the qualitative $\MCP$ condition \eqref{eq:QualMCP}. 
Then for $\mm$-a.e. $z \in X$ the following holds: 
\begin{equation}\label{E:manycurves}
\mm \left( X \setminus {\rm e}_{(0,1)}(R(z) ) \right)= 0.
\end{equation}
\end{proposition}

\begin{proof}
We will prove the following: for any $x \in X$ and $r >0$, for $\mm$-a.e. $z \in B_{r}(x)$, it holds 
$$
\mm(B_{r}(x) \setminus \ee_{(0,1)}(R(z))) = 0. 
$$
So let us fix $x \in X$ and $r > 0$. We also consider the function $f = f_{j}$ given by the qualitative $\MCP$ condition associated to an open ball containing $B_{r}(x)$.
Then we have
$$
\int_{B_{r}(x)} \mm \left(  \left(B_{r}(x) \right)_{t,z}   \right) \mm(dz) \geq f(t) \, \mm(B_{r}(x))^{2}.
$$
In equivalent terms,
\begin{align*}
\mm(B_{r}(x))^{2} \, f(t)  	&~ \leq  \mm\otimes \mm \left( \left\{ (z,w) \in B_{r}(x) \times X : w = \gamma_{t}, \, \gamma \in \Geo(z), \, \gamma_{0} \in B_{r}(x)  \right\} \right) \crcr
	&~ = \int_{X} \mm\left( \left\{ z \in B_{r}(x) : w = \gamma_{t}, \, \gamma \in \Geo(z), \gamma_{0} \in B_{r}(x) \right\} \right) \mm(dw) \crcr
	&~ = \int_{B_{(1+2t)r}(x)} \mm\left( \left\{ z \in B_{r}(x) : w = \gamma_{t}, \, \gamma \in \Geo(z), \gamma_{0} \in B_{r}(x) \right\} \right) \mm(dw) \crcr
	&~ \leq \int_{B_{(1+2t)r}(x)} \mm\left( \left\{ z \in B_{r}(x) : w \in \gamma_{(0,1)}, \, \gamma \in \Geo(z), \gamma_{0} \in B_{r}(x) \right\} \right) \mm(dw).
\end{align*}
Taking the limit as $t \to 0$, since $f$ is continuous and  $f(0)=1$, we get that for $\mm$-a.e. $w \in B_{r}(x)$ 
$$
\mm \left( \left\{ z \in B_{r}(x) : w \in \gamma_{(0,1)}, \, \gamma \in \Geo(z), \, \gamma_{0} \in B_{r}(x) \right\} \right) =  \mm(B_{r}(x)).
$$
To conclude we prove that set in the left hand side of the previous identity 
coincides with  ${\rm e}_{(0,1)}(R(z))$, up to a set of $\mm$-measure zero. 
To this purpose it is enough to show that the set of initial points of maximal geodesics has $\mm$-measure zero; 
but this follows from Proposition \ref{P:disintegration} below, since $\mm(D(z))=0$ by the essential non-branching assumption.
\end{proof}

Proposition \ref{P:manycurves} has a nice consequence on the symmetric cut locus. 
Let us recall that the \emph{symmetric cut locus} $SC \subset X\times X$ is defined by
\begin{eqnarray}
SC &:=&X\times X\setminus\{ (x,y) \in X\times X \,:\, \exists\, \gamma\in \Geo , \, \exists\, s,t \in (0,1) \text{ such that } \gamma_{s}=x, \gamma_{t}=y\} \nonumber \\
      &=&  \{ (x,y) \in X\times X\,:\, x \in C(y) \text{ or } y \in C(x)\} \label{eq:defSC},
\end{eqnarray}
where, given $z \in X$, the \emph{cut locus} of $z$ denoted with $C(z) \subset X$ is defined by
\begin{equation}\label{eq:defCutLocus}
C(z):=X\setminus\{ x \in X \,:\, \exists \gamma\in \Geo, \, \exists t \in [0,1) \text{ such that } \gamma_{0}=z, \gamma_{t}=x\}.
\end{equation}
It was already well-known that qualitative $\MCP$ implies that $\mm(C(z))=0$, see for instance \eqref{eq:Apm0}. 
From Proposition \ref{P:manycurves}, once we observe that for every $z \in X$ it holds 
$$
%{\rm e}_{(0,1)}(R(z))\subset X \setminus C(z)\quad  \text{ and } \quad 
\{z\}\times {\rm e}_{(0,1)}(R(z)) \subset X\times X \setminus SC,
$$
it is then possible to deduce the following property of $SC$. 

\begin{corollary}\label{cor:CutLocus}
Let the m.m.s. $(X,\sfd,\mm)$ be essentially non-branching and satisfy the qualitative $\MCP$. Then 
\begin{equation}\nonumber
%\mm(C(z))=0, \quad \text{ for $\mm$-a.e. $z \in X$,} \qquad \text{and }\quad  
\mm \otimes \mm \,(SC)=0 .
\end{equation} 
\end{corollary}

The next result is already present in literature, see for instance \cite{biacava:strettconv} where a more general case is considered. 
To keep presentation as self contained as possible, we include it  here  in a form adjusted to our framework.

\begin{proposition}[Disintegration with $\MCP$]\label{P:disintegration}
Let the m.m.s. $(X,\sfd,\mm)$ with $\mm(X)=1$ be essentially non-branching, and  satisfy the qualitative $\MCP$ condition \eqref{eq:QualMCP}. 
Fix any $z \in X$ and define $\hat X : = X \setminus \{D(z) \cup \{z\} \}$. Consider in $\hat X$ the equivalence relation $T$:
$$
(x,y) \in T  \iff \sfd(y,z) = \sfd(y,x) + \sfd(x,z), \textrm{ or }  \sfd(x,z) = \sfd(x,y) + \sfd(y,z),
$$
whose equivalence classes are maximal geodesics in $\Geo(z)$. 
Then there exists an $\mm$-measurable map $Q : \hat X \to \hat X$ and a Borel subset $K \subset \hat{X}$ with $\mm(K)=1$ so that 
$$
(x,y) \in T \iff Q(x) = Q(y) \qquad \textrm{ and } \qquad (x,Q(x)) \in T, \qquad \forall  x,y \in K.
$$
Moreover
$$
\mm = \mm \llcorner_{\hat X}\, = \, \int_{Q(\hat X) } \mm^{\alpha} q(d\alpha), \qquad 
\mm^{\alpha } \ll \mathcal{H}^{1}\llcorner_{\gamma^{\alpha}}, \quad q\textrm{-a.e.} \, \alpha \in Q(\hat X),
$$
where $q = (Q)_{\sharp} \, \mm \llcorner_{\hat X}$ is the quotient measure and $\gamma^{\alpha} \in \Geo(z)$ is the maximal geodesic containing $\alpha \in Q(\tilde{X})$. 
\end{proposition}

\medskip

\begin{proof}

\textbf{Step 1.} 
Since $\mm(\hat X) = 1$, we can assume the existence of $\hat K \subset \hat X$, $\sigma$-compact, so that $\mm(\hat K) = 1$; observe that the equivalence relation restricted to $\hat K$
$$
T = \{ (x,y) \in \hat K \times \hat K : \sfd(y,z) = \sfd(y,x) + \sfd(x,z) \} \cup \{  (x,y) \in K \times K : \sfd(x,z) = \sfd(x,y) + \sfd(y,z) \},
$$
is $\sigma$-compact as well. With a slight abuse of notation we set $T(\omega)$ to be the class of $\omega$, i.e. the maximal geodesic through $\omega$ and $z$.
Fix now a countable dense family of points $\{x_{i} \}_{i \in \N} \subset \hat K$ and consider for $i,j,k\in \N$
$$
W_{ijk} : = \{ w \in \bar B_{2^{-j}}(x_{i}) \cap \hat K :  L( T(w) \cap B_{2^{1-j}}(x_{i}) ) \leq 2^{-k} , \, L(T(w)) \geq 2^{2-k} \},
$$
where $L(\gamma)$ denotes the length of the geodesic $\gamma$.
It can be proved that $W_{ijk}$ form a countable cover of $\hat K$ of class $\mathcal{A}$, the $\sigma$-algebra generated by the analytic sets, 
see \cite[Lemma 4.1]{biacava:strettconv} for the details. As $W_{ijk} \in \mathcal{A}$, there exists $N_{ijk} \subset W_{ijk}$ with $\mm(N_{ijk}) = 0$ 
so that $W_{ijk} \setminus N_{ijk}$ is Borel. Then 
$$
H_{ijk} : = T^{-1}(W_{ijk} \setminus N_{ijk}) = P_{1} \left( \{ (v,w) \in T  :  w \in W_{ijk} \setminus N_{ijk}  \} \right)
$$
is a countable covering of $\hat K \setminus  \cup N_{ijk}$ into \emph{saturated} (sets of the form $T^{-1}(A)$) analytic sets.
Observing that the difference of two saturated sets is still saturated, we can find a countable disjoint family of sets 
$\{K_{ijk}\}$ made of $\mathcal{A}$ saturated sets so that 
$$
K_{ijk} \subset H_{ijk}, \qquad \cup_{ijk} K_{ijk} = \cup_{ijk} H_{ijk}.
$$
For ease of notation we fix a bijection from $\N \to \N^{3}$ and denote with $K_{n}$ the set $K_{i(n)j(n)k(n)}$.
We can now define the following multivalued map: 
$$
K_{n} \ni x \longmapsto F(x)  : = T(x) \cap \bar B_{2^{-(j(n))}}(x_{i(n)}),
$$
that is easily seen to be $\mathcal{A}$-measurable. Observe that $F(x)$ is a closed subset of $\hat K$. Then by \cite[Corollary 2.7]{biacava:strettconv} 
there exists $f_{n} : K_{n} \to \bar B_{2^{-(j(n))}}(x_{i(n)})$ $\mathcal{A}$-measurable so that 
$$
(x,f_{n}(x)) \in T, \qquad  (x,y) \in T \iff f_{n}(x) = f_{n}(y).
$$
We can repeat the same procedure for any $m\in \N$ and then define $Q(x) : = f_{n}(x)$ for all $x \in K_{n}$ and arbitrarily outside of $K :  = \cup_{n} K_{n}$.
As all $K_{n}$ are disjoint and saturated, the definition of $Q$ is well posed and as $\hat X \setminus K$ has $\mm$-measure $0$, 
$Q$ is $\mm$-measurable and the first part of the claim follows. 

\textbf{Step 2.} 
The existence of an $\mm$-measurable map $Q$ obtained before is in fact equivalent to the strong consistency of the disintegration, see 
Proposition 4.4 of \cite{biacava:strettconv} and subsequent discussions.
Hence 
$$
\mm = \mm \llcorner_{\hat X}\, = \, \int_{Q(\hat X) } \mm^{\alpha} q(d\alpha), \qquad \mm^{\alpha}(\gamma^{\alpha}) = 1, \quad q\textrm{-a.e.} \, \alpha \in Q(\hat X),
$$
where $q = (Q)_{\sharp} \, \mm \llcorner_{\hat X}$ is the quotient measure and $\gamma^{\alpha} \in \Geo(z)$ is the unique maximal geodesic containing $\alpha$. 
To obtain that $q$-a.e. $\mm^{\alpha}$ is absolutely continuous with respect to  $\H^{1}\llcorner_{\gamma^{\alpha}}$ it is enough 
to write the qualitative $\MCP$ between any Borel set and the fixed point $z \in X$. Then the equivalence classes are invariant sets for the evolution and, for $q$-a.e. $\alpha$,   each 
measure $\mm^{\alpha}$ has to satisfy a non-degeneracy property of the same type of the one introduced in Section \ref{S:Rn}. Hence the claim follows (for the details of this second part of the proof, see \cite[Theorem 5.7]{biacava:strettconv}).
\end{proof}

%%%%%%%%%%%%%%%%%%%%%%%%%%%%%%%%%%%%%%%%%%%%

\medskip
\medskip

\section{Reference measures}

We are now ready to define what is for us a reference measure. Recall that $(X,\sfd)$ is a geodesic, proper, complete and separable metric space 
even if the next definition will make sense in a general complete and separable metric space.
During this Section all the assumptions regarding finiteness of measures are dropped. Denote with $\M^+(X)$ the space of positive Radon measures.
%Recall that given a $\{ \mm_{k} \}_{k \in \N} \subset \M^{+}(X)$ and $\mm \in \M^{+}(X)$ we say that $\mm_{k}$ converges to $\mm$ in $\left(C_{c}(X)\right)'$ provided
%%
%$$
%\lim_{ k \to \infty} \int_{X} \f(x) \, \mm_{k}(dx) = \int_{X} \f(x) \, \mm(dx), 
%$$
%%
%for all $\f \in C_{c}(X)$, the space of continuous functions with compact support. Another possible notion of convergence in $\left(C_{b}(X)\right)'$, 
%where $C_{b}(X)$ denotes the space of continuous and bounded functions. In this case we will adopt the short notation $\mm_{k} \weak \mm$.
%Also recall that if $\mm_{k} \to \mm$ in $\left(C_{c}(X)\right)'$ and $\mm_{k}(X) \to \mm(X)$, then $\mm_{k} \weak \mm$.
%

\begin{definition}\label{D:referencemeasure}
A positive Radon measure  $\mu \in \M^+(X)$ is a \emph{reference measure for $(X,\sfd)$} provided it is non-zero, and
%$(X,\sfd, \mu)$ 
%is essentially non-branching, it 
%satisfies the qualitative $\MCP$ condition \eqref{eq:QualMCP}, and 
for $\mu$-a.e. $z \in X$ there exists $\pi^{z} \in \mathcal{M}^+(X \times X)$ so that 
\begin{equation}\label{E:reverse}
(P_{1})_{\sharp}\,\pi^{z} = \mu, \qquad \pi^{z}  (X\times X \setminus H(z))=0,  \qquad (P_{2})_{\sharp}\, \pi^{z} \ll \mu, 
\end{equation}
where $P_{i} : X\times X \to X$ is the projection on the $i$-th component, for $i =1,2$ and $H(z)$ is defined in \eqref{E:couples}.
The measure $\pi^{z}$ will be called  \emph{inversion plan}.
\end{definition}
\noindent
Let us briefly observe that, since $H(z)\subset X \times X$ is closed, clearly the requirement   $\pi^{z}  (X\times X \setminus H(z))=0$ is equivalent to $\spt (\pi^z)\subset H(z)$.
Moreover we will use the following notation 
\begin{equation}\label{E:inversionpoints}
\mathsf{Ip} (\mu) : = \{ z\in X : \, \exists \, \pi^{z}\, \textrm{inversion plan}\},
\end{equation}
to denote the set of \emph{Inversion points of} $\mu$, those point where an inversion plan exists.

\begin{remark}[Inversion points are regular points]\label{rem:Inversion}
Let us remark that the condition  for a point $p \in X$ to be an inversion point   is strictly related to the regularity of the space $(X,\sfd)$ at $p$. Indeed if $(X,\sfd)$ has a conical singularity at $p$ then clearly $p$ cannot be an inversion point (unless $\mm(\{p\})>0$) since all geodesics end to minimize length once they cross $p$. Indeed, since   $H(p)\subset X\times \{p\} \cup \{p\}\cup X$, if $\pi \in \M^+(X)$ with $ \pi (X\times X \setminus H(p))=0$ and $(P_1)_\sharp(\pi)=\mm$, then $(P_2)_\sharp(\pi)= \mm(X) \, \delta_p$.  

On the other hand in Section \ref{S:InvAlex} we will prove that $\H^n$-a.e. regular point in an 
Alexandrov space with lower curvature bound is an inversion point (for the precise statement see Theorem \ref{T:Alexandrov-general}).  
\end{remark}
%
%
%

%Even if it does not have a direct consequence in this note, we show that a reference measure enjoys a weak reverse type of $\MCP$, 
%provided the geodesic space $(X,\sfd)$ is non-branching.
%
%
%
%\begin{lemma}\label{L:reverseMCP} 
%Let $(X, \sfd, \mu)$ be a m.m.s. with $\mu$ reference measure for $(X,\sfd)$.
%Let $A$ be any $\mu$-measurable set so that $\mu(A) = 0$. Then for all $t\in (0,1)$ there exists a set $N_{t}$ of $\mu$-measure zero so that 
%%
%$$
%\mu(A_{t,x}) = 0,
%$$
%%
%for all $x \in X \setminus N_{t}$. 
%\end{lemma}
%
%\begin{proof}
%Fix any $t \in (0,1)$, and observe that 
%%
%$$
%\int_{X} \mu(A_{t, x}) \mu(dx) = \int \mu\left( \left\{ x \in X : \exists\, \gamma \in \Geo(x), \gamma_{0}\in A, \gamma_{t} = z \right\} \right) \mu(dz).
%$$
%%
%If $\pi^{z}$ is the inversion plan, since $(P_{2})_{\sharp} \pi^{z} \ll \mu$ and $(X,\sfd)$ is non-branching, the integrand in right hand side of the previous identity is null, 
%for $\mu$-a.e. $z \in \Omega$.
%\end{proof}
%
%
%
%
%
%
%%

\medskip
In order to have a clear statement for the main result of this section, 
we drop the identification between the ambient space $X$ and the support of $\mm$, $\Omega$. 
Notice that, once we assume $(X,\sfd,\mm)$ to be essentially non-branching and satisfy the qualitative $\mathsf{MCP}$ condition \eqref{eq:QualMCP},  
then by Proposition \ref{Prop:d-convex} we have that  $\Omega = \spt(\mm)$ is a geodesic space.

\begin{theorem}\label{T:metric}
Let $(X,\sfd,\mm)$ be an essentially non-branching m.m.s. that verifies the qualitative $\mathsf{MCP}$ condition  \eqref{eq:QualMCP}.
Assume the existence of a reference measure $\mu$ for $(\Omega,\sfd)$, where $\Omega= \spt (\mm)$,
so that $(\Omega, \sfd, \mu)$ verifies the qualitative $\MCP$  condition  \eqref{eq:QualMCP} and it is essentially non-branching. \\
\noindent
If $\mm ( X \setminus \mathsf{Ip}(\mu) ) = 0$, then
$$
\mm \ll \mu.
$$
\end{theorem}

\begin{proof}
We start by observing that from Proposition \ref{Prop:d-convex}, the triples $(\Omega,\sfd, \mm)$ and $(\Omega, \sfd, \mu)$
are geodesic, essentially non-branching m.m.s. verifying the qualitative $\MCP$. Once the measures will be restricted to sets of finite measure, 
we will use the result of Proposition \ref{P:disintegration} 

\textbf{Step 1.} 
Assume by contradiction the claim is false. Then there exists $A \subset \Omega$, that we can assume to be a subset of $B_{r}(o) \subset \Omega$ for some $o \in \Omega$ and $r>0$,
so that 
\[
\mm(A) > 0 \quad \textrm{and} \quad \mu(A) = 0.
\]
Since $\mm$ verifies the qualitative $\MCP$, we have the following strict inequality
\[
\int_{B_{r}(o)}   \mm (A_{1/2,x} \cap \mathcal{T}_{nb}(x)) \, \mu(dx) > 0,
\]
where $\mathcal{T}_{nb}(x)$ has been introduced in Section \ref{Ss:essentially} and guarantees that the transportation towards $x$ moves along non-branching geodesics (recall that  $\mathcal{T}_{nb}(x)$ has full measure thanks to \eqref{eq:Apm0}).  The measurability  of  $A_{1/2,x} \cap \mathcal{T}_{nb}(x)$ follows from the fact that $\mathcal{T}_{nb}$ is $\sigma$-compact; note we also used that $\mathcal{T}_{nb}(x)$ is of full $\mm$-measure thanks to the essential non-branching assumption.
Using Fubini's Theorem to change order of integration, we get
\begin{equation}\label{eq:pfm(0)}
0  <  \int_{B_{2r}(o)  } \mu \left( \left\{ x \in B_{r}(o) : \exists \, \gamma \in \Geo(x)  \text{ s.t } z = \gamma_{1/2}, \,  \gamma_{0} \in A, \, \gamma_{0},z \in \mathcal{T}_{nb}(x) \right\} \right) \mm(dz). 
\end{equation}
As $\mm ( X \setminus \mathsf{Ip}(\mu) ) = 0$, the set of integration can be substituted by $B_{2r}(o) \cap \mathsf{Ip}(\mu)$ without changing anything.
Fix any $z \in B_{2r}(o)\cap \mathsf{Ip}(\mu)$. Since $(\Omega,\sfd, \mm)$ is a geodesic m.m.s.,
from the qualitative $\MCP$ and Proposition \ref{P:disintegration} applied to $z$ and $\mu\llcorner_{B_{2r}(z)}$ 
it follows that \footnote{Observe  that convexity of $B_{2r}(z)$ is not needed to apply 
Disintegration theorem, we only use that for each point in $x \in B_{2r}(z)$, any intermediate point between $x$ and $z$ 
belongs to $B_{2r}(z)$. 
The same property permits to use $\MCP$ to obtain the absolute continuity of the conditional measures 
with respect to the Hausdorff measure of dimension 1.}
\begin{equation}\label{eq:pfm(1)}
\mu \llcorner_{B_{2r}(z)}\, = \, \int_{Q( B_{2r}(z) ) } \mu^{\alpha} q(d\alpha), \qquad 
\mu^{\alpha } \ll \mathcal{H}^{1}\llcorner_{\gamma^{\alpha}}, \quad q\textrm{-a.e.} \, \alpha \in Q( B_{2r}(z) ),
\end{equation}
where $q = Q_\sharp ( \mu \llcorner_{ B_{2r}(z)})$ is the quotient measure and $\gamma^{\alpha}$ is the maximal geodesic in $\Geo(z)$ containing $\alpha$. 
Moreover for $q$-a.e. $\alpha \in Q( B_{2r}(z) )$, the density of $\mu^{\alpha}$ with respect to $\H^{1}\llcorner_{\gamma^{\alpha}}$ 
is strictly positive at any point and so $\H^{1}\llcorner_{\gamma^{\alpha}} \ll \mu^{\alpha}$ for $q$-a.e. $\alpha \in Q( B_{2r}(z) )$.
Hence, $\mu(A)=0$ implies that 
$$
\int_{Q( B_{2r}(z) ) } \H^{1}(\gamma^{\alpha} \cap A) \, q(d\alpha) = 0.
$$
There exists therefore a $q$-measurable set $B \subset Q( B_{2r}(z) )$ with $q(B) = 0$ so that 
\begin{equation}\label{eq:pfm(2)}
\H^{1}(\gamma^{\alpha} \cap A)  = 0, \qquad \forall \, \alpha \in Q(B_{2r}(z))\setminus B.
\end{equation}

\textbf{Step 2.} 
We now write the set
 $$E=E(z): = \left\{ x \in B_{r}(o) :  \exists \, \gamma \in \Geo(x)  \text{ s.t } z = \gamma_{1/2}, \,  \gamma_{0} \in A, \, \gamma_{0},z \in \mathcal{T}_{nb}(x)  \right\}$$
  as  union of the following two sets:
\begin{align*}
E_{1} : = &~ \{x \in B_{r}(o) : \exists \, \gamma \in \Geo(x)  \text{ s.t } z = \gamma_{1/2}, \,  \gamma_{0} \in A, \, \gamma_{0},z \in \mathcal{T}_{nb}(x),  \, Q(\gamma_{0}) \in B \},  \crcr
E_{2} : = &~ \{x \in B_{r}(o): \exists \, \gamma \in \Geo(x)  \text{ s.t } z = \gamma_{1/2}, \,  \gamma_{0} \in A, \, \gamma_{0},z \in \mathcal{T}_{nb}(x),  \,
Q(\gamma_{0}) \in Q(B_{2r}(z)) \setminus B \}.
\end{align*}
Note that by the definition of $E$, the map $Q$ is well defined over $E$ and the definitions of $E_{1}$ and $E_{2}$ are well posed. In order to get the thesis it is enough to show that  both $\mu(E_1)=0$ and $\mu(E_2)=0$; this will give $\mu(E_z)=0$, for every $z \in B_{2r}(o) \cap \mathsf{Ip}(\mu)$, contradicting \eqref{eq:pfm(0)}.

In order to prove $\mu(E_1)=0$, observe that thanks to the non-branching property of $\mathcal{T}_{nb}$ it holds
\begin{equation}\label{eq:QHT}
\left( X \times Q^{-1}(B) \right) \cap H(z) \cap \mathcal{T}^{-1}_{nb} =  
\left(   P_{1}\left( \left(  X \times Q^{-1}(B)  \right) \cap H(z) \cap \mathcal{T}^{-1}_{nb} \right) \times X \right)  \cap H(z) \cap \mathcal{T}^{-1}_{nb},
\end{equation}
where we have used the following notation: $\mathcal{T}^{-1}_{nb}: = \{ (z,x) \in X\times X : \, (x,z) \in \mathcal{T}_{nb} \}$.
Indeed it is easily checked that 
\begin{align*}
&~ \left\{ (x,y)  \in  X \times Q^{-1}(B) :  (x,y) \in  H(z) \cap \mathcal{T}^{-1}_{nb} \right\}  \crcr
&~ = \left(  \left\{ x\in X : \, \exists\, y \in Q^{-1}(B),\, (x,y) \in H(z) \cap \mathcal{T}^{-1}_{nb} \right\} \times X \right) \cap H(z) \cap \mathcal{T}^{-1}_{nb}.
\end{align*}
The inclusion $\subset$ is trivial, and  % If a strict inclusion would hold, say $(x,w)$ belongs to the second set but not to the first one, then necessarily $(x,w) \notin \mathcal{T}^{-1}_{nb}$ and this would imply a contradiction.
the reverse inclusion follows directly by the non branching property of $\mathcal{T}_{nb}$.
As the right hand side term can be rewritten as
$$
\left(   P_{1}\left( \left(  X \times Q^{-1}(B)  \right) \cap H(z) \cap \mathcal{T}^{-1}_{nb} \right) \times X \right)  \cap H(z) \cap \mathcal{T}^{-1}_{nb},
$$
the identity \eqref{eq:QHT} is proved.

Since  $\mu(Q^{-1}(B)) = 0$, the definition  \eqref{E:reverse} of inversion plan $\pi^z$ implies that 
\begin{align*}
0 &~ =  \pi^{z} \left(  \left( X \times Q^{-1}(B) \right) \cap H(z) \cap \mathcal{T}^{-1}_{nb} \right) \crcr
&~    = \pi^{z}\left(  \left(   P_{1}\left( \left(  X \times Q^{-1}(B)  \right) \cap H(z) \cap \mathcal{T}^{-1}_{nb} \right) \times X \right) \cap H(z)\cap \mathcal{T}^{-1}_{nb} \right) \crcr
&~    = \mu \left(  P_{1}\left( \left(  X \times Q^{-1}(B)  \right) \cap H(z)\cap \mathcal{T}^{-1}_{nb} \right) \right). 
\end{align*}
As $E_{1} \subset  P_{1}\left( \left(  X \times Q^{-1}(B)  \right) \cap H(z) \cap \mathcal{T}^{-1}_{nb}\right)$, we have that $\mu(E_{1}) = 0$.

\noindent
The fact that $\mu(E_2)=0$ follows easily by the disintegration \eqref{eq:pfm(1)}. 
Indeed by \eqref{eq:pfm(2)}, we infer $\H^{1}(\gamma^{\alpha} \cap A)  = 0$ for every $\alpha \in Q(B_{2r}(z))\setminus B$.
Therefore, after the inversion with respect to $z$, we have
$$ 
\H^{1}(\gamma^{\alpha} \cap E_2)  = 0, \quad \forall  \alpha\in Q(E_2).
$$
Recalling that $ \mu^{\alpha } \ll \mathcal{H}^{1}\llcorner_{\gamma^{\alpha}},$ for $q$-a.e.  $\alpha \in Q( B_{2r}(z))$, the disintegration \eqref{eq:pfm(1)} then implies 
$$
\mu(E_2)= \int_{Q( E_2) } \mu^{\alpha}(E_2)\, q(d\alpha)=0. 
$$
\end{proof}

\medskip
One can drop the hypothesis that $\mm$ is concentrated on $\mathsf{Ip}(\mu)$, and still obtain the same claim at the price of 
adding regularity properties on the set of inversion points of the reference measure $\mu$.  
We see two different cases; but, before that, let us give the following definition.

\begin{definition}\label{D:muconnected}
A set $C \subset X$ is said to be \emph{$\mu$-connected} if for any compact set $A \subset C$ there exists a set $U \subset C$ with $\mu(U) > 0$ and $\bar t \in (0,1)$ so that 
$$
(A,U)_{\bar t} : = \left\{ \gamma_{\bar t} : \gamma \in \Geo, \, \gamma_{0} \in A, \, \gamma_{1} \in U \right\} \subset C. 
$$
\end{definition}

Let us mention that Definition \ref{D:muconnected} resembles the convexity property of regular points in Alexandrov spaces.

\begin{theorem}\label{T:metric2}
Let $(X,\sfd,\mm)$ be an essentially non-branching m.m.s. that verifies the qualitative $\mathsf{MCP}$ condition  \eqref{eq:QualMCP}.
Assume the existence of a reference measure $\mu$ for $(\Omega,\sfd)$, where $\Omega= \spt (\mm)$,
so that $(\Omega, \sfd, \mu)$ verifies the qualitative $\MCP$  condition \eqref{eq:QualMCP} and it is essentially non-branching. \\
\noindent
If $\mathsf{Ip}(\mu)$ is $\mu$-connected, then
$$
\mm\llcorner_{\mathsf{Ip}(\mu)} \ll \mu.
$$
\end{theorem}

\begin{proof}
The proof follows the same lines of the proof of Theorem \ref{T:metric} and we only sketch the first part of it.
Assume by contradiction the claim is false. Then there exists a compact set $A \subset \Omega \cap \mathsf{Ip}(\mu)$, that we can assume to be a subset of $B_{r}(o) \subset \Omega$ for some $o \in \Omega$ and $r>0$,
so that 
\[
\mm(A) > 0 \quad \textrm{and} \quad \mu(A) = 0.
\]
Since $\mm$ verifies the qualitative $\MCP$, we have the following strict inequality
\[
\int_{U}   \mm (A_{\bar t,x} \cap \mathcal{T}_{nb}(x)) \, \mu(dx) > 0,
\]
where $U\subset \mathsf{Ip}(\mu)$ and $\bar t \in (0,1)$ are given by the $\mu$-connectedness of $\mathsf{Ip}(\mu)$. 
Using Fubini's Theorem to change order of integration, we get
$$
0  <  \int_{ (A,U)_{\bar t}  } \mu \left( \left\{ x \in B_{r}(o) :  \exists \, \gamma \in \Geo(x) \text{ s.t. } z = \gamma_{\bar t}, \,  \gamma_{0} \in A, \, \gamma_{0},z \in \mathcal{T}_{nb}(x) \right\} \right) \mm(dz). 
$$
Since $(A,U)_{\bar t} \subset \mathsf{Ip}(\mu)$, the proof continues as in Theorem \ref{T:metric}, obtaining a contradiction and therefore the claim.
\end{proof}

We can also assume $\mathsf{Ip}(\mu)$ to be open and prove the next result.

\begin{theorem}\label{T:metric3}
Let $(X,\sfd,\mm)$ be an essentially non-branching m.m.s. that verifies the qualitative $\mathsf{MCP}$ condition  \eqref{eq:QualMCP}.
Assume the existence of a reference measure $\mu$ for $(\Omega,\sfd)$, where $\Omega= \spt (\mm)$,
so that $(\Omega, \sfd, \mu)$ verifies the qualitative $\MCP$  condition \eqref{eq:QualMCP} and it is essentially non-branching. \\
\noindent
If $\mathsf{Ip}(\mu)$ is open in $\Omega$, then
$$
\mm\llcorner_{\mathsf{Ip}(\mu)} \ll \mu.
$$
\end{theorem}

\begin{proof}
Assume by contradiction the claim is false. 
Then there exists $A \subset \Omega \cap \mathsf{Ip}(\mu)$, that we can assume to be a subset of $B_{R}(o) \subset \Omega$ for some $o \in \Omega$ and $R>0$,
so that 
\[
\mm(A) > 0 \quad \textrm{and} \quad \mu(A) = 0.
\]
By inner regularity we can also assume $A$ to be compact in $\Omega$. As $\mathsf{Ip}(\mu)$ is open, for any $w \in A$ there exists $r_{w}$ so that 
$B_{2r_{w}}(w) \subset \mathsf{Ip}(\mu)$. Then by compactness there exists $w \in A$ so that 
$$
\mm( A \cap B_{r_{w}}(w) ) > 0,
$$
and therefore we can assume with no loss of generality that $A \subset B_{r_{w}}(w) $. 
Since $\mm$ verifies the qualitative $\MCP$, we have the following strict inequality
\[
\int_{B_{r_{w}}(w)}   \mm (A_{\frac{1}{2},x} \cap \mathcal{T}_{nb}(x)) \, \mu(dx) > 0.
\]
Using Fubini's Theorem to change order of integration, we get
$$
0  <  \int_{ B_{2r_{w}}(w)   } \mu \left( \left\{ x \in B_{r}(o) : \exists \, \gamma \in \Geo(x) 
\text{ s.t. }  z = \gamma_{\bar t}, \,  \gamma_{0} \in A, \, \gamma_{0},z \in \mathcal{T}_{nb}(x) \right\} \right) \mm(dz). 
$$
Since $B_{2r_{w}}(w)  \subset \mathsf{Ip}(\mu)$, the proof continues as in Theorem \ref{T:metric}, obtaining a contradiction and therefore the claim.
\end{proof}

\section{Properties of reference measures}\label{S:stability}

In this section we study few natural properties of reference measures.  
\\As one would expect, the property of being a reference measure is invariant under multiplication by a positive function. 
 
\begin{proposition}[Invariance]
Let $(X,\sfd,\mu)$ be a m.m.s. with $\mu$ reference measure for $(X,\sfd)$. Let $f : X \to \R$ be a Borel function such that $\mm : = f \cdot \mu$ is still a Radon measure. If 
$$
\mu \left(  \{ x\in X \colon f(x) = 0 \} \right)= 0, 
$$
then $\mm$ is a reference measure for $(X,\sfd)$.
\end{proposition}

\begin{proof}
Consider the set of inversion points of $\mu$, $\Ip(\mu)$, and observe that $\mm(X \setminus \Ip(\mu)) = 0$.
For $z \in \Ip(\mu)$ let $\pi^{z}$ be the associated inversion plan and define the rescaled plan 
$$
\hat \pi^{z} : = f\circ P_1 \; \cdot \pi^{z}.
$$
Since clearly $(P_{1})_{\sharp} \hat \pi^{z} = \mm$,  it only remains to prove that $(P_{2})_{\sharp} \hat \pi^{z} \ll \mm$. To this aim let  $A \subset X$ be a Borel set with $\mm(A) = 0$ and observe that,  
 by the assumption on $f$, also $\mu(A) = 0$.  By using that $\pi^z$ is an inversion plan we infer 
$$
\pi^{z} (X \times A) =  (P_2)_{\sharp} \pi^z (A)=0. 
$$
Hence $(P_{2})_{\sharp} \hat \pi^{z}(A) = f\circ P_1 \;  \cdot \pi^{z} (X \times A ) =  0$ and the claim follows.
\end{proof}

\medskip

The next property we would like to investigate is the locality, that means to obtain an inversion plan starting from a \emph{local inversion plan}. Let us define the latter object:   \\
Given a m.m.s. $(X,\sfd, \mu)$ and $z \in X$ we say that $\pi^{z} \in \M_{+}(X \times X)$ is a \emph{local inversion plan} for $\mu$
provided 
\begin{equation}\label{E:localreverse}
(P_{1})_{\sharp}\,\pi^{z} = \mu\llcorner_{B_{r}(z)}, \qquad \pi^{z}  (X\times X \setminus H(z))=0,  \qquad (P_{2})_{\sharp}\, \pi^{z} \ll \mu, 
\end{equation}
for some $r > 0$. The set of those points admitting a local inversion plan will be denoted by $\Ip_{loc}(\mu)$ 
and will be called the set of  points of local inversion.

Trivially restricting the inversion plan to the set $B_{r}(z) \times X$, any inversion point is a local inversion point, that is $\Ip_{loc}(\mu) \subset \Ip(\mu)$.
Also the converse inclusion holds, provided $(X,\sfd)$ is non-branching.

\begin{proposition}[Locality] \label{P:local} 
Let $(X,\sfd,\mu)$ be a non-branching m.m.s. that verifies the qualitative $\MCP$ condition. 
Then 
$$
\Ip_{loc}(\mu) = \Ip(\mu).
$$
\end{proposition}

\begin{proof}
Consider $z \in \Ip_{loc}(\mu)$, the corresponding local inversion plan $\pi^{z}$ and the corresponding open ball $B_{r}(z)$.
We will now construct a measure $\eta \in \M_{+}(X \times X)$ such that 
$$
(P_{1})_{\sharp} \eta = \mu\llcorner_{C_{r,2r}(z)}, \quad \eta (X \times X \setminus H(z)) = 0, \quad (P_{1})_{\sharp} \eta \ll \mu,
$$
where $C_{r,2r}(z) = B_{2r}(z) \setminus B_{r}(z)$. Once such a $\eta$ is obtained, one can repeat the same argument for $C_{nr,(n+1)r}(z)$ and 
obtain the measure $\eta^{n}$, for any $n \in \N$ with $n\geq 2$. Then 
$$
\hat \pi^{z} :  = \pi^{z} + \sum_{n\geq1} \eta^{n},
$$
will be an inversion plan.

Reasoning as in \eqref{eq:pfm(1)} we get
$$
\mu \llcorner_{B_{2r}(z)}\, = \, \int_{Q( B_{2r}(z) ) } \mu^{\alpha} \, q(d\alpha), \qquad 
\mu^{\alpha } \sim \mathcal{H}^{1}\llcorner_{\gamma^{\alpha}}, \quad q\textrm{-a.e.} \, \alpha \in Q( B_{2r}(z) ),
$$
where $q = Q_\sharp ( \mu \llcorner_{ B_{2r}(z)})$ is the quotient measure and $\gamma^{\alpha}$ 
is the maximal geodesic in $\Geo(z)$ containing $\alpha$. We can also assume each $\gamma^{\alpha}$ is parametrized with speed 1.
From the non-branching assumption, each $\gamma^{\alpha} \cap B_{r}(z)$ is sent by $\pi^{z}$ to a unique $\gamma^{\beta(\alpha)}$, that is  
$$
\pi^{z}\left( X \times X \setminus \{(x,y) \in B_{r}(z) \times X \colon x \in \gamma^{\alpha}, \, y \in \gamma^{\beta(\alpha)}\} \right) = 0,\qquad \beta : Q \to Q, 
$$
with $\beta$ $\mu$-measurable map. 
Moreover from Proposition \ref{P:manycurves},  for $q$-a.e. $\alpha \in Q$ 
$$
\mathcal{L}^{1} \left(\{ s \in [r,2r]  \colon \gamma^{\alpha}_{r} \notin \ee_{(0,1)}R(z)  \} \right) = 0.
$$
In particular for $q$-a.e. $\alpha \in Q$
$$
\sup\{ \tau \colon (\gamma^{\alpha}_{2r},\gamma^{\beta(\alpha)}_{\tau}) \in H(z) \} 
= \max\{ \tau \colon (\gamma^{\alpha}_{2r},\gamma^{\beta(\alpha)}_{\tau}) \in H(z) \}> 0.
$$
Denote the maximal $\tau$ with $\tau_{\alpha}$ and consider the map 
$$
T^{\alpha} : \{ \gamma^{\alpha}_{s} \colon s \in [r,2r]\}  \to   \{ \gamma^{\beta(\alpha)}_{s} \colon s \in [\tau_{\alpha}/2,\tau_{\alpha}] \},
\quad 
T^{\alpha} (\gamma^{\alpha}_{  (1-\ve) r +  \ve 2r   }) : = \gamma^{\beta}_{ (1-\ve) \frac{\tau_{\alpha}}{2} + \ve \tau_{\alpha}    }.
$$
Then the map $T : C_{r,2r}(z) \to X$, defined by $T(x) = T^{\alpha}(x)$ for $x \in \gamma^{\alpha}$, is $\mu$-measurable and satisfies
$$
(x,T(x)) \subset H(z), 
$$
for $\mu$-a.e. $x \in C_{r,2r}(z)$. Now define $\eta : = (id, T)_{\sharp} \mu\llcorner_{C_{r,2r}(z)}$ and observe that 
in order to have the claim it only remains to show that $(P_{2})_{\sharp} \eta \ll \mu$.

To this aim  consider a Borel set $A \subset X$ such that $\mu(A) = 0$. The claim is equivalent  to prove that $\mu(T^{-1}(A)) = 0$.
Consider the set $Q(A)$ and write it as the union of 
$$
Q_{1}(A) :  = \{ \alpha \in Q(A) \colon \mu^{\alpha}(A) > 0   \}, \quad Q_{2}(A) :  = \{ \alpha \in Q(A) \colon \mu^{\alpha}(A) = 0   \}. 
$$
For each $\alpha \in Q_{2}(A)$, since $\mu^{\alpha} \sim \H^{1}\llcorner_{\gamma^{\alpha}}$,
$$
\mu^{\beta^{-1}(\alpha)} \left( T^{\beta^{-1}(\alpha)} (A) \right) = 0.
$$
Moreover, since $\mu(A) = 0$, it follows that $q (Q_{1}(A)) = 0$. Then since $\beta$ was also associated with the local inversion plan, 
$q(\beta^{-1}(Q_{1}(A))) = 0$. Hence  $\mu(T^{-1}(A)) = 0$ and the claim follows.
\end{proof}

\begin{remark}
It seems to be a challenging problem to weaken the non-branching assumption in Proposition \ref{P:local} with the essentially non-branching one.  Regarding this issue, a relevant role is played by the following phenomenon. \\
Consider a point $z \in X$ of local inversion for $\mu$ and $\gamma \in \Geo(z)$, i.e. $\gamma_1=z$. Suppose that $\gamma$ can be extended through $z$ with a geodesic 
$\eta^{1}$ of length $\ve > 0$ with $\eta^{1}_{0} = z$, so that 
$$
\sfd(\gamma_{0},\eta_{1}^{1}) = \sfd(\gamma_{0},\gamma_{1}) + \sfd(\gamma_{1},\eta^{1}_{1}).
$$
Now consider the constant speed reparametrization $\hat \gamma$  of $\gamma$ restricted to $[1/2, 1]$. 
With no violation of essentially non-branching, it could happen that $\hat \gamma$ admits also a different extension through $z$, that is there exists 
a geodesic $\eta^{2} \neq \eta^{1}$ again such that $\eta^{2}_{0} =z$ and 
$$
\sfd(\hat \gamma_{0},\eta_{1}^{2}) = \sfd(\hat \gamma_{0},\hat \gamma_{1}) + \sfd(\hat \gamma_{1},\eta^{2}_{1}).
$$
Clearly also $\eta^{1}$ is a good extension of $\hat \gamma$. 

Since $z$ is a point of local inversion, we know that the local (multivalued) inverting map sends sets of positive $\mu$-measure to sets of positive $\mu$-measure. 
It may be the case that it sends $\hat \gamma_{0}$ to $\eta^{2}_{1}$ but  $\gamma_{0}$ is forced to be moved to a point belonging to $\eta^{1}$, if for istance 
$$
\sfd(\gamma_{0}, \eta^{2}_{1}) > \sfd(\gamma_{0},z) + \sfd(z,\eta^{2}_{1}).
$$
In this situation  no information can be obtained from the existence of the local inversion plan. 
\\Notice that the non-branching assumption indeed forces $\eta^{1}$ to coincide with $\eta^{2}$.
\end{remark}

\subsection{Stability}
Only for this section the reference measures will be assumed to be  a probability measure. \\
We consider a sequence of  pointed metric measures spaces $(X_{k}, \sfd_{k}, \mu_{k}, \bar{x}_k)$ with $\mu_{k}$  \emph{almost uniform} reference measure for $(X_{k},\sfd_{k})$.
We will show that if $(X_{k}, \sfd_{k}, \mu_{k}, \bar{x}_k)$ converges to $(X, \sfd, \mu, \bar{x})$ in the pointed measured Gromov-Hausdorff topology, then $\mu$ is an \emph{almost uniform} reference measure for $(X,\sfd)$. 

\begin{definition}\label{D:uniform}
Let $(X,\sfd,\mu)$ be a m.m.s. with $\mu\in  \mathcal{P}(X)$ reference measure for $(X,\sfd)$.
\begin{itemize}
\item We say that  $\mu$ is an \emph{almost uniform} reference measure provided  that  for $\mu$-a.e. $z \in X$ there exist a constant $C_z>0$ and 
an inversion plan $\pi^{z}$ satisfying
\begin{equation}\label{eq:defUni}
(P_{2})_{\sharp} \pi^{z} \leq C_z \; \mu.
\end{equation}
\item We say that  $\mu$ is a \emph{uniform} reference measure provided  that there exists  a constant $C>0$ such that for $\mu$-a.e. $z \in X$ there exists an inversion plan $\pi^{z}$ satisfying \eqref{eq:defUni} with $C_z \equiv C$.
\end{itemize}
\end{definition}

\noindent
%Recall that for $\mu$-a.e.  $z \in X$, the inversion plan $\pi^{z}$ has to be concentrated on the set 
%$$
%H(z) = \{ (x,y) \in X \times X : \sfd(x,y) = \sfd(x,z) + \sfd(z,y) \}.
%$$
Let us briefly recall the notion of pointed measured Gromov-Hausdorff convergence. A map $f : (X,\sfd_{X}) \to (Y,\sfd_{Y})$ between metric spaces 
is an $\varepsilon$-isometry if and only if: 
\begin{itemize}
\item[-] it almost preserves distances: for all $z,w \in X$,
\[
| \sfd_{X}(z,w) - \sfd_{Y}(f(z),f(w))| \leq \ve;
\]
\item[-] it is almost surjective: 
\[
\forall \, y \in Y, \ \  \exists \, x \in X \;\text{ such that} \quad \sfd_{Y}(f(x),y) \leq \ve.
\]
\end{itemize}
The following is a well-known notion of convergence that we will use for pointed m.m.s..

\begin{definition}\label{D:mGH}
Let $(X_{k}, \sfd_{k}, \mu_{k},\bar x_{k}), (X, \sfd, \mu, \bar x)$ be pointed m.m.s. with $\mu_{k} \in \mathcal{P}(X_{k})$, for all $k\in \N$, and $\mu \in \mathcal{P}(X)$. 
We say that  $(X_{k}, \sfd_{k}, \mu_{k},\bar x_{k}) \to (X, \sfd, \mu, \bar x)$ in the \emph{pointed measured Gromov-Hausdorff} topology,  ($pmGH$ for short)
provided there exist sequences $R_{k} \to + \infty$ and $\varepsilon_{k} \to 0$ and Borel maps $f_{k}: X_{k} \to X$ so that
\begin{itemize}
\item[$i)$] $f_{k}(\bar x_{k}) = \bar x$;  \medskip
\item[$ii)$] $\sup\left\{ |\sfd_{k}(x,y)   - \sfd(f_{k}(x), f_{k}(y) )| \, : \, x,y \in B_{R_{k}}(\bar x_{k}) \right\} \leq \varepsilon_{k}$; \medskip
\item[$iii)$]  the $\varepsilon_{k}$-neighborhood of $f_{k}\left( B_{R_{k}}(x_{k}) \right)$ contains $B_{R_{k}-\varepsilon_{k}}(\bar x)$; \medskip
\item[$iv)$] the sequence of probability measures $(f_{k})_{\sharp} \, \mu_{k}$ converges to $\mu$ narrowly; 
i.e  calling $C_b(X)$ the space of bounded continuous functions on $X$ it holds
$$
\lim_{k \to \infty} \int_{X} \varphi \, d \left[ (f_{k})_{\sharp} \, \mu_{k}\right] =  \int_{X} \varphi \, d \mu, \qquad \forall \varphi \in  C_b(X).
$$
\end{itemize}
%%
%%and $(K^{(l)})_{l \in \N}$
%%such that for each $l$, the space $K^{(l)}_{k}$, seen as a subspace of $X_{k}$, converges 
%%in the measured Gromov - Hausdorff topology to $K(l)$ as $k \to \infty$; and the union of all $K^{(l)}$ is dense in $X$.
%%
%%
%%For compact m.m.s. $\{(Y_{k}, \sfd_{k}, \mu_{k})\}_{k\in \N}$ and $(Y,\sfd,\mu)$ with $\mu_{k}$ and $\mu$ finite measure, 
%%we say that $(Y_{k}, \sfd_{k}, \mu_{k})$ converges in the measured Gromov-Hausdorff topology 
%%provided the existence of $\varepsilon_{k}$-isometries $f_{k}$ with $\varepsilon_{k} \to 0$
%%%
%%$$
%%f_{k} : Y_{k} \longrightarrow Y,  \qquad (f_{k})_{\sharp}\mu_{k} \weak \mu.
%%$$
%%%
\end{definition}

In Definition \ref{D:mGH} one can consider also sequences of pointed m.m.s. with measures not necessarily of total mass one. 
In the case of measures with possibly infinite total mass, i.e. Radon measures, one asks the weak convergence of $iv)$ to hold for any $f \in C_{b}(X)$
with bounded support. We can state the stability result. 
Recall that during this section we will always assume all the metric measure spaces to be geodesic, proper, complete and separable.

\begin{theorem}\label{T:stability}
Let $(X_{k}, \sfd_{k}, \mu_{k},\bar x_{k})$ be a sequence of pointed m.m.s.  which converges in  $pmGH$-sense to a limit  pointed m.m.s. $(X,\sfd,\mu, \bar x)$. 
Assume that, for infinitely $k$, $\mu_k$ is an almost uniform reference measure for $X_k$ and that for $\mu$-a.e. $z \in X$ there exists a sequence $\{z_k\}_{k \in \N}$   with  $z_k \in X_k$, such that  
\begin{equation}\label{eq:HpStability}
 \sfd(z, f_{k}(z_{k})) \to 0 \quad \text{and} \quad \liminf_{k \in \N} C_{z_k} < \infty. 
\end{equation} 
Then $\mu$ is an almost uniform  reference measure for $(X,\sfd)$.
\end{theorem}

\begin{proof}
Let $z \in \spt(\mu)$ be as in the assumption and consider a sequence of points $z_{k} \in X_{k}$ so that $\sfd(z, f_{k}(z_{k})) \to 0$ and, for infinintely many $k$'s,  $z_k$ are inversion points for $\mu_k$.
\\
Up to extracting a suitable subsequence, for each $k\in \N$ we can consider the inversion plan $\pi^{k} \in \mathcal{P}(X_{k}\times X_{k})$ around $z_{k}$; moreover, by assumption,  there exists $C=C_z$ such that 
$$
(P_{1})_{\sharp}\, \pi^{k} =  \mu_{k}, \qquad   (P_{2})_{\sharp}\, \pi^{k} \leq C \mu_{k}, \qquad \spt (\pi^{k}) \subset  H_{k}(z_{k}).
$$
%%
%It is then natural to consider the following symmetric inversion plan: 
%%
%$$
%\hat \pi^{k} : = \frac{1}{2} \left( \pi^{k} + (\iota)_{\sharp} \, \pi^{k} \right) \in \mathcal{M}^+(X \times X),
%$$
%%
%where $\iota(x,y) = (y,x)$. It follows that 
%%
%\begin{equation}\label{eq:Ppi}
%\left( P_{1} \right)_{\sharp} \hat \pi^{k} = \left( P_{2} \right)_{\sharp} \hat \pi^{k} =  \frac{1}{2}( 1 + h_{k} ) \, \mu_{k} \geq \frac{1}{2} \mu,
%\end{equation}
%%
%where $h_{k} \in L^{1}_{loc}(\mu_{k})$ is the density 
%of $(P_{2})_{\sharp} \pi^{k}$ with respect to $\mu_{k}$. In order to keep notation simple we denote $\hat \pi^{k}$ still with $\pi^{k}$.
%
%
It is natural to consider  the push-forward of the  inversion plan via the $\varepsilon_{k}$-isometries as follows: 
$$
\eta^{k} : = (f_{k},f_{k})_{\sharp}\, \pi^{k} \in \mathcal{P}(X \times X).  
$$
Since by assumption $(f_k)_\sharp \mu_k \to \mu$, it follows  that both $(P_1)_\sharp \eta_k= (f_k)_\sharp \mu_k$ 
and $(P_2)_\sharp \eta_k \leq C \cdot  (f_k)_\sharp \mu_k$ are tight. 
This in turn implies that $\eta_k \in  \mathcal{P}(X \times X)$ are tight and then, by Prokhorov Theorem, 
they converge narrowly to some $\eta \in \mathcal{P}(X \times X)$, up to subsequences:
$$
\int_{X\times X} \psi(x,y) \, d\eta_k \to  \int_{X\times X} \psi(x,y) \, d\eta, \quad \text{as } k \to \infty, \quad \forall \psi \in C_b(X\times X).
$$
In the rest of the proof we show that $\eta$ is an inversion plan for $\mu$ at $z$.
\\

{\bf Claim 1:} $(P_1)_\sharp \eta=\mu$ and $(P_2)_\sharp \eta \leq C \mu$. \\
First of all observe that  for $\varphi\in C_b(X)$ one has  $\varphi \circ P_i \in C_b(X\times X)$, for $i=1,2$. It follows that, for every  $\varphi\in C_b(X)$, it holds
\begin{eqnarray}
	\int_X \varphi \, d \mu &=&  \lim_{k \to \infty} \int_X \varphi \, d \left[(f_k)_\sharp \mu_k \right] 
 							= \lim_{k \to \infty}   \int_{X} \varphi  \; d \left[(f_k)_\sharp ( {P_1}_\sharp  \pi_k) \right]  \nonumber \\
                    	                	&=& \lim_{k \to \infty}   \int_{X\times X} (\varphi\circ P_1) \, d \eta_k 
                                    			=  \int_{X\times X} (\varphi\circ P_1) \, d \eta 
							=  \int_X \varphi \, d \left[(P_1)_\sharp \eta\right]. \nonumber 
\end{eqnarray}
Since $\varphi\in C_b(X)$ is arbitrary, we infer that $(P_1)_\sharp \eta=\mu$.  By analogous computations, using this time that $(P_2)_\sharp (\pi_k) \leq C \mu_k$, we get that,  for every $\varphi\in C_b(X)$ with $\varphi\geq 0$, it holds
$$ 
\int_X \varphi \, d \left[(P_2)_\sharp \eta\right]
 	=\lim_{k \to \infty} \int_{X\times X} (\varphi \circ P_2)\, d\eta_k 
 	= \lim_{k \to \infty} \int_X \varphi \, d\left[(f_k)_{\sharp}({P_2}_\sharp \pi_k) \right] 
 	\leq C \int_X \varphi \, d \mu. 
$$
Since $\varphi\in C_b(X)$ with $\varphi\geq 0$  is arbitrary, we infer that $(P_2)_\sharp \eta \leq C \mu$ and thus the proof of claim 1 is complete.
\\
 
{\bf Claim 2:}  $\spt (\eta) \subset H(z)$.  
\\Fix a point  $y \in X \times X$ and an  increasing sequence of real numbers $R_{i} \to \infty$ so that 
$$
\eta \left( \partial B_{R_{i}}(y) \right) = 0,  \qquad \forall \, i,k\in \N,
$$
where the ball is taken in $X \times X$. Thus for each $i\in \N$ it holds 
$$
\lim_{k\to \infty} \eta^{k} \left(  B_{R_{i}}(y) \right) = \eta( \left(  B_{R_{i}}(y) \right) ) > 0 .
$$
Consider then, for each $i \in \N$,  the  following  probability measures: 
$$
\xi^{k} : = \left( \eta^{k}(B_{R_{i} }(y)) \right)^{-1} \eta^{k}\llcorner_{B_{R_{i} }(y) }, \qquad \xi : = \left( \eta(B_{R_{i} }(y)) \right)^{-1} \eta\llcorner_{B_{R_{i} }(y) },
$$
and observe that $\xi^{k} \weak \xi$ narrowly. For ease of notation we have ignored the dependence on $i$.
Thanks to the inner regularity of probability measures, for each $k \in \N$ there exists a compact set $C_{k}\subset B_{R_{i}}(y) $ so that 
$$
\xi^{k}(C_{k}) \geq 1 - \frac{1}{k}, \qquad C_{k} \subset (f_{k},f_{k})(H_{k}(z_{k})) \cap B_{R_{i}}(y).
$$
Up to subsequences, there exists $C \subset \bar B_{R_{i}}(y) $ so that $C_{k} \to C$ in the Hausdorff distance. 
Therefore for any $\delta > 0$ there exists $k_{\delta} \in \N$ so that $C_{k} \subset C^{\delta}$, for all $k \geq k_{\delta}$
where $C^{\delta}$ denotes the closed tubular neighborhood of $C$ of radius $\delta$. 
Using upper-semicontinuity over compact sets of weakly converging measures, we infer
$$
\xi(C^{\delta}) \geq \limsup_{k\to \infty} \xi^{k}(C^{\delta}) \geq \limsup_{k \to \infty} \xi^{k}(C_{k}) = 1.
$$
Sending $\delta \to 0$, one obtains $\xi(C) = 1$.  It remains to show that $C \subset H(z)$.
As $C_{k} \to C$, for each $(x,y) \in C$ there exists a sequence $\{ (x_{k},y_{k}) \}_{ k \in \N}$ with 
$$
(x_{k},y_{k}) \in C_{k} \subset (f_{k},f_{k})(H_{k}(z_{k})) \cap B_{R_{i}}(y), \qquad (x_{k},y_{k}) \to (x,y). 
$$
Hence $(x_{k},y_{k}) = (f_{k}(u_{k}), f_{k}(w_{k}))$, with $(u_{k}, w_{k}) \in H_{k}(z_{k})$.
We can deduce from $ii)$ and $iii)$ of Definition \ref{D:mGH} that
\begin{align*}
\sfd(x,y) - \sfd(x,z) - \sfd(z,y)  = &~ \lim_{k \to \infty} \sfd(x_{k},y_{k}) - \sfd(x_{k},f_{k}(z_{k})) - \sfd(f_{k}(z_{k}),y_{k}) \crcr
= &~ \lim_{k \to \infty}  \sfd_{k}(u_{k},w_{k}) - \sfd_{k}(u_{k},z_{k}) - \sfd_{k}(z_{k},w_{k})  \crcr
= &~  0. 
\end{align*}
It follows that $(x,y) \in H(z)$ and therefore $\xi(H(z)) = 1$. Hence $\spt (\eta)  \cap B_{R_{i}}(y) \subset H(z)$.
\\ Letting $R_{i} \to \infty$ yields $\spt (\eta) \subset H(z)$, as desired.
\end{proof}

The first part of the next corollary follows directly from the statement of  Theorem \ref{T:stability}; the second claim is instead a consequence of the proof above.

\begin{corollary}\label{cor:Stab}
\begin{itemize}
\item Let $(X_{k}, \sfd_{k}, \mu_{k},\bar x_{k})$ be a sequence of pointed m.m.s. with $\mu_{k} \in \mathcal{P}(X_{k})$  uniform 
reference measure for $(X_{k},\sfd_{k})$ with the  constant in \eqref{eq:defUni} uniform in $k\in \N$. Assume the existence of a pointed m.m.s. $(X,\sfd,\mu, \bar x)$ so that 
$(X_{k}, \sfd_{k}, \mu_{k}, \bar x_{k}) \to (X, \sfd, \mu, \bar x)$ in the $pmGH$-sense, where $\mu \in  \mathcal{P}(X)$.
Then $\mu$ is a uniform  reference measure for $(X,\sfd)$. 

\item If $\mu$ is a \emph{uniform} reference measure for $(X,\sfd)$ 
then \emph{every} point $z \in \spt(\mu)$ is an inversion point.
\end{itemize}
\end{corollary}

\begin{remark}
It is a challenging  open problem   to replace in Theorem  \ref{T:stability} the assumption that  $\mu_k$ are \emph {almost uniform reference measures} satisfying \eqref{eq:HpStability} by a \emph{uniform lower bound on the Ricci curvature} of $(X_k,\sfd_k,\mu_k)$ (in MCP or CD or RCD sense).  This would have  as remarkable consequence that the limit measure on a Ricci limit space is unique up to multiplication by  an $L^1$ function (note that this last fact was already established by Cheeger-Colding \cite{cc:IIIJDE2000} via a completely  different argument).  
\end{remark}

\subsection{Uniqueness of the limit}\label{Ss:unique}
 
In order to state and prove the uniqueness result, let us introduce a last small piece of notation: a sequence of pointed m.m.s.  
$(X_{k}, \sfd_{k}, \mm_{k}, \bar{x}_k)$ is said to \emph{uniformly} satisfy the 
{qualitative $\MCP$ condition} if there exist a sequence of radii  $R_j\uparrow +\infty$  and  continuous functions 
$f_j:[0,1]\to \R^+$ with $f_j(0)=1$,  $f_j(1)=0$ and $f_j(t)>0$ for any $t \in [0,1)$, $j \in \N$, 
such that the qualitative $\MCP$ condition \eqref{eq:QualMCP} 
holds for  every  $(X_{k}, \sfd_{k}, \mm_{k}, \bar{x}_k)$, $k \in \N$, with these fixed sequence  $\{R_j, f_j\}_{j \in \N}$ .

\begin{theorem}\label{T:unique}
Let $(X_{k}, \sfd_{k}, \mu_{k}, \bar{x}_k)$ be a sequence of essentially non-branching pointed  m.m.s. that uniformly  verify the qualitative $\MCP$ condition in the above sense. 
Assume also that, for every $k \in \N$,  $\mu_{k} \in \mathcal{P}(X_{k})$ is a uniform reference measure for $(X_{k}, \sfd_{k})$, 
with the constant in \eqref{eq:defUni} uniform in $k\in \N$.

Let $(X, \sfd, \mu, \bar x)$ and $(X,\sfd, \xi, \bar y)$ be two limit points in the pointed measured Gromov Hausdorff convergence of 
the two sequences $(X_{k}, \sfd_{k}, \mu_{k}, \bar x_{k})$ and $(X_{k}, \sfd_{k}, \mu_{k}, \bar y_{k})$, respectively.
If $(X,\sfd, \mu)$ and $(X,\sfd,\xi)$ are both essentially non-branching and
$$
\spt(\xi)  = \spt(\mu),
$$
then $\mu \sim \xi$, i.e. $\mu \ll \xi$ and $\xi \ll \mu$.
\end{theorem}

\begin{proof}
The proof of the stability of the uniform qualitative $\MCP$  condition is completely analogous to the proof of the stability of $\MCP(K,N)$, see \cite[Theorem 6.8]{Ohta07}. It follows that 
both $(X, \sfd, \mu)$ and $(X,\sfd, \xi)$ verify the qualitative $\MCP$ and both are proper. Hence if $\Omega =\spt (\xi) = \spt(\mu)$, 
then by Corollary \ref{cor:Stab}  both $\mu$ and $\xi$ are  uniform reference measures for $(\Omega, \sfd)$, in particular $\mathsf{Ip}(\mu)=\mathsf{Ip}(\xi)=\Omega$.  Then by Theorem \ref{T:metric} it follows that $\xi \ll \mu$ and $\mu \ll \xi$, i.e. $\xi \sim \mu$. 
\end{proof}

%%%%%%%%%%%%%%%%%%%%%%%%%%%%%%%%%%%%%%%%%%%%%%%%%%%%%%%
%%%%%%%%%%%%%%%%%%%%%%%%%%%%%%%%%%%%%%%%%%%%%%%%%%%%%%%
%%%%%%%%%%%%%%%%%%%%%%%%%%%%%%%%%%%%%%%%%%%%%%%%%%%%%%%
%%%%%%%%%%%%%%%%%%%%%%%%%%%%%%%%%%%%%%%%%%%%%%%%%%%%%%%
%%%%%%%%%%%%%%%%%%%%%%%%%%%%%%%%%%%%%%%%%%%%%%%%%%%%%%%
%%%%%%%%%%%%%%%%%%%%%%%%%%%%%%%%%%%%%%%%%%%%%%%%%%%%%%%
%%%%%%%%%%%%%%%%%%%%%%%%%%%%%%%%%%%%%%%%%%%%%%%%%%%%%%%
%%%%%%%%%%%%%%%%%%%%%%%%%%%%%%%%%%%%%%%%%%%%%%%%%%%%%%%
%%%%%%%%%%%%%%%%%%%%%%%%%%%%%%%%%%%%%%%%%%%%%%%%%%%%%%%
%%%%%%%%%%%%%%%%%%%%%%%%%%%%%%%%%%%%%%%%%%%%%%%%%%%%%%%
%%%%%%%%%%%%%%%%%%%%%%%%%%%%%%%%%%%%%%%%%%%%%%%%%%%%%%%
%%%%%%%%%%%%%%%%%%%%%%%%%%%%%%%%%%%%%%%%%%%%%%%%%%%%%%%

\section{Inversion over Alexandrov spaces with curvature bounded from below}\label{S:InvAlex}

The goal of this section is to show that, given a compact  $n$-dimensional Alexandrov space $(X,\sfd)$, the $n$-dimensional Hausdorff measure $\H^n$ is a reference measure. The idea is to  combine the Lipschitz regularity of the exponential map together with the fact that it maps sets of positive measure to sets of positive measure (see Lemma \ref{L:misurapositiva}), and then 
use  Disintegration Theorem  in order to  construct a ``local'' inversion plan.  Suitably iterating this construction one can obtain a ``global'' inversion plan and prove Theorem \ref{T:Alexandrov-general}.

% In order to achieve this, we will exploit the local regularity of $(X,\sfd)$ near regular points (namely that every regular point has a neighborhood  bi-Lipschitz equivalent to an open subset of $\R^n$, see  \cite[Theorem 10.8.4]{BBI}) and a remarkable result by David  \cite{David88,David2014} stating that a Lipschitz map, whose image has positive measure, is actually bi-Lipschitz on a suitable subset. 

We first recall few definitions and notations that will be needed only during this section. 
\\Let $(X,\sfd)$ be an Alexandrov space with curvature bounded below by $k \in \R$, that is $(X, \sfd)$ is a complete, locally compact, geodesic space of curvature $\geq - k$
and of Hausdorff dimension $n < \infty$.

%%%%%%%%%%%%%%%%%%%%%%%%%%%%%%%%%%%%%%%%%%%%%%%%%%%%%%%%%%
%%%%%%%%%%%------------------Properties of the Exponential map--------%%%%%%%%%%%%%%%%%%%%

\subsection{Exponential map} 
Fix any $p \in X$ and consider the set of segments emanating from $p$
$$
\tilde W_{p} : = \left\{ pq : q \in X \right\},
$$
where $pq : [0,1] \to X$ with $pq \in \Geo$, $pq_0=p$ and $pq_{1} = q$. Denote then with $\tilde \Sigma_{p}$ the set of equivalence classes of minimal segments emanating from $p$
where $pq$ is equivalent to $pr$ if and only if one of $pq$ and $pr$ is contained in the other. The space $\tilde \Sigma_{p}$ has distance naturally 
induced by the angle $\angle$ between minimal segments from $p$ (for the definition of $\angle$ see for instance  \cite{BBI} or \cite{OS94}) . 
The completion of $\tilde \Sigma_{p}$ is the \emph{space of directions} and is denoted with $\Sigma_{p}$. For any minimal segment $pq$ in $X$, 
the symbol $v_{pq}$ denotes the direction at $p$ corresponding to $pq$.
Then for any $pq$ and $pr$, it holds 
$$
\sfd_{\Sigma_{p}}(v_{pq}, v_{pr}) : = \angle qpr. 
$$
For ease of notation we also denote $\sfd_{\Sigma_{p}}$ with $\angle$. We say that  $p$ is a \emph{regular point} if $(\Sigma_{p}, \sfd_{\Sigma_{p}})$ is isometric to the standard $(n-1)$ sphere, $S^{n-1}$.

The \emph{tangent cone} $K_{p}$ is obtained from $[0,\infty) \times \Sigma_{p}$ by identifying together all elements of the form $(0, u_{0})$, with $u_{0} \in \Sigma_{p}$. 
Its elements are denoted by $tu_{0}$ for $t \geq 0$ and $u_{0} \in \Sigma_{p}$. The natural cone distance on $K_{p}$ is given by 
$$
\sfd_{K_{p}}(t u_{0}, sv_{0}) : = \sqrt{ t^{2} + s^{2} - 2st \cos \angle u_{0} v_{0} }.
$$
Equivalently, a point $p \in X$ is  regular if and only if $K_{p}$ is isometric to $\R^{n}$. By considering $\tilde W_{p} \subset K_{p}$ 
one can define the exponential map at $p$ as follows
$$
\textrm{Exp}_{p} : \tilde W_{p} \to X, \qquad \Exp_{p}( |pq| v_{pq}) : = q.
$$
By Toponogov convexity it follows that $\Exp$ is Lipschitz on balls: there exists $L=L(k)>0$ such that  
\begin{equation}\label{eq:ExpLip}
\sfd(\Exp_{p}( v ), \Exp_{p}(u) ) \leq L \; \sfd_{K_{p} } (v,u), \quad \forall u,v \in B_{1}^{K_p}(0),
\end{equation}
with $L = 1$ when $k = 0$.
Denote now with $D(\Exp_{p})$ the identification of $\tilde W_{p}$ in $K_{p} \equiv \R^{n}$.

\begin{lemma}\label{L:domainexp}
Let $p\in X$ be a regular point. Then  
\begin{equation}\label{E:DExp}
\mathcal{L}^{n} \left( B_{1}^{K_{p}}(0) \setminus D( \Exp_{p} ) \right) = 0.
\end{equation}
\end{lemma}

\begin{proof}
Consider the cut locus at $p$, $C(p)$, see \eqref{eq:defCutLocus} for its definition. One can prove that $\H^{n}(C(p)) = 0$.
Denote then with $E : = X \setminus C(p)$ and define 
$$
\log_{p} : E \to D(\Exp_{p})\, \subset\, K_{p} , 
$$
so that $\Exp_{p} \circ \log_{p} \llcorner_{E} = \id\llcorner_{E}$. As geodesics do not branch, the definition is well-posed.
Moreover,  recalling \eqref{eq:ExpLip},
 for any $x,y \in E$ it trivially holds
\[
\sfd(x,y) \leq L \, \sfd_{K_{p}}(\log_{p}(x), \log_{p}(y)),
\]
with $L$ depending on $k$ and $n$, such that $L \to 1$ as $k \to 0$.
%\footnote{AM: ho cancellato  il claim ''Note moreover that $\log_{p}$ is surjective''. Direi che non si usa, infatti per concludere \'e sufficiente osservare che, 
%per definizione, $\log_{p}(E)\subset D(\Exp_{p})$. L'ho tolto anche perch\'e non mi \'e affatto chiaro: 
%in $D(\Exp_{p})$ ci sono anche ''i punti finali'' delle geodetiche, mentre in $E$ no, visto che stiamo togliendo il cut locus.}
%
For any $r \in (0,1)$ consider $X_{r}$, the Alexandrov space $X$ endowed with the rescaled distance $r^{-1} \sfd$;  note that $B^{X_{r}}_{1}(p) : =  B_{r}(p)$ and that 
\[
\H^{n}_{X_{r}} = \frac{1}{r^{n}} \H^{n}, \qquad curv(X_{r}) \geq r^{2} k.
\]
Then we have: 
%\footnote{AM: nella stima con la constante di Lipschitz direi che $L_r$ va sostituita con $L_r^n$, ho aggiunto anche un piccolo step intermedio nella catena di diseguaglianze }
\begin{align*}
\frac{1}{r^{n}} \H^{n}(B_{r}(p)) =  \frac{1}{r^{n}} \H^{n}( E \cap B_{r}(p))  &~  = ~\frac{1}{r^{n}} \H^{n}(E \cap B^{X_{r}}_{1}(p))  \crcr
&~ = ~\H^{n}_{X_{r}}(E\cap B^{X_{r}}_{1}(p))  \crcr
&~ \leq L^n_r \cdot \mathcal{L}^{n}( \log_p(E) \cap B_{1}(0)  ) \crcr
&~ \leq ~ L_{r}^n \cdot \mathcal{L}^{n}(D(\Exp_{p}) \cap B_{1}(0)  ). 
\end{align*}
Since $p$ is a regular point, $ r^{-n} \H^{n}(B_{r}(p)) \to \omega_{n}$, where $\omega_{n}$ is the volume of the $n$-dimensional Euclidean unit ball. Since, as $r\to 0$ we have that  $curv(X_{r})\to 0$, we get that  $L^{n}_{r} \to 1$ and  
the claim follows.
\end{proof}

We also need the following crucial property of the Exponential map. 
Since we don't have a reference for it, we include our proof. 

\begin{lemma}\label{L:misurapositiva}
Let $(X,\sfd)$ be an Alexandrov space with curvature bounded from below and dimension $n$. Fix $p \in X$ a regular point. 
Then for any $A \subset K_{p}$ of positive $n$-dimensional Lebesgue measure, it holds
$$
\H^{n} \left( \Exp_{p}(A) \right) > 0.
$$
\end{lemma}

\begin{proof}
We present here only the proof in the case of non-negative curvature. 

{\bf Step 1.} Suppose by contradiction the claim is false and pick $A \subset K_{p}$ contradicting the claim.
For ease of notation we identify $K_{p}$ with $\R^{n}$ and for any $v \in S^{n-1}$ we denote the half line in the direction $v$, 
that is the set $\{ tv : t \geq 0 \}$, with $\Span_{+}\{v\}$.
Without loss of generality we can also assume that 
\[
\mathcal{L}^{1} \left( A \cap \Span_{+}\{v \} \right) > 0, 
\]
for all those $v \in S^{n-1}$ so that $A \cap \Span_{+}\{v \} \neq \emptyset$.

Observing that on each half line $\Span_{+}\{v \}$ the map $\Exp_{p}$ is an isometry, and from Disintegration Theorem 
$$
0 = \H^{n} \left(  \Exp_{p}\left( A \right)  \right)=  \int_{Q}  \left(h_{\alpha } \, \H^{1}\right) \left( \Exp_{p}\left( A \right) \cap \gamma^{\alpha} \right) q(d\alpha), 
$$
where $q$ is the quotient measure of $\H^{n}$ associated to the rays decomposition and 
$h_{\alpha}\, \H^{1}$ the conditional measures associated to the disintegration, 
we obtain that 
$$
q \left( \left\{ \alpha \in Q : \Exp_{p}(A) \cap \gamma^{\alpha} \neq \emptyset  \right\} \right) = 0.
$$
This implies that we can consider the following set: 
\[
\Span(A) : = \{ t \cdot v : t  \in [0,1], \,v \in S^{n-1} \, s.t. \ \exists \,s>0,\, s \cdot v \in A \} \subset \R^{n},
\]
and still obtain that 
$$
\mathcal{L}^{n}\left( \Span(A) \right) > 0, \qquad \H^{n}\left( \Exp_{p}(\Span(A))\right) = 0.
$$

{\bf Step 2.} We now obtain a contradiction using the non-expanding property of $\Exp_{p}$.
Since $p$ is a regular point  the volume of the $B^{X}_{r}(p)$ has the same infinitesimal behavior of $r^{n}$.  So we have 
\begin{align*}
1 & ~ = \lim_{r\to 0} \frac{\H^{n} (B^{X}_{r}(p)) }{\omega_{n } r^{n}}  = \lim_{r\to 0} \frac{\H^{n} \left(B^{X}_{r}(p)  \setminus \Exp_{p}\left( \Span(A) \right)   \right) }{\omega_{n } r^{n}} \crcr
&~ \leq \lim_{r\to 0} \frac{\mathcal{L}^{n}\left( B_{r}(0) \setminus \Span(A) \right)}{\omega_{n } r^{n}},
\end{align*}
where for the last inequality we used Lemma \ref{L:domainexp}, the non-exapanding of $\Exp_{p}$ and $\omega_{n}$ is the volume of the unitary ball in $\R^{n}$.
It follows that $0 \in \R^{n}$ is a point of density $0$ for $\Span(A)$. From the definition of $\Span(A)$ we conclude that this is possible only if 
\[
\H^{n-1} \left( v \in S^{n-1} :  s\cdot v \in A,\ s > 0 \right) = 0.
\]
Since the last identity implies that $\mathcal{L}^{n}(A)=0$ we have a contradiction and the claim is proved.
\end{proof}

\bigskip

\subsection{Inversion plan}
We now show that $\H^{n}$ is a reference measure for an $n$-dimensional  compact Alexandrov space $(X,\sfd)$. More precisely we will construct an inversion plan 
at $\H^n$-a.e. regular point $p \in X$.

\begin{theorem}\label{T:Alexandrov-general}
Let $(X,\sfd)$ be a compact  Alexandrov space with curvature bounded from below by $k$ and Hausdorff dimension $n \in \N$.
Then $\H^{n}$ is a reference measure for $(X, \sfd)$. More precisely we show that  an inversion plan exists at all regular point $p \in X$ so that \eqref{E:manycurves} holds.
\end{theorem}

\begin{proof}
Since $(X, \sfd)$ is compact and it is well known that $(X,\sfd, \H^{n})$ is non-branching and  verifies $\MCP(k,n)$.
We will construct an inversion plan for any regular point $p$ such that \eqref{eq:Rp} holds true.

{\bf Step 1.}
From Proposition \ref{P:manycurves} it follows that for $\H^{n}$-a.e. $p \in X$
\begin{equation}\label{eq:Rp}
\H^{n} \left( {\rm e}_{(0,1)}(R(p) ) \right) = \H^{n}(X),
\end{equation}
where $R(p) = \{ \gamma \in \Geo : \gamma_{t} = p, \text{ for some } t \in (0,1) \}$. 
So we fix once and for all a regular point $p \in X$ so that \eqref{eq:Rp} holds true.
For ease of notation we denote with $R_{p}$ the set ${\rm e}_{(0,1)}(R(p))$. 
We perform the disintegration of $\H^{n}\llcorner_{R_{p}}$ as in Proposition \ref{P:disintegration}:
$$
\H^{n} \llcorner_{R_{p} } = \int_{S} \eta^{\alpha} \, q(d\alpha), \quad \eta^{\alpha}(\gamma^{\alpha} ) = 1, \quad q-a.e. \, \alpha \in S,
$$
where $\gamma^{\alpha} \in \Geo(p)$ is the unique geodesic in $\Geo(p)$ through $\alpha$, $S = Q(R_{p})$ is the quotient set and $Q$ the quotient map.
Observe that convexity of $R_{p}$ is not needed to apply Disintegration theorem, we only use that for each point in $z \in R_{p}$, any intermediate point between $p$ and $z$ 
belongs to $R_{p}$. The same property permits to use $\MCP$ to obtain the absolute continuity of the conditional measures with respect to the Hausdorff measure of dimension 1.

We can write $R_{p}$ in the following way: 
$$
R_{p} = \bigcup_{m = 2}^{\infty} {\rm e}_{(0,1)}\left( \left\{ \gamma \in \Geo : \gamma_{t} = p,  \, t \leq 1 - 1/m \right\} \right) = \bigcup_{m = 2}^{\infty} {\rm e}_{(0,1)} \left( R_{m}(p) \right) 
= \bigcup_{m = 2}^{\infty} R_{p,m}.
$$
Note that  if $\gamma \in \Geo(p)$ is so that $\gamma_{0} \in R_{p,m}$, then $\gamma_{s} \in R_{p,m}$ for all $s \in [0,1)$.
Indeed $\gamma_{0} \in R_{p,m}$ implies the existence of $\xi \in \Geo$ so that 
$$
\xi_{0} = \gamma_{0}, \quad \xi_{t} = p.
$$
with $t \leq 1- 1/m$. Then for any $\bar s \in (0,1)$ there exists $\tau < t$ so that $\xi_{\tau} = \gamma_{\bar s}$. Consider $\sigma(\tau) = t + (1-t)(t-\tau)t^{-1}$ and 
define the new curve
$$
\xi^{\tau}_{s} : = \xi_{\tau(1-s) + \sigma(\tau) s}.
$$
Being a subset of $\xi$, it follows that  $\xi^{\tau} \in \Geo$, moreover $\xi^{\tau}_{0} = \gamma_{\bar s}$ and $\xi^{\tau}_{t} = p$. Hence $\gamma_{\bar s } \in R_{p,m}$.
It follows that for each $\alpha \in S$ either $\gamma^\alpha_{s} \notin R_{p,m}$ for all $s \in [0,1)$ (but this will never happen because of the definition of $R_{p}$) or there exists a minimal $s(\alpha)\in [0,1)$ so that 
$$
\gamma^{\alpha}_{s} \in R_{p,m}, \, \forall\, s \in [s(\alpha), 1), \qquad \gamma^{\alpha}_{s} \notin R_{p,m}, \, \forall s < s(\alpha).
$$
As $R_{m,p}$ can be written as countable union of closed sets, $\alpha \mapsto s(\alpha)$ is measurable. 
%\footnote{AM: un cenno di argomento  tipo ''by a standard selection principle''?}. 
Therefore we have the following representation: 
$$
\H^{n} \llcorner_{R_{p,m}} = \int_{S} \eta_{m}^{\alpha} \, q_{m}(d\alpha), \quad \eta_{m}^{\alpha}(\gamma^{\alpha,m} ) = 1, \quad q_{m}\text{-a.e. } \, \alpha \in S,
$$
where $\gamma^{\alpha,m}$ is the reparametrization of $\gamma^{\alpha}$ such that $\gamma^{\alpha,m}_{0} = \gamma^{\alpha}_{s(\alpha)}$ and 
$\gamma^{\alpha,m}_{1} = p$.
To keep notation simple, we will omit the dependence  on $m$ of the geodesic, so $\gamma^{\alpha,m} = \gamma^{\alpha}$.

{\bf Step 2.}
From the definition of $R_{p,m}$ and the non-branching property of $(X,\sfd)$, for $q_m$-a.e.  $\alpha \in S$ there exists a unique $\beta \in S$, 
that we will denote with $\beta(\alpha)$, such that 
\begin{equation}\label{E:invertinH}
\{ (\gamma_{s}^{\alpha}, \gamma^{\beta}_{t}) \in X\times X : s \in [0,1], t \in I(\alpha) \} \subset H(p), 
\end{equation}
with $I(\alpha) : = \left\{ s\in [0,1] :  s \geq 1 - L(\gamma^{\alpha}) \left( (m-1)L(\gamma^{\beta(\alpha)}) \right)^{-1} \right\}$, 
where $L$ denotes the length of the geodesic.
Indeed notice that for $q_m$-a.e.  $\alpha \in S$ there exists, as before, $\xi \in \Geo$ so that 
$$
\xi_{0} = \gamma^{\alpha}_{0}, \quad \xi_{t} = p, \quad \xi_{1} \in \gamma^{\beta(\alpha)}_{(0,1)}, \quad \sfd(p,\xi_{1}) \geq \frac{1}{m-1} L(\gamma^{\alpha})
$$
and moreover $\{ (\xi_{s},\xi_{\tau}) : s\leq t \leq \tau \} \subset H(p)$. 
From the bound $t \leq 1-1/m$ it is immediate to check that if $s \in I(\alpha)$ then
$$
\gamma^{\beta(\alpha)}_{s } \in \xi_{(t,1)} 
$$
and therefore \eqref{E:invertinH} is proved. For the next step we find convenient the following abuse of notation: 
we denote with $\gamma^{\beta(\alpha)}$ the reparametrization of the whole geodesic $\gamma^{\beta(\alpha)} \in \Geo(p)$ to the interval $I(\alpha)$.
\bigskip

{\bf Step 3.} We construct the inversion plan. First of all observe that  $R_{p,1}= \emptyset$; for any $m \geq 2$ and 
%Possibly taking a subset of $R_{p}$, still of full measure, we can assume that for all $\alpha \in S$, $\eta^{\alpha}(\gamma^{\alpha}) = 1$. 
for $q_m$-a.e.  $\alpha \in S$, we consider the coupling $\pi_m\in \mathcal{P}(X \times X)$ defined by
$$
\pi_m : = (\id, F_{m})_{\sharp} \H^{n}\llcorner_{(R_{p,m} \setminus \bigcup_{l=1}^{m-1} R_{p,l})},
$$
where
$$
F_m : R_{p,m} \longmapsto R_{p,m},  \qquad F_m(\gamma^{\alpha}_{t}) : = \gamma^{\beta(\alpha)}_{1-t}.
$$
%
%By definition
%%
%$$
%(P_{1})_{\sharp} \pi^{\alpha}_m = \eta^{\alpha}_m , \qquad (P_{2})_{\sharp}\pi^{\alpha}_m 
%= \frac{1}{c(\alpha)} \H^{1}\llcorner_{\gamma^{\beta(\alpha)} \cap U },
%$$
%
%where $c(\alpha)$ is the renormalization constant, and $\pi^{\alpha}_m$ is concentrated 
%on the set $\{ (\gamma^{\alpha}_{s}, \gamma^{\beta(\alpha)}_{s}) : s\in [0,1]  \}$. Then define 
%%
%$$
%\pi_m : = \int_{S} \pi^{\alpha}_m \; q_{m}(d\alpha). 
%$$
%%
%In order to show that the definition of $\pi_m$ is well-posed we need to prove that $\H^{1}(\gamma^{\beta(\alpha)} \cap U) > 0$ for $q$-a.e. $\alpha \in S$
%but this is granted by the fact that $U$ has full measure. 
From Step 2, $\pi_m$ is well defined and by construction we have 
$$
\pi_m(X \times X \setminus H(p)) =0 \quad \textrm{and} \quad (P_{1})_{\sharp} \pi_m = \H^{n}\llcorner_{(R_{p,m} \setminus \bigcup_{l=1}^{m-1} R_{p,l})}.
$$
The fact that  $(P_{2})_{\sharp} \pi_{m} \ll \H^{n}$ follows easily from Lemma \ref{L:misurapositiva}: indeed if  by contradiction it was false, then there would exist a subset 
$A \subset R_{p,m}$ so that 
$$
(P_{2})_{\sharp} \pi_m (A) > 0, \qquad \H^{n}(A) = 0, 
$$
that can be restated as $\H^{n}(F_{m}^{-1}(A)) > 0$ and $\H^{n}(A) = 0$. 
But since $\Exp_{p}$ is locally Lipschitz we have on one hand $\mathcal{L}^{n}\left(\log_{p}(F_{m}^{-1}(A))\right) > 0$ and on the other hand, from Lemma \ref{L:misurapositiva}, 
$\mathcal{L}^{n}\left(\log_{p}(A)\right) = 0$. 
Now we can map $\log_{p}(A)$ to $\log_{p}(F_{m}^{-1}(A))$ using the inversion around the origin composed with a rescaling on each line, therefore either both have  
Lebesgue measure $0$ or Lebesgue measure strictly positive. Hence we have a contradiction and $(P_{2})_{\sharp} \pi_{m} \ll \H^{n}$.

The desired inversion plan is then given by summation:
$$\pi:=\sum_{m=2}^\infty \pi_m. $$
Indeed, since $R_p$ has full measure by construction,  we have
$$(P_1)_\sharp(\pi)=\sum_{m=2}^\infty  \H^{n}\llcorner_{(R_{p,m} \setminus \bigcup_{l=1}^{m-1} R_{p,l})}= \H^{n}\llcorner_{\bigcup_{m=2}^\infty R_{p,m}} =  \H^{n}\llcorner_{R_p}=  \H^{n};$$
it is also clear that $(P_2)_\sharp(\pi) \ll \H^n$ and $\pi(X\times X \setminus H(p))=\sum_{m=2}^{\infty} \pi_m(X\times X \setminus H(p))=0.$

\end{proof}

The following is a straightforward consequence of Theorem \ref{T:metric} and Theorem \ref{T:Alexandrov-general}.

\begin{corollary}
Let $(X,\sfd)$ be a compact Alexandrov space with curvature bounded from below by $k$ and with Hausdorff dimension $n \in \N$. 
Let $\mm$ be a Borel finite measure on $X$ so that $(X,\sfd, \mm)$ verifies the qualitative $\MCP$ condition \eqref{eq:QualMCP}.
\noindent
If $\mm$ is concentrated on those regular points of $X$ such that $\H^{n}$ verifies \eqref{E:manycurves}, then 
$$
\mm \ll \H^{n}.
$$
\end{corollary}

%%%%%%%%%%%%%%%%%%%%%%%%%%%%%%%%%%%%%%%%%%%%%%%%%%%%%%%%%%%%%

\section{Examples and applications}\label{S:example}

In this final section, we give ''smooth'' examples where the inversion plan exists at every point; as a consequence, the application of Theorem  \ref{T:metric} will be very neat. We will consider the following examples: smooth Riemannian manifolds, Alexandrov spaces with bounded curvature, and the Heisenberg group endowed with the Carnot-Carath\'eodory distance and the Haar measure.

%A natural question is now how restrictive are the assumptions of Theorem \ref{T:metric} and if they include interesting examples of metric measure spaces. As a possible answer to this  question will show that  Theorem \ref{T:metric}  applies to Riemannian manifolds, to $pmGH$-limits of Riemannian manifolds satisfying lower Ricci curvature bounds,  to Alexandrov spaces with curvature bounded from above and below and to the sub-Riemannian Heisenberg group.
%\subsection{Uniqueness of blow up measures}
%When considering the tangent space of a reference MCP space, if it is not unique then the metric spaces must be different; indeed if they where equal the measures must be equivalent. The same statement for Ricci limit spaces.
\subsection{Example 1: Riemannian manifolds}

\begin{proposition}[The volume measure is reference] \label{prop:VolRiem}
  Let $(M,g)$ be a complete (non necessarily compact but connected and  without boundary) $n$-dimensional smooth Riemannian manifold and   denote respectively with  $\sfd_g, \mu_g$ the Riemannian distance  and the volume measure induced by $g$.  Then $\mu_g$ is a reference measure for $(M, \sfd_g)$, in the sense of Definition \ref{D:referencemeasure}. More precisely, we show that there exists an inversion plan at every  point, i.e. ${\sf Ip}(\mu_g)=M$. 
\end{proposition}

\begin{proof}
For every $z \in M$ we have to construct an inversion plan $\pi^z$ satisfying \eqref{E:reverse}. 

{\bf Step 1}. 
First of all we claim  that there exist compact subsets $K_n \subset M$, $n \in \N$, satisfying the following:
\begin{itemize}
\item $\mu_g(M\setminus \bigcup_{ n \in \N} K_n)=0$; 
\item for every $x \in K_n$, there exists a unique minimizing geodesic $\gamma_{xz}:[0,1]\to M$ from $x$ to $z$ and it  is extendable to $[0,1+\frac{2}{n}]$ as minimizing geodesic;
\item the map $\Phi_n:K_n \to M$ defined by $\Phi_n(x):=\gamma_{xz}(1+\frac{1}{n})$ is bi-Lipschitz onto its image. 
\end{itemize}
In order to prove the above properties, first of all recall that the cut locus of $z$ has null measure, i.e. $\mu_g(C(z))=0$. 
Moreover for every $x \in M\setminus C(z)$ it is known that $z \notin C(x)$ (see for instance \cite[Lemma 2.1.11] {Kling}), 
so that for every $x \in M\setminus C(z)$ there exists a unique minimizing geodesic $\gamma_{xz}:[0,1]\to M$ 
from $x$ to $z$ and it is extendable (still as minimizing geodesic) to a bigger interval $[0, 1+\varepsilon]$ for some $\varepsilon=\varepsilon(x)>0$. 
Therefore, called
\begin{align}\label{eq:defKn}
\tilde{K}_n := \bigg\{ x \in M \, : \,& \exists ! \gamma_{xz}\in \Geo,  \text{ $\gamma_{xz}:[0,1]\to M$ geodesic from $x$ to $z$}, \big. \crcr 
&~ \big. \text{and it is extendable to  $\Big[0, 1+\frac{2}{n}\Big]$} \bigg\},
\end{align}
we get that $M\setminus C(z)=\bigcup_{n \in \N} \tilde{K}_n$, 
so  $\mu_g(M\setminus \bigcup_{ n \in \N} \tilde{K}_n)=0$; notice also that $K_n\subset K_{n+1}$, for every $n \in \N$. 
By continuous dependence on initial data in the geodesic equation (see for instance the proofs of \cite[Lemma 2.1.5,Proposition 2.1.10]{Kling}), 
one has that   $\tilde{K}_n\subset M$ is a closed subset. 
Hence, called 
$$
K_n:= \tilde{K}_n \cap \overline{B_n(z)},
$$
we get that $K_n\subset M$ is a compact subset for every $n \in \N$, and  $\mu_g(M\setminus \bigcup_{ n \in \N} K_n)=0$. 
The second claim is clear from the definition \eqref{eq:defKn} and the third claim follows by  the compactness of $K_n$ 
and the fact that the exponential map  centered at $z$ is  a diffeomorphism once restricted to $M\setminus C(z)$.
 \\
 
{\bf Step 2} Define a map $\Phi:M\to M$ in the following way: 
\begin{itemize}
\item if $x \in \bigcup_{n \in \N} K_n$,  called $n_0:=\min\{n\in \N \,: \, x \in K_{n} \}$, set $\Phi(x):=\Phi_{n_0}(x)$;
\item if $x \in M\setminus \bigcup_{n\in \N} K_n$, set $\Phi(x):=z$. 
\end{itemize}
Now let $\pi^z:=(Id, \Phi)_\sharp (\mu_g)$. Clearly we have $(P_1)_\sharp(\pi^z)=\mu_g$ and, by construction, $\pi^z$ is concentrated on  $H(z)$. Finally we also infer that $(P_2)_\sharp (\pi^z)\ll \mu_g$, indeed for every $E \subset M$ with $\mu_g(E)=0$ we have
\begin{eqnarray}
\Big[(P_2)_\sharp(\pi^z)\Big] (E)&=&\pi^z\left(P_2^{-1}(E)\right)=\mu_g\left( \Phi^{-1} (E)\right)=\mu_g\Big( \Phi^{-1} (E) \cap  ( \bigcup_{n \in \N} K_n ) \Big)  \nonumber \\
                                                  &= & \mu_g\Big( \bigcup_{n \in \N} \Phi^{-1} (E) \cap (K_n \setminus \cup_{1\leq j \leq n-1} K_j)\Big)  \nonumber \\
                                                  &\leq& \sum_{n \in \N} \mu_g(\Phi_n^{-1}(E))=0,  \nonumber 
\end{eqnarray}
where in the last equality we used that $\Phi_n$ are bi-Lipschitz by construction.  We conclude that $\pi^z$ is an inversion plan satisfying \eqref{E:reverse} hence, since $z \in M$ was arbitrary,   $\mu_g$ is a reference measure. 
\end{proof}

\begin{corollary}[$\MCP$ measures on complete Riemannian manifolds]\label{cor:Mg}
Let $(M,g)$ be a (possibly non compact) complete $n$-dimensional smooth Riemannian manifold with Ricci curvature bounded from below:
$$
Ric_g \geq K g, \quad\text{for some } K\in \R.
$$
Let $\mm\in {\mathcal M}^+(M)$ be a positive Radon measure with $\spt(\mm)=M$, such that $(M,\sfd_g,\mm)$ verifies the qualitative $\MCP$ condition \eqref{eq:QualMCP}. Then, called $\mu_g$ the Riemannian volume measure of $(M,g)$, we have
\begin{equation}\label{eq:Riem}
\mm \ll \mu_g.
\end{equation}  
If moreover $\mm$ is a reference measure for $(M,\sfd_g)$ and $\mu_g(M\setminus {\sf Ip}(\mm))=0$, then $\mm \sim \mu_g$.
\end{corollary}

 \begin{proof}
 By Proposition \ref{prop:VolRiem}, we know that $\mu_g$ is a reference measure with ${\sf Ip}(\mu_g)=M$, moreover it is clear that $(M, \sfd_g, \mu_g)$ is a non-branching metric space satisfying the qualitative $\MCP$ (notice that, more precisely it satisfies the $\MCP(K,n)$ condition). The first claim \eqref{eq:Riem}  is then a direct consequence of Theorem \ref{T:metric}. If $\mm$ is a reference measure,  we can reverse the role of $\mm$ and $\mu_g$ and conclude that, if $\mu_g(M\setminus {\sf Ip}(\mm))=0$, then  $\mu_g \ll \mm$ and  hence $\mm \sim \mu_g$.    
  \end{proof}

\subsection{Example 2: Alexandrov spaces with bounded curvature}
 
\begin{corollary}[$\MCP$ measures on Alexandrov spaces with curv. bounded above and below] \label{C:Alex}
Let $(X,\sfd)$ be an $n$-dimensional geodesically complete Alexandrov space with curvature bounded 
from above and below, and let $\mu:={\mathcal H}^n$ be the $n$-dimensional Hausdorff measure. 
Let $\mm\in {\mathcal M}^+(X)$ be a  positive Radon measure with $\spt(\mm)=X$, such that $(X,\sfd,\mm)$ verifies the qualitative $\MCP$ condition \eqref{eq:QualMCP}. 
Then
\begin{equation}\label{eq:Alex}
\mm \ll \mu. 
\end{equation}  
\end{corollary}

\begin{proof}
Thanks to the structural properties of Alexandrov spaces, the proof is not too far from the  Riemannian one above. Indeed it is well known that the curvature bound from below implies non-branching, so that trivially any measure is essentially non-branching.
On the other hand, the geodesic completeness combined with the upper curvature bound ensures that every point $p \in X$ has a neighborhood $U$ with the following properties (see for instance \cite[Sections 9.1.3, 9.1.7]{BBI}):
\begin{itemize}
\item the closure $\bar{U}$ is compact with Lipschitz boundary $\partial U$;
\item for every $x,y \in U$ there exists a unique geodesic $\gamma$ of $X$ joining them which is contained in $U$, i.e. $\gamma(0)=x, \gamma(1)=y, \gamma([0,1])\subset U$; moreover such geodesic is extendable to a minimizing geodesic  until it intersects $\partial U$;
\item every triple of points $xyz \in U$ satisfies the triangle comparison conditions with the two model spaces corresponding to the upper and lower curvature bounds.
\end{itemize}
In particular, since  $(U, \sfd)$ is itself an Alexandrov space with curvature bounded from below, it is non-branching and  its $n$-dimensional Hausdorff measure $\mu \llcorner U$  satisfies the qualitative $\MCP$ condition \eqref{eq:QualMCP} (more precisely it satisfies $\MCP((n-1)K, n)$ where $K\in \R$ stands for the lower curvature bound, see \cite[Proposition 2.8]{Ohta07} which is based on the previous work \cite{KS01}).
\\  In order to apply Theorem \ref{T:metric} it is then enough to show that $\mu \llcorner U$ is a reference measure for  $(U, \sfd)$ in the sense of Definition \ref{D:referencemeasure}. To this aim, for every $z \in U$, define the following geodesic inversion map $\Phi_z:U\to U$ with center $z$
\begin{equation}\label{eq:defPhiR}
\Phi_z(x)=Exp_z\left(- \alpha_z(x) \cdot Exp^{-1}_z(x)\right), \quad \forall x \in U,
\end{equation}
where  $Exp_z$ is the exponential map of the Alexandrov space centered at $z$ and 
$$\alpha_z(x):=\frac{\sup\{t\,:\, Exp_z(-t\, Exp_z^{-1}(x)) \in U \}}{\sup\{t\,:\, Exp_z(t\, Exp_z^{-1}(x)) \in U \} };$$ 
notice that $\alpha:U \to (0, +\infty)$ is Lipschitz  since $U$ is geodesically convex and $\partial U$ is Lipschitz.  Hence, by  comparison with geodesic triangles in the two model spaces corresponding to the upper and lower curvature bounds we get  that  the above geodesic inversion $\Phi_z:U\to U$ with center $z$  is a bi-Lipschitz map. Therefore  $(\Phi_z)_\sharp (\mu \llcorner U) \ll \mu \llcorner U$ and,  setting 
$$\pi^z:= (Id|_U, \Phi_z)_{\sharp} (\mu \llcorner U), $$
 we get that $\mu_g \llcorner U$ is a reference measure for  $(U, \sfd)$. The thesis follows  then by applying Theorem \ref{T:metric} locally to  $\mu_g \llcorner U$ and $\mm\llcorner U$, since $U$ was a suitable neighborhood of an arbitrary point $p \in X$.
\end{proof}

\begin{remark} Let us remark that from the work of Nikolaev \cite{Nikolaev}, it is known that a space with  curvature bounded from above and below  in the Alexandrov sense has the structure of a $C^{3,\alpha}$-manifold with a $C^{1,\alpha}$-metric tensor. But, since the geodesic equations involve first order derivatives of the metric which a priori may not be locally Lipschitz, this remarkable structural result is not so useful to our purpose, that is to show  that there is bi-Lipschitz  geodesic inversion at every point.  This is the reason why we used the argument above using classical comparison, and why we believe that the notion of inversion plan is not trivial also for such spaces.  
\end{remark}

\subsection{Example 3: the Heisenberg group}

Let us start by recalling some basic facts about the $n$-dimensional Heisenberg group $\HH^n$.  
Such space is a non-commutative stratified nilpotent Lie group; as a set it can and will  be identified with its Lie algebra $\R^{2n+1} \simeq \C^{n} \times \R$. 
We denote a point in $\HH^n$ indifferently by $x=(\xi,\eta,t)=[\zeta, t]$, where $\xi=(\xi_1,\ldots,\xi_n), \eta=(\eta_1,\ldots,\eta_n)\in \R^n$, $t \in \R$ 
and $\zeta=\xi+i \eta \in \C^n$. The group law in this system of coordinates is given by 
$$
[\zeta,t] \cdot [\zeta', t']=\left[\zeta+\zeta', t+t'+ 2 \sum_{j=1}^n {\rm Im} \zeta_j \bar{\zeta}_j' \right]. 
$$
The center of the group is 
\begin{equation}\label{eq:defL}
 L:=\{[0,s] \in \HH^n\, : \, s \in \R\};
\end{equation}  
define also $L^*, L^+,L^-$
restricting $s$ to belong to, respectively, $\R\setminus\{0\}, [0,+\infty), (-\infty,0]$. 
The Haar measure of the group is given by the Lebesgue measure $\L^{n+1}$ on  $\HH^n \simeq \R^{2n+1}$. 
The left invariant vector fields
\begin{equation}
 X_j:=\partial_{\xi_j} + 2 \eta_j \partial_t, \quad Y_j:=\partial_{\eta_j}-2 \xi_j \partial_t
 \end{equation}
Lie generate the Lie algebra of the group (which is $\R^{2n+1}$), the only nontrivial bracket relation being  $[X_j,Y_j]=-4Z$, where $Z:=\partial_t$. 
We endow $\HH^n$ with the standard Carnot-Carath\'eodory distance $\sfd_{c}$ defined as follows.  
An absolutely continuous curve $\gamma:[0,T]\to \R^{2n+1}$  is said to be \emph{horizontal} if there exist measurable functions $h_j:[0,T]\to \R$ such that
\begin{equation}\label{eq:horcurve}
 \dot\gamma(s)=\sum_{j=1}^n h_j(s) X_j(\gamma(s))+h_{n+j}(s) Y_j(\gamma(s)).
 \end{equation}
We say that $\gamma$ is sub-unit if moreover $\sum_{j=1}^{2n} h_j(s)^2\leq 1$ for a.e. $s \in [0,T]$. Since $(X_1, \ldots, X_n, Y_1, \ldots, Y_n)$ 
Lie generate the Lie algebra of the group it follows from Chow's Theorem that any couple of points $x, y$ in $\HH^n$ can be joined by a sub-unit curve. Then the Carnot-Carath\'eodory distance $\sfd_c$ between $x$ and $y$ is defined as
\begin{equation}\label{eq:defCCdist}
\sfd_c(x,y):=\inf\{T>0\, : \, \exists \text{ a sub-unit  curve } \gamma:[0,T]\to \HH^n \text{ such that } \gamma(0)=x, \gamma(T)=y \}.
\end{equation}
It follows easily from the left invariance and homogeneity with respect to the dilations of the vector fields $X_j$ and $Y_j$ that $\sfd_c$ is a left invariant and homogeneous distance on $\HH^n$.
We say that a curve $\gamma:[0,T]\to \HH^n$ is a \emph{minimal geodesic with unit speed} if 
$$\sfd_c(\gamma(t), \gamma(t'))=|t-t'|, \quad \forall t, t' \in [0,T].$$
 By the work of Pansu \cite{Pansu} on differentiability of Lipschitz functions  between stratified and nilpotent Lie groups (see also \cite[Theorem 3.2]{AR}), it follows that for any $x,y \in \HH^n$ there exists a sub-unit  (a fortiori unit speed) minimal geodesic joining $x$ to $y$. The equations for the minimal geodesics can be explicitly computed and can be found in the literature (see for instance \cite{Gaveau,Montgo, Monti}), here we follow the presentation of \cite{AR} by recalling just the main results that we will need to prove Corollary \ref{cor:Heisenberg}. 

Set $\bS:=\{a+ib \in \C^n \,:\, |a+ib|=1\}$. For any $a+ib \in \bS, v \in \R$ and $r>0$,  we say that a curve $\gamma:[0,r]\to \HH^n$ is a \emph{curve with parameter $(a+ib,v,r)$}, if $\gamma(s)=(\xi(s),\eta(s),t(s))$, where
\begin{eqnarray}
\xi_j(s):=\frac{b_j\left(1-\cos \frac{vs}{r}\right)+a_j \sin  \frac{vs}{r}}{v} r,& \eta_j:=\frac{-a_j\left(1-\cos  \frac{vs}{r}\right) +b_j \sin  \frac{vs}{r} }{v} r, \; \;  \; t= 2 \frac{\frac{vs}{r} -\sin  \frac{vs}{r} }{v^2} r^2, \;  \;  \;\text{ if } v\neq 0;  \nonumber \\ 
&\xi_j(s):=a_j s,\quad  \eta_j:= b_j  s, \quad t\equiv 0,  \qquad \text{ if } v=0 .  \label{eq:defgamma}
\end{eqnarray}
As the next theorem states (for the proof see for instance  \cite[Theorem 3.4]{AR}), these are the only geodesics in $\HH^n$ starting from the origin.

\begin{theorem}[Geodesics in $\HH^n$]\label{thm:GeoHn}
Sub-unit minimal geodesics starting from the origin $0$ are curves $\gamma$ with parameter $(a+ib, v, r)$, for some  $(a+ib, v, r)\in \bS \times [-2\pi, 2 \pi] \times (0,+\infty)$. In particular, curves $\gamma$ with parameter $(a+ib, v, r)\in \bS \times \R \times (0,+\infty)$ with $|v|> 2 \pi$   are not minimal geodesics. Conversely, any curve $\gamma$ with parameter $(a+ib, v, r)\in \bS \times [-2\pi, 2 \pi] \times (0,+\infty)$ is a sub-unit minimal geodesic starting from $0$. More precisely the following holds:   
\begin{itemize}
\item For any $x=[0,t]\in L^*$ with $t>0$, resp. $t<0$, sub-unit minimal geodesics from $0$ to $x$ form a  family of curves parametrized on $\bS$, each one being obtained from one another by a
rotation around the axis $L$. More precisely this family is composed by all curves $\gamma$ with parameter $(a+ib,2\pi, \sqrt{\pi t})$, resp.  $(a+ib,-2\pi, \sqrt{\pi |t|})$, where $a+ib\in \bS$.

\item  For any $x \in \HH^n\setminus L$ there exists a unique $(a+ib,v,r)\in  \bS \times (-2\pi, 2 \pi) \times (0,+\infty)$  so that the curve $\gamma$ with parameter $(a+ib,v,r)$ joins $0$ to $x$ and in particular $\gamma:[0,r]\to \HH^n $ is the unique sub-unit minimal geodesic from $0$ to $x$. Moreover such $\gamma$ is extendable to  a minimizing geodesic (via  the same formulas \eqref{eq:defgamma}) defined on the maximal interval $[0, 2\pi r/|v|]\supset [0,r]$.
\end{itemize}
\end{theorem}

Consider now the endpoint map $\Phi:\bS \times [-2\pi, 2 \pi]\times (0, +\infty) \to \HH^n, \; (a+ib,v, r)\mapsto (\xi,\eta,t),$ 
which associates to each geodesic $\gamma$ with parameter $(a+ib,v,r)$ its endpoint  $(\xi,\eta,t)$ given by 
\begin{eqnarray}
\xi_j:=\frac{b_j(1-\cos v)+a_j \sin v}{v} r,&&  \eta_j:=\frac{-a_j(1-\cos v)+b_j \sin v}{v} r, \quad  t= 2 \frac{v-\sin v}{v^2} r^2, \qquad \text{ if } v\neq 0;  \nonumber \\ 
&&\xi_j:=a_j r,\quad  \eta_j:= b_j  r, \quad t\equiv 0,  \qquad \text{ if } v=0 .  \label{eq:defPhi}
\end{eqnarray}

Theorem \ref{thm:GeoHn} directly implies that 
\begin{itemize}
\item the range of $\Phi$ is $\HH^n$;
\item $\sfd_c(0, \Phi(a+ib, v, r))=r$, for all  $(a+ib, v, r)\in \bS \times [-2\pi, 2 \pi] \times (0,+\infty)$;
\item the map $\Phi$ is bijective from $D$ to $\HH^n \setminus L$, where
\begin{equation}\label{eq:defD}
D:= \bS \times (-2\pi, 2 \pi) \times (0,+\infty);
\end{equation}
\item the map $\Phi$ is not injective on $\{(a+ib, v, r)\,: \, a+ib \in \bS, |v|=2\pi, r>0\}$ whose image corresponds to the center $L$ of the group. More precisely
\begin{eqnarray}
L^+\setminus\{0\}&=&\{\Phi(a+ib, v, r)\in \HH^n \, : \, a+ib \in \bS, v = 2\pi, r>0 \} \, ; \nonumber \\
L^-\setminus\{0\}&=&\{\Phi(a+ib, v, r)\in \HH^n \, : \, a+ib \in \bS, v =- 2\pi, r>0 \} \, . \nonumber 
\end{eqnarray}
\end{itemize}
Moreover, an explicit computation of the Jacobian of $\Phi$  (see \cite[Page 161]{Monti} or \cite[Propositions 1.7,1.12]{Juillet}) 
shows that it does not vanish on $D$. Therefore we have the next

\begin{proposition}\label{prop:PhiDiffeo}
The map $\Phi$ is a diffeomorphism from $D$ onto $\HH^n\setminus L$.
\end{proposition}

The next lemma of linear algebra will be useful in the sequel.

\begin{lemma}\label{lem:Av}
Let $v\neq 0$ and consider the linear map $L_v:\R^2 \to \R^2$ defined by 
$$(x_1,x_2)\mapsto L_v(x_1,x_2):= (x_1 \, \sin v  + x_2 \, (1-\cos v)  \; , \; - x_1\, (1-\cos v) + x_2 \, \sin v  ).$$
Then the operator $A_v:\R^2 \to \R^2$ given by $A_v:=(L_{-v})^{-1} L_v$ is an orthogonal transformation of $\R^2$.
\end{lemma}

\begin{proof}
First of all observe that $\det(L_v)=\det(L_{-v})=2 (1-\cos v)\neq 0$, since  by assumption $v\neq 0$. Moreover the inverse map is given by 
\begin{equation}\label{eq:defLv}
(L_v)^{-1} = \frac{1}{\det(L_v)} L_v^T,
\end{equation}
where $L_v^T$ is the transpose operator of $L_v$.  We conclude  that
$$
A_v A_v^T =   \frac{1}{\det(L_v)^2}  \left( L_{-v}^T \; L_v \right) \left( L_{-v}^T \; L_v \right)^T=    L_{-v}^{-1} \; L_v \;  L_v^{-1}  \; L_{-v}= Id. 
$$
\end{proof}

Let us consider the following map
\begin{eqnarray}
\Psi: \bS \times \Big((-2\pi,0) \cup (0,2 \pi) \Big) \times (0, +\infty) &\to&  \bS \times \Big((-2\pi,0) \cup (0,2 \pi) \Big) \times (0, +\infty)  \nonumber \\
 (a+ib, v, r) &\mapsto& \left(A_v(a+ib),\;  -\frac{v\, (2\pi-|v|)}{2|v|} \, , \; \frac{2\pi-|v|}{2|v|} \,  r \right), \label{eq:defPsi}
\end{eqnarray}
where $A_v$ was defined in Lemma  \ref{lem:Av}, and $A_v(a+ib)$ is a short notation for the $2n$-vector $(A_v(a_j+ib_j))_{j=1,\ldots,n}$. Notice that $A_v$ maps $\bS$ to $\bS$ thanks to Lemma \ref{lem:Av}.
The next lemma is the key to show that the Heisenberg group $\HH^n$  enters into the framework of Theorem  \ref{T:metric}. 
%Let $\psi:[0, 2 \pi) \to (1,2]$ be a non increasing $C^{1}$ function such that (le condizioni non sono compatibili!!)
%\begin{eqnarray}
%\psi(v)=2, \; \text{ if  } v \in [0, \pi/2]\; , \qquad  \psi(v)=\frac{2\pi+v}{2v} \text{  if  }v \in [\pi, 2 \pi)\, \nonumber\\
 %\psi(v)\leq \frac{2\pi+v}{2v}\; \; \forall v \in [0, 2\pi)\, , \;  \text{ and $v\mapsto v(1-\psi(v))$ is strictly decreasing.}   \label{eq:defpsi}
%\end{eqnarray} 

\begin{lemma}\label{lem:Lambda}
The following hold:
\begin{enumerate}
\item The map $\Psi:  \bS \times \Big((-2\pi,0) \cup (0,2 \pi) \Big) \times (0, +\infty)\; \to  \; \bS \times \Big((-2\pi,0) \cup (0,2 \pi) \Big) \times (0, +\infty)$ 
defined in \eqref{eq:defPsi} is a diffeomorphism. 

\item The map
\begin{equation}\label{eq:defLa}
\Lambda: \HH^n\setminus\Big(L\cup \{[\zeta,t]\in \HH^n\; :\, t=0 \} \Big) \to  \HH^n\setminus\Big(L\cup \{[\zeta,t]\in \HH^n\; :\, t=0 \} \Big), \quad \Lambda:= \Phi \circ\Psi \circ \Phi^{-1}
\end{equation}
is a diffeomorphism. Moreover 
\begin{equation} \label{eq:LaH0}
\sfd_c(x, \Lambda(x))= \sfd_c(x, 0)+ \sfd_c(0, \Lambda(x)), \quad \forall x \in \HH^n\setminus\Big(L\cup \{[\zeta,t]\in \HH^n\; :\, t=0 \} \Big).
\end{equation}
In other words $(x, \Lambda(x))\in H(0)$, for every  $x \in \HH^n\setminus\Big(L\cup \{[\zeta,t]\in \HH^n\; :\, t=0 \} \Big)$,  where $H(0)$ was defined in \eqref{E:couples}.
\end{enumerate}
\end{lemma}
 
\begin{proof}
 The first claim directly follows from the formula \eqref{eq:defPsi} and Lemma  \ref{lem:Av}, we pass then to the proof of the second claim. 
 Recalling the definition \eqref{eq:defPhi} of $\Phi$, Theorem \ref{thm:GeoHn} and Proposition \ref{prop:PhiDiffeo}, we have that 
$$
\Phi: \bS \times \Big((-2\pi,0) \cup (0,2 \pi) \Big) \times (0, +\infty)\to   \HH^n\setminus\Big(L\cup \{[\zeta,t]\in \HH^n\; :\, t=0 \} \Big) \quad  \text{is a diffeomorphism.}
$$ 
In particular, for every $x \in \HH^n\setminus\Big(L\cup \{[\zeta,t]\in \HH^n\; :\, t=0 \} \Big)$ we can set
\begin{equation} \label{eq:defa+ib}
(a+ib, v, r) := \Phi^{-1}(x), \quad \text{ with }  (a+ib, v, r) \in \bS \times \Big((-2\pi,0) \cup (0,2 \pi) \Big) \times (0, +\infty).
\end{equation}
Let us denote with $\gamma_{a+ib,v,r}$ the curve with parameter $(a+ib,v,r)$ defined in \eqref{eq:defgamma}. 
By Theorem \ref{thm:GeoHn}, we know that $\gamma_{a+ib,v,r}:[0, 2\pi r /|v|] \to \HH^n$ is  a minimizing geodesic and, 
by construction,  $\gamma_{a+ib,v,r}(0)=0$, $\gamma_{a+ib,v,r}(r)=x$. 
Clearly  the reverse-parametrized curve
$$s \in [0,r] \mapsto \gamma_{a+ib,v,r}(r-s)$$
is a geodesic from $x$ to $0$. Since $x^{-1}=-x$ and  the left translation in $\HH^n$ is an isometry,  we infer that 
$$s \in [0,r] \mapsto -x \cdot  \gamma_{a+ib,v,r}(r-s) $$ 
is a geodesic from $0$ to $-x=x^{-1}$. On the other hand,  from the definition of $A_v$ in Lemma \ref{lem:Av} and  the explicit expression of the endpoint map $\Phi$ given in \eqref{eq:defPhi}, it is readily checked that 
$$ \Phi(A_v(a+ib),-v, r)=-\Phi(a+ib, v, r)=-x.$$
In other words $\gamma_{A_v(a+ib), -v, r}: [0,r] \to \HH^n$ is a geodesic from $0$ to $-x$. 
Since $x \notin L$, the geodesic from $0$ to $-x$ is unique thanks to Theorem \ref{thm:GeoHn}, so in particular we have that 
$$
\gamma_{A_v(a+ib),-v,r}(s)= -x \cdot \gamma_{a+ib,v,r}(r-s), \quad \forall s \in [0, r].
$$
Recalling that actually  $\gamma_{A_v(a+ib),-v,r}$ is length minimizing on the larger interval $[0, 2\pi r /|v|]$,  it follows that we can extend  $s \mapsto  \gamma_{a+ib,v,r}(r-s)$ up to  $[0, 2\pi r /|v|]$ to a length minimizing geodesic.
By observing that $\frac{|v|+2\pi}{2|v|}r \in (r, 2\pi r/ |v|)$, in particular we get
\begin{eqnarray}
 \gamma_{A_v(a+ib),-v,r} \Big(\frac{2\pi-|v|}{2|v|}\, r  \Big)&=& \gamma_{a+ib,v,r}\Big( \frac{|v|-2\pi}{2|v|} \, r\Big)\;  \nonumber \\
\sfd_c \left(x, \gamma_{A_v(a+ib),-v,r} \Big(\frac{2\pi-|v|}{2|v|}\, r  \Big)  \right)&=&   \sfd_c (x,0) 
+  \sfd_c\left(0, \gamma_{A_v(a+ib),-v,r} \Big(\frac{2\pi-|v|}{2|v|}\, r  \Big) \right) \; \label{eq:antipo} .
\end{eqnarray}
The explicit parametrizations given in  \eqref{eq:defgamma} and  \eqref{eq:defPhi}   imply that 
\begin{equation}\label{eq:gaLa}
 \gamma_{A_v(a+ib),-v,r}  \Big(\frac{2\pi-|v|}{2|v|}\, r  \Big) = \Phi \left(A_v(a+ib),\;  -\frac{v\, (2\pi-|v|)}{2|v|} \, , \; \frac{2\pi-|v|}{2|v|} \,  r \right) = \Phi \circ \Psi \circ \Phi^{-1}(x) = \Lambda (x),
\end{equation}
where we used \eqref{eq:defa+ib}, \eqref{eq:defPsi} and \eqref{eq:defLa}. 
The claim \eqref{eq:LaH0} follows then by the combination of  \eqref{eq:antipo}  and \eqref{eq:gaLa}; finally, 
$\Lambda$ is a diffeomorphism since composition of diffeomorphisms.
\end{proof}

We are now ready to prove the main result of this subsection.
 
\begin{corollary}\label{cor:Heisenberg}
Let $(\HH^n, \sfd_c)$ be the $n$-dimensional Heisenberg group endowed with the Carnot-Carath\'eodory distance and let $\mm \in {\mathcal M}^+ (\HH^n)$ be a positive Radon measure with $\spt(\mm)=\HH^n$, such that $(\HH^n, \sfd_c, \mm)$ 
satisfies the qualitative $\MCP$ condition  \eqref{eq:QualMCP}. 
Then $\mm$ is absolutely continuous with respect to the Haar measure on $(\HH^n, \sfd_c)$, 
which coincides with the $(2n+1)$-Lebesgue measure $\L^{2n+1}$ under the identification $\HH^n\simeq \R^{2n+1}$.
\end{corollary} 

\begin{proof}
First of all it is well known (for example it follows from Theorem \ref{thm:GeoHn})  that $(\HH^n, \sfd_c)$ is a non-branching geodesic metric space, moreover   the Heisenberg group $(\HH^n, \sfd_c)$  endowed with the Haar measure   $\L^{2n+1}$
satisfies the qualitative $\MCP$ condition  \eqref{eq:QualMCP} (more precisely it satisfies $\MCP(0,N)$ if and only if $N\geq 2n+3$ by the work of Juillet \cite{Juillet}). 
In order to apply Theorem \ref{T:metric} it is then 
enough to show that $\L^{2n+1}$ is a reference measure for $(\HH^n, \sfd_c)$. 
\\To this aim, observe since  that the map 
$$
\Lambda: \HH^n\setminus\Big(L\cup \{[\zeta,t]\in \HH^n\; :\, t=0 \} \Big) \to  \HH^n\setminus\Big(L\cup \{[\zeta,t]\in \HH^n\; :\, t=0 \} \Big)
$$
defined in  \eqref{eq:defLa} is a diffeomorphism, and since 
$$
\L^{2n+1}\left(L\cup \{[\zeta,t]\in \HH^n\; :\, t=0 \}\right)=0,  
$$
then 
\begin{equation}\label{eq:Lll}
\Lambda_\sharp (\L^{2n+1}) \ll  \L^{2n+1}.
\end{equation}
Defining then $\pi^0:=(\L^{2n+1}, \Lambda_\sharp (\L^{2n+1}) )$, we get that $\pi^0$ is concentrated on the graph of $\Lambda$, which is contained in $H(0)$ thanks to  
Lemma \ref{lem:Lambda}. 
Therefore $\pi^0(X\times X \setminus H(0))=0$ and \eqref{E:reverse} holds at $0$. 
The construction of $\pi^x$ for $x \in \HH^n$, satisfying  \eqref{E:reverse}, can be reduced to the one of $\pi^0$ just by conjugation with $x$. We conclude that ${\sf Ip}(\L^{2n+1})=\HH^n$ and the thesis follows directly by Theorem \ref{T:metric}.
\end{proof}

%%%%%%%%%%%%%%%%%%%%%%%%%%%%%%%%%%%%%%%%%%%%%%%%%%%%%%%
%%%%%%%%%%%%%%%%%%%%%%%%%%%%%%%%%%%%%%%%%%%%%%%%%%%%%%%
%%%%%%%%%%%%%%%%%%%%%%%%%%%%%%%%%%%%%%%%%%%%%%%%%%%%%%%
%%%%%%%%%%%%%%%%%%%%%%%%%%%%%%%%%%%%%%%%%%%%%%%%%%%%%%%
%%%%%%%%%%%%%%%%%%%%%%%%%%%%%%%%%%%%%%%%%%%%%%%%%%%%%%%
%%%%%%%%%%%%%%%%%%%%%%%%%%%%%%%%%%%%%%%%%%%%%%%%%%%%%%%
%%%%%%%%%%%%%%%%%%%%%%%%%%%%%%%%%%%%%%%%%%%%%%%%%%%%%%%
%%%%%%%%%%%%%%%%%%%%%%%%%%%%%%%%%%%%%%%%%%%%%%%%%%%%%%%
%%%%%%%%%%%%%%%%%%%%%%%%%%%%%%%%%%%%%%%%%%%%%%%%%%%%%%%
%%%%%%%%%%%%%%%%%%%%%%%%%%%%%%%%%%%%%%%%%%%%%%%%%%%%%%%
%%%%%%%%%%%%%%%%%%%%%%%%%%%%%%%%%%%%%%%%%%%%%%%%%%%%%%%
%%%%%%%%%%%---------ALEXANDROV----------------%%%%%%%%%%%%%%%%%%%%%%%%%%%


\begin{thebibliography}{10}

\bibitem{AGMR}  L.~Ambrosio, N.~Gigli, A.~Mondino and T.~Rajala,
 \newblock Riemannian Ricci curvature lower bounds in metric spaces with $\sigma$-finite measure,
\newblock{arxiv:1207.4924, (2012)}, to appear on Trans. Amer. Math. Soc.



\bibitem{AGS} L.~Ambrosio, N~Gigli, G.~Savar\'e
\newblock Bakry-\'Emery curvature-dimension condition and Riemannian Ricci curvature bounds.
\newblock{\em Annals of Probab.}, to appear.


\bibitem{AGS11b} \leavevmode\vrule height 2pt depth -1.6pt width 23pt,
 \newblock Metric measure spaces with {R}iemannian {R}icci curvature bounded from below, 
 \newblock{\em Duke Math. J.}, 163, 1405--1490, (2014).  

\bibitem{AMS} L.~Ambrosio, A.~Mondino, G.~Savar\'e,  
\newblock Nonlinear diffusion equations and curvature conditions in metric measure spaces, 
\newblock {\em in preparation}.
 

\bibitem{AR} L.~Ambrosio, S.~Rigot,
\newblock Optimal mass transportation in the Heisenberg group.
\newblock{\em J. Funct.  Anal.},  208,  261--301, (2004).
 
 
\bibitem{BakryEmery_diffusions} D.~Bakry and M.~\'Emery,
\newblock Diffusions hypercontractives
\newblock{ \em Seminaire de Probabilit\'es XIX, Lecture Notes in Math.}, Springer-Verlag, New York. 1123 (1985), 177--206.

\bibitem{BakryLedoux} D.~Bakry and M.~Ledoux,
\newblock A logarithmic Sobolev form of the Li-Yau parabolic inequality.
\newblock{ \em Rev. Mat. Iberoam.}  22,   (2006), 683--702. 
 

\bibitem{BBI}  D.~Burago, Y.~Burago and S.~Ivanov,
\newblock A course in metric geometry.
\newblock {\em Grad. Studies in Math.,} Vol.33, Americ. Math. Soc.,  (2001).

%\bibitem{BH99} M. R.~Bridson and  A.~Haefliger,
%\newblock Metric Spaces of Non-Positive Curvature.
%\newblock{\em Grundlehren der mathematischen Wissenschaften,} Vol. 319, Springer, (1999).

\bibitem{biacava:strettconv}
S.~Bianchini and F.~Cavalletti,
\newblock The {M}onge problem for distance cost in geodesic spaces.
\newblock {\em Commun. Math. Phys.},  318, 615 -- 673 (2013).


\bibitem{cava:RCD}
F.~Cavalletti,
\newblock Monge problem in metric measure spaces with {R}iemannian curvature-dimension condition.
\newblock {\em Nonlinear Anal.}, 99 (2014), 136 -- 151.



\bibitem{cavahues:existence}
F.~Cavalletti and M.~Huesmann,
\newblock Existence and uniqueness of optimal transport maps.
\newblock {\em Ann. I. H. Poincar\'e AN}, (2014) http://dx.doi.org/10.1016/j.anihpc.2014.09.006.



\bibitem{C}
G.~Cheeger
\newblock Differentiability of Lipschitz functions on metric measure spaces.
\newblock {\em Geom. Funct. Anal.}, 9 (1999), no. 3, 428 -- 517.



\bibitem{CC1}
G.~Cheeger and T.H.~Colding,
\newblock On the structure of spaces with {R}icci curvature bounded below. I.
\newblock {\em J. Diff. Geom.}, 45 (1997),  406 -- 480.

\bibitem{CC2}
\leavevmode\vrule height 2pt depth -1.6pt width 23pt,
\newblock On the structure of spaces with {R}icci curvature bounded below. II.
\newblock {\em J. Diff. Geom.},  54,  (2000), 13--35.  



\bibitem{cc:IIIJDE2000}  
\leavevmode\vrule height 2pt depth -1.6pt width 23pt,
\newblock On the structure of spaces with {R}icci curvature bounded below. III.
\newblock {\em J. Diff. Geom.}, 54 (2000), 37 -- 74.


%\bibitem{Col} T.H.~Colding,
%\newblock Ricci curvature and volume convergence. 
%\newblock {Ann. of Math.} (2), 145, (1997), 477--501.  


\bibitem{CN} T.H.~Colding and A. Naber,
\newblock Sharp H\"older continuity of tangent cones for spaces with a lower Ricci curvature bound and applications.
\newblock{ Ann. of Math.} (2), 176 (2012), 1173--1229.

%\bibitem{David88} G. David,
%\newblock Morceaux de graphes lipschitziens et int\'egrales singuli\'eres sur un surface.
%\newblock Rev. Mat. Iberoamericana, 4 (1988), 73--114.


%\bibitem{David2014} G.C~David,
%\newblock Bi-Lipschitz Pieces between Manifolds
%\newblock{\em preprint  arXiv:1312.3911}, (2014). 

\bibitem{CJ} M.~Csornyei and P. Jones,
\newblock Product Formulas for Measures and Applications to Analysis and Geometry.
\newblock{www.math.sunysb.edu/Videos/dfest/PDFs/38-Jones.pdf.} 
 
 
 
 
\bibitem{EKS} M~Erbar, ~Kuwada, K.T.~Sturm,
\newblock On the Equivalence of the Entropic Curvature-Dimension Condition and Bochner's Inequality on Metric Measure Space,
\newblock Preprint arXiv:1303.4382, (2013). To appear in Invent. Math.



\bibitem{Fre:measuretheory4}
D.~H. Fremlin,
\newblock {\em Measure Theory}, volume~4.
\newblock Torres Fremlin, 2002.



\bibitem{Gaveau} B.~Gaveau, 
\newblock Principe de moindre action, propagation de la chaleur et estim\'ees sous elliptiques sur certains groupes nilpotents.
\newblock{\em Acta Mathematica }139,  1-- 2, 95--153,  (1977).

\bibitem{GigliMap}  N.~Gigli,
\newblock Optimal maps in non branching spaces with Ricci curvature bounded from below,
\newblock{\em Geom. Funct. Anal.}, 22 (2012) no. 4, 990--999.


\bibitem{GigliSplit}  N.~Gigli,
\newblock The splitting theorem in non-smooth context,
\newblock{\em preprint arXiv:1302.5555}, (2013).

\bibitem{GMR}  N.~Gigli, A.~Mondino and T.~Rajala,
\newblock  Euclidean spaces as weak tangents of infinitesimally Hilbertian metric measure spaces with Ricci curvature bounded below,
\newblock{\em Journal fur die Reine und Ang. Math. (Crelle's journal)} (in press). DOI 10.1515/ crelle-2013-0052. 

\bibitem{Gri92}  A.~Grigor'yan,
\newblock Heat equation on a non-compact Riemannian manifold, 
\newblock{\em Math USSR Sb.} 72 (1992) no. 1, 47--77.  (Translated from  Russian Matem. Sbornik, 182 (1991), no. 1, 55--87).

%\bibitem{Gro87}M.~Gromov, 
%\newblock Hyperbolic groups. 
%\newblock{\em  Essays in group theory} (S. M. Gersten, ed),
%Springer Verlag, MSRI Publ. 8, 75--263, (1987).

\bibitem{Juillet} N.~Juillet,
\newblock Geometric Inequalities and Generalized Ricci Bounds in the Heisenberg Group.
\newblock{\em Int. Math. Res. Not.},  2009, 13,   2347--2373, (2009).

%\bibitem{Kirch94} B.~Kirchheim,
%\newblock Rectifiable metric spaces: local structure and regularity of the Hausdorff measure. 
%\newblock{\em Proc. Amer. Math. Soc.},  121,  1,  (1994).

%\bibitem {Klein99} B.~Kleiner,
%\newblock The local structure of length spaces with curvature bounded above.
%\newblock{\em Math. Zeit.,} 231, 3,  409-456,  (1999).


\bibitem{Kling} W.P.~Klingenberg, 
\newblock Riemannian geometry. Second edition.
\newblock{\em de Gruyter Studies in Mathematics}, 1. Walter de Gruyter \& Co., Berlin, (1995), x+409 pp.

\bibitem{KS01} K.~Kuwae and T.~Shioya, 
\newblock On generalized measure contraction property and energy functionals over Lipschitz maps.
\newblock{\em Potential Anal.}, 15, 105--121, (2001). 


\bibitem{KS10} \leavevmode\vrule height 2pt depth -1.6pt width 23pt,
\newblock  Infinitesimal Bishop-Gromov condition for Alexandrov spaces. Probabilistic approach to geometry.
\newblock{\em Adv. Stud. Pure Math.} 57, Math. Soc. Japan, Tokyo, (2010)  

\bibitem{LV} J.~Lott and C.~Villani, 
\newblock Ricci curvature for metric-measure spaces via optimal transport,
\newblock {\em Ann. of Math.}, 169, 903--991  (2009).


\bibitem{MN}  A.~Mondino and A.~Naber,
\newblock  Structure Theory of Metric-Measure Spaces with Lower Ricci Curvature Bounds I,
\newblock{\em preprint  arXiv:1405.2222}, (2014).

\bibitem{Montgo} R.~Montgomery, 
\newblock A tour of subriemannian geometries, their geodesics and applications.
\newblock{\em Mathematical Surveys and Monographs,} Vol. 91, Amer. Math. Soc., Providence, RI, (2002),
xx+259pp.

\bibitem{Monti} R.~Monti, 
\newblock Some properties of Carnot- Carath\'eodory balls in the Heisenberg group,
\newblock {\em Rend. Mat. Acc. Lincei}, 11, 155--167, (2000).

\bibitem{Nikolaev} I.G. ~Nikolaev, 
\newblock Smoothness of the metric of spaces with two-sided bounded Alexandrov curvature,
\newblock {\em Siberian Math. Journ.}, 24, 247--263, (1983).



\bibitem{Ohta07} S.I.~Ohta,
\newblock On the measure contraction property of metric measure spaces
\newblock {\em Comment. Math. Helv.}, 82,  805--828, (2007).


\bibitem{OS94} Y.~Otsu and T.~Shioya,
\newblock  Riemannian structure of Alexandrov spaces. 
\newblock {\em J. Diff.. Geom.},  39, 629--658, (1994). 

\bibitem{Sturm98} K.T.~Sturm,
\newblock  Diffusion processes and heat kernels on metric spaces.
\newblock{\em Annals of Probab.}, 26, 1--55, (1998).

\bibitem{Sturm06I}
\leavevmode\vrule height 2pt depth -1.6pt width 23pt,
\newblock On the geometry of
  metric measure spaces. {I},
 \newblock{\em  Acta Math.} 196, 65--131, (2006) .

\bibitem{Sturm06II}
\leavevmode\vrule height 2pt depth -1.6pt width 23pt,
On the geometry of  metric measure spaces. {II},
 \newblock{ Acta Math.,} 196, 133--177, (2006) 



\bibitem{RS} T.~Rajala and  K.T.~Sturm,
\newblock Non-branching geodesics and optimal maps in strong $CD(K,\infty)$-spaces. 
\newblock{Calc. Var. Partial Differential Equations}, 50, 3-4, 831--846, (2014). 


\bibitem{Pansu} P. ~Pansu,
\newblock  M\'etriques de Carnot-Carath\'eodory et quasiisom\'etries des espaces sym\'etriques de rang un.
\newblock{\em Ann. of Math.}, 129,  1, 1--60, (1989).

\end{thebibliography}
\end{document}